\documentclass[12pt]{amsart}
\usepackage{fullpage}
\usepackage[T1]{fontenc}
\usepackage[latin9]{inputenc}
\usepackage{amsmath}
\usepackage{amstext}
\usepackage{amsthm}
\usepackage{amssymb}
\usepackage{etoolbox}
\apptocmd{\sloppy}{\hbadness 10000\relax}{}{}
\usepackage[unicode=true]{hyperref}
\usepackage{xcolor}

\makeatletter

\newcommand{\ZZ}{\mathbb{Z}}
\newcommand{\NN}{\mathbb{N}}
\newcommand{\QQ}{\mathbb{Q}}

\newcommand{\Stab}{\operatorname{Stab}}
\newcommand{\Hom}{\operatorname{Hom}}

\numberwithin{equation}{subsection}
\numberwithin{figure}{subsection}
\theoremstyle{plain}
\newtheorem{thm}[equation]{\protect\theoremname}
\theoremstyle{remark}
\newtheorem{rem}[equation]{\protect\remarkname}
\theoremstyle{definition}
\newtheorem{defn}[equation]{\protect\definitionname}
\theoremstyle{definition}
\newtheorem{example}[equation]{\protect\examplename}
\theoremstyle{definition}
\newtheorem{problem}[equation]{\protect\problemname}
\theoremstyle{plain}
\newtheorem{prop}[equation]{\protect\propositionname}
\theoremstyle{plain}
\newtheorem{lem}[equation]{\protect\lemmaname}
\theoremstyle{plain}
\newtheorem{cor}[equation]{\protect\corollaryname}
\theoremstyle{plain}
\newtheorem{conj}[equation]{\protect\conjecturename}
\theoremstyle{plain}
\newtheorem{algo}[equation]{\protect\algorithmname}

\makeatother

\usepackage{listings}
\providecommand{\algorithmname}{Algorithm}
\providecommand{\conjecturename}{Conjecture}
\providecommand{\corollaryname}{Corollary}
\providecommand{\definitionname}{Definition}
\providecommand{\examplename}{Example}
\providecommand{\lemmaname}{Lemma}
\providecommand{\problemname}{Problem}
\providecommand{\propositionname}{Proposition}
\providecommand{\remarkname}{Remark}
\providecommand{\theoremname}{Theorem}

\begin{document}
\title{computing classical modular forms for arbitrary congruence subgroups}
\author{eran assaf}
\begin{abstract}
In this paper, we prove the existence of an efficient algorithm for the computation of systems of Hecke eigenvalues of modular forms of weight $k$ and level $\Gamma$, where $\Gamma \subseteq SL_{2}({\mathbb{Z}})$ is an arbitrary congruence subgroup. We also discuss some practical aspects and provide the necessary theoretical background.
\end{abstract}

\maketitle

\tableofcontents

\section{Introduction}

\subsection{Motivation}

The absolute Galois group of the rational numbers $G_{\mathbb{Q}} = \text{Gal}(\bar{\mathbb{Q}}/\mathbb{Q})$ is an object of central importance in number theory. It is studied through its action on geometric objects, among which elliptic curves and their torsion points have been playing a dominant role. Let $E$ be an elliptic curve defined over $\mathbb{Q}$, and let $p$ be a prime. Then $G_{\mathbb{Q}}$ acts on its $p$-torsion points $E[p]$. Therefore, we obtain a representation $\bar{\rho}_{E,p} : G_{\mathbb{Q}} \rightarrow GL_2(\mathbb{F}_p)$. 

In \cite{serre1972proprietes}, Serre proved that the index $[GL_2(\mathbb{F}_p) : \bar{\rho}_{E,p} (G_{\mathbb{Q}})]$ is bounded by a constant $C_E$ depending only on $E$.  Serre went on to ask the following question, which remains open to this day, after half a century.

\begin{conj} [Serre's uniformity problem over $\mathbb{Q}$, \cite{serre1972proprietes}]
Is there a constant $C > 0$ such that for any prime $p > C$ and any elliptic curve $E$ over $\mathbb{Q}$ without complex multiplication, the mod $p$ Galois representation $\bar{\rho}_{E,p}:G_{\mathbb{Q}} \rightarrow PGL_2(\mathbb{F}_p)$ is surjective?
\end{conj}

There has been much progress in working on the conjecture. The remaining and most difficult case is to exclude the possibility that $\bar{\rho}_{E,p}$ has image contained in the normalizer of a non-split Cartan subgroup of $GL_2(\mathbb{F}_p)$ for large $p$. 
These elliptic curves are classified by a modular curve $X_{ns}^{+}(p)$.
Equivalently, one has to show that for all large enough primes $p$, the only rational points of the modular curve $X_{ns}^{+}(p)$ are CM points.

Therefore, there is great interest in finding explicit equations for $X_{ns}^{+}(p)$ over $\mathbb{Q}$.

More generally, Mazur's Program B \cite{mazur1977rational} suggests, given an open subgroup $G \subseteq GL_2(\hat{\mathbb{Z}})$ to classify all elliptic curves $E$ such that the image of $\rho_{E}$ is contained in $G$. To this end, one would like to have explicit equations for the modular curve describing those, $X_G$, over $\mathbb{Q}$, in order to search for rational points.

The general method to find explicit equations for the modular curves $X_G$ is to compute $q$-expansions of a basis of cusp forms  in $S_k(\Gamma)$ where $\Gamma \subseteq SL_2(\mathbb{Z})$ is the preimage of $G$ in $SL_2(\mathbb{Z})$, and find polynomial relations between them. The complexity of this computation is dominated by the complexity of computing the Hecke operators $\{T_n\}$ on the space of modular forms $S_k(\Gamma)$. 

Thus, a main cornerstone of many efforts related to Serre's uniformity conjecture and Mazur's Program B is the computation of Hecke operators and $q$-expansions of modular forms with arbitrary level $\Gamma$. 

\subsection{Main Results}

Let $G \subseteq GL_2(\mathbb{Z}/N \mathbb{Z}) $ be a subgroup, and let $k \ge 2$ be an integer. Then $G$ pulls back to a congruence subgroup $\Gamma = \Gamma_{G} \subseteq SL_2(\mathbb{Z})$ of level $N$.  
This paper is concerned mainly with the computation of the space of modular forms of weight $k$ and level $\Gamma_{G}$, which we denote by $M_{k}(\Gamma_{G})$. 

Our main goal is to construct an explicit model for $M_{k}(\Gamma_{G})$ over $\mathbb{Q}$,
which includes a finite dimensional $\mathbb{Q}$-vector space equipped with a subspace
of cusp forms $S_{k}(\Gamma_{G})$, and a family of Hecke operators $\{T_{n}\}$
acting on it. 

After constructing such a model, we proceed to evaluate the systems of eigenvalues
of the eigenforms in this space, leading to computation of the zeta functions of the distinct components of the associated Jacobian.
We then proceed to compute $q$-expansions.
However, as we shall discuss later, these computations do not bridge the gap in the computational aspect of the efforts for computing $q$-expansions.

This is achieved most efficiently, under some mild assumptions, using Merel's results \cite{merel1994universal}. Let us presently describe these assumptions.

We shall make use of $\eta := \left(\begin{array}{cc}
-1 & 0\\
0 & 1
\end{array}\right)\in GL_{2}(\mathbb{Z})$, and denote by $\lambda_N: GL_2(\mathbb{Z}) \rightarrow GL_2(\mathbb{Z} / N \mathbb{Z})$ the natural reduction map.

\begin{defn} \label{def: real type}
We say that $G \subseteq GL_2(\mathbb{Z} / N \mathbb{Z})$ is of \emph{real type} if 
$ \lambda_{N}(\eta) G \lambda_{N}(\eta)^{-1}=G. $ 
We say that a group $\Gamma \subseteq SL_2(\mathbb{Z})$ is of \emph{real type} if 
$ \eta \Gamma \eta^{-1} = \Gamma. $
\end{defn}

\begin{rem}
The assumption that $G$ is of real type is essential for the Hecke operators to
commute with the star involution (see Lemma \ref{lem: Hecke commutes with star}). 
The slightly weaker assumption, that $\Gamma_G$ is of real type, is necessary for 
$Y_{\Gamma} := \Gamma \backslash \mathcal{H}$ to be defined over $\mathbb{Q}$.  
\end{rem}

In the description of complexities below, we denote by $I_G := [SL_2(\mathbb{Z}) : \Gamma_G]$ the index of $\Gamma_G$, and by $d := \dim S_2(\Gamma_G)$ the dimension of the space of cusp forms.
We shall describe the complexity of our algorithms in terms of field operations. 

\begin{rem}
Since the size of our matrix entries is bounded linearly by the level $N$, these are essentially logarithmic in $N$. 
However, when computing the eigenforms, some of the arithmetic is done over number fields, and the complexity depends on the algorithms used to implement the arithmetic over those fields. 
\end{rem}

In what follows, we assume the existence of a function $\operatorname{CosetIndex}(x,R,\Gamma)$
that for an element $x\in SL_2(\mathbb{Z})$, a group $\Gamma$ and a list
of coset representatives $R$ for $\Gamma\backslash SL_2(\mathbb{Z})$, returns
the unique index $i$, such that $x\in\Gamma\cdot R[i]$. 
We will denote by $C$ the complexity of $\operatorname{CosetIndex}$.

\begin{rem}
The function $\operatorname{CosetIndex}$ could be implemented in general, by noting
that $\Gamma_{G}\backslash SL_2(\mathbb{Z})\cong G_0\backslash SL_{2}(\mathbb{Z}/N\mathbb{Z})$, where 
$G_0 := G \cap SL_2(\mathbb{Z} / N \mathbb{Z})$,
and using the Todd-Coxeter algorithm  \cite{todd1936practical}, with complexity of $O(I_G^{1/2})$
field operations. However, for certain groups $\Gamma$,
there are more efficient implementations. For example, if $\Gamma=\Gamma_{0}(N)$,
this can be done in $O(1)$ field operations using the fact that 
$\Gamma\backslash SL_2(\mathbb{Z})\cong\mathbb{P}^{1}(\mathbb{Z}/N\mathbb{Z})$. 
Moreover, by allowing pre-computation of $O(N^3)$ field operations and memory of $O(N^3)$ field elements, we can use a table and access it using $O(\log N)$ field operations.
\end{rem}

For the convenience of the reader, a table is included in which the complexities of the different algorithms for computing Hecke operators are compared to one another, and specialized to $\Gamma_0(N)$, $\Gamma_{ns}(N)$. 

\begin{table}
\begin{center}
\begin{tabular}{| c | c | c | c |}
\hline
\bf{Algorithm} & $T_p$, $p \nmid N$  & $T_{\alpha}$, $(D(\alpha),N)=1$ & $T_{\alpha}$ \\
\hline
\bf{Complexity} & $O(C d p \log p)$ & $O(C d  I_{\alpha,\Gamma_G} \log(N D(\alpha)))$ & 
$O(d \left( C  I_{\alpha,\Gamma_G} \log(N  D(\alpha))+ I_G^2  \operatorname{In} \right))$ \\
\hline
$\Gamma_0(N)$ & $O(d p \log p)$ & $O(d p \log (Np))$ & $O(dp \log(N p) + N^2 \log^2 N)$ \\
\hline
$ \Gamma_{ns}(N)$ & $O(d p \log p)$ & $O(d p \log (Np))$ & $O(dp \log(N p) + N^2)$ \\
\hline
\end{tabular}
\caption{Summary of complexities for computing Hecke operators}
\end{center}
\end{table}

\begin{thm} [Corollary \ref{cor:Complexity Hecke Operator}]  \label{thm: Main result Merel}
There exists an algorithm that given a group of real type $G \subseteq GL_{2}(\mathbb{Z}/N\mathbb{Z})$, a prime $p \nmid N$ such that $p \bmod N \in \det(G)$ and an integer $k \ge 2$, computes the Hecke operator $T_p$ on the space of modular forms $S_k(\Gamma_G)$ using
$O(C  d  \cdot k \log k \cdot p \log p)$ field operations. 
\end{thm}

\begin{example} \
\begin{enumerate}
\item
If $\Gamma = \Gamma_0(N)$, then $C = O(1)$. Thus we obtain $O(d \cdot k \log k \cdot p \log p)$, the standard complexity.
\item
If $\Gamma = \Gamma_{ns}(N)$ is the non-split Cartan subgroup, then $C = O(1)$. Again we obtain $O(d \cdot k \log k \cdot p \log p)$. Note that here $d = \dim S_k(\Gamma)$, where standard methods cost $O(d_0 \cdot k \log k \cdot  p \log p)$,  with $d_0 = \dim S_k(\Gamma_0(N^2)) > d$.
\end{enumerate}
\end{example}

We also provide an algorithm for computing Hecke operators corresponding to arbitrary double cosets.
For a matrix $\alpha \in GL_2^{+}(\mathbb{Q})$, we let $D(\alpha) = \det(\alpha) / d_1(\alpha)^2 \in \mathbb{Z}$, with $d_1(\alpha)$ the  greatest common divisor of all the entries of $\alpha$. Denote also $ I_{\alpha,\Gamma} =  [\Gamma : \alpha^{-1} \Gamma \alpha \cap \Gamma]$.

\begin{thm}[Corollary \ref{cor:complexity-double-coset-operator}] \label{thm: Main result double coset coprime}
There exists an algorithm that given a group of real type $G \subseteq GL_{2}(\mathbb{Z}/N\mathbb{Z})$, an element $\alpha \in GL_2^{+}(\mathbb{Q})$ such that $\eta^{-1} \alpha \eta \in \Gamma_G \alpha \Gamma_G$ and $\gcd(D(\alpha),N) = 1$, and an integer $k \ge 2$, computes the Hecke operator $T_{\alpha}$, corresponding to the double coset $\Gamma_G \alpha \Gamma_G$, on the space of modular forms $S_k(\Gamma_G)$, using $O(C d  I_{\alpha,\Gamma_G}  k \log k \cdot   \log(N D(\alpha)))$ field operations.
\end{thm}

\begin{example}
If $\Gamma = \Gamma_0(N)$ and $\alpha = \left( \begin{array}{cc} 1 & 0 \\ 0 & p \end{array} \right)$ with $p \nmid N$, then 
$I_{\alpha, \Gamma} = p + 1$, $C = O(1)$ and $D(\alpha) = p$. Thus we obtain $O(d \cdot k \log k \cdot p \log(N p))$, slower than the algorithm described in Theorem~\ref{thm: Main result Merel}.
However, this algorithm works in greater generality.
\end{example}

\begin{rem}
If $G$ is not of real type, the above algorithm still computes the Hecke operator $T_{\alpha}$ on the space
$S_k(\Gamma_G) \oplus \overline{S_k(\Gamma_G)}$. This, in turn, gives rise to a system of eigenvalues, from which one can deduce the zeta function. 
An example is (group $8E^1$ in \cite{cummins2017congruence})
\[
G = \left \langle \left( \begin{array}{cc} 2 & 3 \\ 3 & 5 \end{array} \right),
 \left( \begin{array}{cc} 3 & 0 \\ 0 & 3 \end{array} \right),
  \left( \begin{array}{cc} 7 & 0 \\ 0 & 7 \end{array} \right),
   \left( \begin{array}{cc} 7 & 4 \\ 3 & 1 \end{array} \right),
    \left( \begin{array}{cc} 4 & 1 \\ 3 & 4 \end{array} \right)
 \right \rangle \subseteq GL_2(\ZZ / 8 \ZZ)
\]
which is not conjugate to any subgroup of real type. In this case, $\dim S_2(G) = 1$, and one may compute the Hecke operators using the code as follows.
\begin{verbatim}
> gens := [[ 7, 0, 0, 7 ],[ 2, 3, 3, 5 ],[ 0, 7, 7, 7 ],
	     [ 3, 0, 0, 3 ],[ 4, 7, 7, 3 ]];
> N := 8;
> H_N := sub<GL(2,Integers(N)) | gens>;
> H := PSL2Subgroup(H_N);
> M := ModularSymbols(H, 2, Rationals(), 0);
> S := CuspidalSubspace(M);
> HeckeOperator(S,97);
[18  0]
[ 0 18]
\end{verbatim}
\end{rem}

We describe a general algorithm to compute the Hecke operators, which works for the case $\gcd(D(\alpha),N) > 1$ as well. To describe its complexity, denote the complexity of membership testing in $G$ by $\operatorname{In}$.

\begin{thm} [Theorem  \ref{thm: Hecke operator double coset naive}]  \label{thm: Main result double coset}
There exists an algorithm that given a group of real type $G \subseteq GL_{2}(\mathbb{Z}/N\mathbb{Z})$, an element $\alpha \in GL_2^{+}(\mathbb{Q})$ such that $\eta^{-1} \alpha \eta \in \Gamma_G \alpha \Gamma_G$ and an integer $k \ge 2$, computes the Hecke operator $T_{\alpha}$ corresponding to the double coset $\Gamma_G \alpha \Gamma_G$, on the space of modular forms $S_k(\Gamma_G)$, using
$$O(d \left( C  I_{\alpha,\Gamma_G} k \log k \cdot  \log(N  D(\alpha))+ I_G^2  \operatorname{In} \right) )$$
field operations. 
\end{thm}

\begin{example}
If $\Gamma = \Gamma_0(N)$ and $\alpha = \left( \begin{array}{cc} 1 & 0 \\ 0 & p \end{array} \right)$ with $p \mid N$, then $I_G = O(N \log N)$, yielding $O(dp \log(N p) + N^2 \log^2 N)$. One may notice that the second term we are picking up, due to the generality, is quite large.
\end{example}

These algorithms allow us to compute zeta functions efficiently. For a complete zeta function, we need also the values of the Hecke operators at primes dividing the level. Therefore, we need to assume an additional hypothesis.

\begin{defn}
Let $G \subseteq GL_2(\mathbb{Z} / N \mathbb{Z})$, and let $p \mid N$ be a prime. Let $\{ \alpha_{p,i} \}_{i=1}^{r_p}$ be elements in $GL_2^{+}(\mathbb{Q})$ such that $T_p$ is a linear combination of the $T_{\alpha_{p,i}}$. We say that the Hecke operator $T_p$ on $S_k(\Gamma_G)$ is \emph{effectively computable} if there exists an algorithm to compute a set of such elements $\{ \alpha_{p,i} \}$ using $O(C  d  (k \log k \cdot p \log p + k I_G \log (k I_G)) + d^3)$ field operations.
\end{defn}

\begin{example} \ 
\begin{enumerate}
\item
When $\Gamma = \Gamma_0(N)$, $T_p$ is effectively computable for all $p \mid N$, as $\alpha_p = 
\left(\begin{array}{cc}
1 & 0\\
0 & p
\end{array}\right)
$
can be computed in $O(1)$ field operations.
\item
When the order generated by $\Gamma$ contains no elements of determinant $p$, $T_p = 0$ is effectively computable.
\end{enumerate}
\end{example}

We further denote by $\operatorname{In}$ the complexity, in field operations, of membership test in $G$.

\begin{cor} [Corollary \ref{cor: q expansion basis}] \label{cor: Main result q expansion}
There exists an algorithm that given a group of real type $G \subseteq GL_2(\mathbb{Z} / N \mathbb{Z})$ with surjective determinant such that for all $p \mid N$, $T_p$ is effectively computable,
and a positive integer $L$, returns the zeta functions associated to each of the factors of $Jac(X(\Gamma_G))$ up to precision $L^{-s}$ using 
$$O(d ( C (L \log L + N \log N) + N I_G^2 \operatorname{In} + I_G \log (I_G)) + d^3)$$
field operations. 
\end{cor}

This further allows one to compute $q$-expansions as follows. 
Consider the inclusion $\iota : S_k(\Gamma_G) \subseteq S_k(\Gamma(N))$, where $N$ is the level of $\Gamma_G$. This inclusion is not Hecke equivariant in general. Therefore, our eigenvalues are not the $q$-expansions of forms in $S_k(\Gamma_G)$. 

However, we have an action of $GL_2(\mathbb{Z} / N \mathbb{Z})$ on $S_k(\Gamma(N))$. Write  $S_k(\Gamma(N)) = \bigoplus_{i \in I} V_i$, with each $V_i$ an irreducible $GL_2(\mathbb{Z} / N \mathbb{Z})$-representation.
Since $V := \iota(S_k(\Gamma_G)) = S_k(\Gamma(N))^G$, we have $V = \bigoplus_{i \in I} V \cap V_i$.
But in $\Gamma(N)$ the $q$-expansions of newforms do correspond to the zeta functions of the appropriate factors. 
Therefore, the $q$-expansions we have found for a basis of $V \cap V_i$ are the $q$-expansions of newforms in $V_i$, denote them by $g_{ij}$. 
By \cite{atkin1978twists}, for each $i$, the space $V_i$ is generated by twists of the $\{g_{ij}\}$.
Thus, we are now led to finding the linear combinations that land in $V \cap V_i$.
This can be done either by computing the $GL(2, \mathbb{Z} / N \mathbb{Z})$-action on cusp forms, as described in \cite{zywina2020computing}, or by using modular symbols for $\Gamma(N)$, and solving linear equations. This can be done in several ways, and is currently in the process of being implemented by Josha Box (see \cite{box2020qexpansions}) and by the author. 

Thus, obtaining the $q$-expansions remains highly inefficient for two reasons.
\begin{enumerate}
\item
In order to obtain the $q$-expansions one has to work in the larger space $S_k(\Gamma(N))$.
\item
The linear algebra required to obtain the $q$-expansions has to be done over large fields, as the twisting operators introduce the Gauss sums of the characters, which are elements in the large cyclotomic field $\mathbb{Q}(\zeta_N, \zeta_{\phi(N)})$. 
\end{enumerate}

%

The method illustrated in this paper has been implemented in the MAGMA
computational algebra system, by building on the existing package
of modular symbols, implemented in MAGMA \cite{bosma1997magma} by
William Stein, with contributions by Steve Donnelly and Mark Watkins. 
The package is publicly available in the github repository \cite{assaf2019github}. 

\subsection{Recovering known results}

As an application of our algorithm we recover the equation for the canonical embedding of $X_{S_4}(13)$ in $\mathbb{P}_{\mathbb{Q}}^2$, as was done in \cite{banwait2014tetrahedral}. Effectively, this repeats the work done in the original paper, only in an automated fashion. The main difference is that the computation of the eigenform in step 3 of Section 4 is done directly in the smaller space $S_2(\Gamma_{S_4})$, instead of inside $S_2(\Gamma_0(169))$.

\begin{cor} [{\cite[Theorem 1.8]{banwait2014tetrahedral}}] 
The modular curve $X_{S_4}(13)$  is a genus $3$ curve whose canonical embedding in $\mathbb{P}_{\mathbb{Q}}^2$ has the model
\begin{align*}
4x^3y - 3x^2y^2 +3xy^3 -x^3z + 16x^2yz - 11xy^2z\\
+5y^3z+3x^2z^2+9xyz^2+y^2z^2+xz^3+2yz^3 = 0.
\end{align*}
\end{cor}

Another application recovers the main result of \cite{baran2014exceptional}, namely explicit equations for the canonical embeddings of $X_{ns}^{+}(13)$ and $X_s(13)$ in $\mathbb{P}_{\mathbb{Q}}^2$. Again, this basically repeats the work done in the original paper, only obtaining the result of Section 4 in a matter of seconds.

\begin{cor}[\cite{baran2014exceptional}]
The modular curves $X_{ns}^+(13)$ and $X_s(13)$ are defined by the equation
\begin{align*}
(-y-z)x^3+(2y^2+zy)x^2+(-y^3+zy^2-2z^2y+z^3)x+(2z^2y^2-3z^3y)=0.
\end{align*}
\end{cor}

Similarly, we are also able to reproduce the results of \cite{mercuri2020modular} for $X_{ns}^+(17)$, $X_{ns}^+(19)$ and $X_{ns}^+(23)$.

\subsection{New results}

An interesting problem that requires the computation of such Hecke operators is the question of decomposition of Jacobians of modular curves. The factors in these decompositions should have interesting arithmetic. We could perform the following using our code in 31 minutes. Note that using Chen's isogeny, one would need to compute and decompose the space of modular symbols $S_2(\Gamma_0(97^2))$.

\begin{cor}
The Jacobian of the modular curve $X_{ns}^{+}(97)$ decomposes over $\mathbb{Q}$ as the direct sum of 13 Hecke irreducible subspaces, of dimensions $$3,4,4,6,7,7,12,14,24,24,24,56,168.$$ In particular, it has no elliptic curve factor.
\end{cor}

\begin{rem}
Any method of computing modular forms of weight $k=1$ relies on computation of modular forms of higher integral weights (e.g. \cite{buzzard2014computing}, \cite{schaeffer2015hecke}), so that our results apply equally well. 
\end{rem}

\subsection{Related Literature}

The theory of modular forms for $GL_{2}$ has been established for
quite some time. However, the theory usually restricts to the Iwahori
level subgroups, $\Gamma_{0}(N)$ and $\Gamma_{1}(N)$. There are
a few exceptions: Shimura, in \cite{shimura1971introduction}, explores
a mild generalization, and much work has been devoted to some special
cases, such as those groups induced by maximal subgroups of $GL_{2}(\mathbb{Z}/N\mathbb{Z})$.
In this vein, computation has also been restricted to the subgroups
$\Gamma_{0}(N)$ and $\Gamma_{1}(N)$, as in \cite{stein2007modular} and in \cite{couveignes2011computational}.

Nevertheless, many new applications arise (\cite{banwait2014tetrahedral},
\cite{baran2010normalizers}, \cite{gonzalez2010non}, \cite{sutherland2017modular})
in which description of the space of modular forms $M_{k}(\Gamma)$
is needed, when $\Gamma\subseteq SL_{2}(\mathbb{Z})$ is an arbitrary
congruence subgroup. 
 
\subsection{Organization}

The paper is structured as follows.

In Section \ref{sec: setup}, we introduce the definitions and the statement of the
problem. 

In Section \ref{sec: explicit cuspforms}, we explain the construction of an explicit model for
$M_{k}(\Gamma)$ using modular symbols, how to construct a boundary
map $\partial$, and through it the subspace of cusp forms $S_{k}(\Gamma)$.
This description of $S_{k}(\Gamma)$ was already known to Merel, see
\cite{merel1994universal}, but we present a general algorithm for
the computation of $\partial$. 
So far, the literature has only discussed the computation of $\partial$
for Iwahori level subgroup. (see \cite{cremona1997algorithms}, \cite{stein2007modular},
\cite{wiese2005modular}).

Section \ref{sec: Hecke operators} contains the core of this work, the computation of the Hecke operators.
In subsection \ref{subsec: Hecke operators definitions}, we begin by defining the Hecke operators corresponding to double cosets.
Then, in subsection \ref{subsec: naive Hecke operator}, we present a general algorithm for computing them, thus proving theorem \ref{thm: Main result double coset}. We then present a variant of this algorithm which is more efficient, but applicable only to Hecke operators away from the level, thus yielding \ref{thm: Main result double coset coprime}.
In order to define the Hecke operators $T_{n}$ for integers
$n$ that are coprime to the level of $\Gamma$, i.e. $(n,N)=1$, in subsection \ref{subsec: Hecke operator at p},
we need to view the modular curve adelically. We start from the adelic Hecke operators and show that they correspond to certain double coset operators that were already defined.
In subsection \ref{subsec: Hecke efficient Merel}, we use the techniques of Merel from 
\cite{merel1994universal} to prove theorem \ref{thm: Main result Merel}.

Section \ref{sec: degeneracy maps} describes the degeneracy maps, and the old and new subspaces, which will be useful when decomposing our space of cusp forms.
Section \ref{sec: zeta functions} contains miscellaneous algorithms needed to compute the zeta functions associated to the eigenforms, given the Hecke operators.
In section \ref{sec: applications} we describe several applications of our result to contemporary research.

\subsection{Acknowledgements}
It is a pleasure to thank John Voight for many insightful discussions on this project.
The author would like to thank deeply Jeremy Rouse for his very useful comments on an earlier draft of this manuscript and for pointing to several interesting references in the literature.
The author is also grateful to David Zywina and Elisa Lorenzo Garci{\'a} for their remarks on the code package, and to the anonymous referees for their comments on the manuscript and the code.
This research was conducted as part of the \emph{Simons Collaboration on
Arithmetic Geometry, Number Theory and Computation,} with the support
of Simons Collaboration Grant (550029, to John Voight). 

\section{Setup and Notation} \label{sec: setup}

In this section, we present the basic setup and introduce some notations that will be used throughout the paper.

\subsection{Congruence Subgroups}

Let $N$ be a positive integer, $G\subseteq GL_{2}(\mathbb{Z}/N\mathbb{Z})$
a subgroup, and $G_{0}=G\cap SL_{2}(\mathbb{Z}/N\mathbb{Z})$. 
Let $\lambda_{N}:M_{2}(\mathbb{Z})\rightarrow M_{2}(\mathbb{Z}/N\mathbb{Z})$ be the natural reduction map, and let 
$\lambda_{N,0} : SL_{2}(\mathbb{Z}) \rightarrow SL_{2}(\mathbb{Z} / N \mathbb{Z})$ be its restriction to $SL_2(\mathbb{Z})$.

Then $\Gamma_{G}:=\lambda_{N,0}^{-1}(G_{0})\subseteq SL_{2}(\mathbb{Z})$
contains $\Gamma(N):=\ker(\lambda_{N,0})$, hence it is a congruence
subgroup. Moreover, every congruence subgroup arises this way. In general, we will denote by $\Gamma$ a congruence subgroup of $SL_2(\mathbb{Z})$.

We denote by $\mathcal{H}$ the complex upper half plane $\mathcal{H} := \{ z \in \mathbb{C} \mid \Im(z) > 0 \}$. 
Then $\mathcal{H}$ admits a natural action of $GL_2^{+}(\mathbb{R})$ via M{\"o}bius transformations and we let $Y_{\Gamma} := \Gamma \backslash \mathcal{H}$ be the affine modular curve of level $\Gamma$. 

\begin{defn}
Denote $\mathcal{H}^{*} = \mathcal{H} \cup \mathbb{P}^1(\mathbb{Q})$. The group $GL_2^{+}(\mathbb{Q})$ also acts on $\mathbb{P}^1(\mathbb{Q})$ via M{\"o}bius transformations, and we let $X_{\Gamma} := \Gamma \backslash \mathcal{H}^{*}$ be the \emph{modular curve of level $\Gamma$}. The points of $\Gamma \backslash \mathbb{P}^1(\mathbb{Q}) = X_{\Gamma} - Y_{\Gamma}$ are called the \emph{cusps} of $X_{\Gamma}$. 
\end{defn}

\begin{defn}
We say that $\Gamma_{G}$ is the congruence subgroup \emph{induced}
by $G$. 

Here are several useful examples.
\end{defn}

\begin{example} 
\label{exa: Gamma0}
The following are standard examples of congruence subgroups, which will be used throughout.
\begin{enumerate}
\item
Let $G=\left\{ \left(\begin{array}{cc}
a & 0\\
0 & 1
\end{array}\right)\mid a\in(\mathbb{Z}/N\mathbb{Z})^{\times}\right\} $. 
Then $\Gamma_{G}=\Gamma(N)$. 
\item
Let $G=\left\{ \left(\begin{array}{cc}
a & b\\
0 & 1
\end{array}\right)\mid a\in(\mathbb{Z}/N\mathbb{Z})^{\times},b\in\mathbb{Z}/N\mathbb{Z}\right\} $. 
Then $\Gamma_{G}=\Gamma_{1}(N)$. 
\item
Let $G=\left\{ \left(\begin{array}{cc}
a & b\\
0 & d
\end{array}\right)\mid a,d \in(\mathbb{Z}/N\mathbb{Z})^{\times},b\in\mathbb{Z}/N\mathbb{Z}\right\} $. 
Then $\Gamma_{G}=\Gamma_{0}(N)$.
\item
Let $\mathfrak{h}$ be a subgroup of $(\mathbb{Z} / N \mathbb{Z})^{\times}$, and let $t$ be a divisor of $N$. Let \\
$G = \left\{ \left(\begin{array}{cc}
a & b\\
0 & d
\end{array}\right)\mid a \in \mathfrak{h},d \in(\mathbb{Z}/N\mathbb{Z})^{\times},b\in t \mathbb{Z}/N\mathbb{Z}\right\} $. 
Then $\Gamma_{G}=\Gamma(\mathfrak{h}, t)$.
\end{enumerate}
\end{example}

Before presenting a non-trivial interesting example, we define the notion of a non-split Cartan subgroup.
\begin{defn}
Let $A$ be a finite free commutative $\mathbb{Z}/N\mathbb{Z}$-algebra of rank $2$ with unit discriminant, such that for every prime $p$ dividing $N$, the $\mathbb{F}_p$-algebra $A/pA$ is isomorphic to $\mathbb{F}_{p^2}$. Then the unit group $A^{\times}$ acts on $A$ by multiplication, and a choice of basis for $A$ induces an embedding $A^{\times} \hookrightarrow GL_2(\mathbb{Z}/N\mathbb{Z})$. The image of $A^{\times}$ is called a \emph{non-split Cartan subgroup} of $GL_{2}(\mathbb{Z}/N\mathbb{Z})$. 
\end{defn}

\begin{example}
Let $G$ be a non-split Cartan subgroup of $GL_{2}(\mathbb{Z}/N\mathbb{Z})$. Then $\Gamma_{G}=\Gamma_{ns}(N)$. Similarly, if $\mathfrak{N}(G)$ is its normalizer in $GL_2(\mathbb{Z}/N\mathbb{Z})$ then $\Gamma_{\mathfrak{N}(G)}=\Gamma_{ns}^{+}(N)$. 
\end{example}

\subsection{Modular Forms}
\begin{defn}
Let $k\ge2$ be an integer. Let $g=\left(\begin{array}{cc}
a & b\\
c & d
\end{array}\right)\in GL_{2}^{+}(\mathbb{Q})$. Let $f:\mathcal{H}\rightarrow\mathbb{C}$, where $\mathcal{H}:=\{z\in\mathbb{C}\mid\Im(z)>0\}$.
We define for any $z\in\mathcal{H}$
\begin{equation}
f\vert_{g}(z)=(\det g)^{k-1}\cdot(cz+d)^{-k}f(gz)\label{eq:action of GL2Q on modular forms}
\end{equation}

Denote by $M_{k}(\Gamma)$ the space of holomorphic modular forms
of weight $k$ with respect to $\Gamma$. (i.e. holomorphic functions
$f:\mathcal{H}\rightarrow\mathbb{C}$ such that $f\vert_{\gamma}=f$
for all $\gamma\in\Gamma$ and $f$ is holomorphic at the cusps of
$\Gamma\backslash\mathcal{H}$). Denote by $S_{k}(\Gamma)$ the subspace
of cusp forms. (those vanishing at the cusps).
\end{defn}

Since $\Gamma(N)\subseteq\Gamma$, we know that $\left(\begin{array}{cc}
1 & N\\
0 & 1
\end{array}\right)\in\Gamma$, hence $f(z+N)=f(z)$ for all $f\in M_{k}(\Gamma)$. It follows that
$f$ admits a Fourier expansion $f(z)=\sum_{n=0}^{\infty}a_{n}q^{n}$,
with $a_{n}\in\mathbb{C}$ and $q:=e^{\frac{2\pi iz}{N}}$. 

This is now enough to state a computational problem:
\begin{problem}
\label{prob:find q expansions}Given a group $G\subseteq GL_{2}(\mathbb{Z}/N\mathbb{Z})$, an integer $k \ge 2$ and a positive integer $L$, find the $q$-expansion of a basis for $S_{k}(\Gamma_{G})$ up to precision $q^L$. 
\end{problem}

This problem has been answered extensively in the literature for $\Gamma_{G}=\Gamma_{1}(N)$
and $\Gamma_{G}=\Gamma_{0}(N)$ (see \cite{stein2002introduction},
\cite{stein2007modular}, \cite{kilford2010modular}, \cite{couveignes2011computational}).
In these cases, the solution to Problem \ref{prob:find q expansions} is through
the following steps:

\begin{enumerate}
\item Construct an explicit vector space representing $S_{k}(\Gamma_{G})$.
\item Compute the action of Hecke operators on this space.
\item Decompose the action into irreducible subspaces.
\item Write down the $q$-expansion from the systems of eigenvalues in the
above decomposition. 
\end{enumerate}

Among these steps, step (3) and (4) follow from (2) by linear algebra techniques. Moreover, the Hecke algebra of $S_k(\Gamma_G)$ is generated by the double coset Hecke operators $\{T_{\alpha}\}$ (see Corollary \ref{cor:Hecke operator double coset}).
Therefore, we concentrate our efforts on the following problem.

\begin{problem}
\label{prob: find Hecke operators}
Given a group $G\subseteq GL_{2}(\mathbb{Z}/N\mathbb{Z})$, an integer $k \ge 2$ and an element $\alpha \in GL_2^{+}(\mathbb{Q})$ compute the matrix of the Hecke operator $T_{\alpha}$ with respect to a basis of $S_k(\Gamma_G)$. 
\end{problem}

We address this problem in full in section \ref{sec: Hecke operators}, and present several solutions - one general solution with no additional assumptions, and under some mild assumptions, we have a more efficient solution.

Unfortunately, in the general case step (4) does not follow from step (3) in the same way. The reason is that the Hecke operators acting on $S_k(\Gamma_G)$ have no longer a nice description in terms of the Fourier coefficients. For an example of such a description, see \cite{zywina2020computing}. 

We can and will show that the Hecke operators we construct correspond to the Frobenius action on the curves, hence the system of eigenvalues of any eigenform $f$ does yield the zeta function of $A_f$, the corresponding factor of the Jacobian.

\section{Explicit Computation of \texorpdfstring{$S_{k}(\Gamma)$}{} }
\label{sec: explicit cuspforms}
First, we have to construct a model for the vector space $S_{k}(\Gamma)$. 
We will do so through the means of modular symbols. For the results of this section, it suffices to assume that $\Gamma$ is a finite index subgroup of $SL_2(\mathbb{Z})$. Therefore, in this section, we will suppress the notation of the group $G$ in $\Gamma_G$. 

\subsection{Modular Symbols}

Let $\mathbb{M}_{2}$ be the free abelian group generated
by the expressions $\{a,b\}$ , for $\left(a,b\right)\in\mathbb{P}^{1}(\mathbb{Q}) \times \mathbb{P}^{1}(\mathbb{Q}) $,
subject to the relations 
\[
\{a,b\}+\{b,c\}+\{c,a\}=0\quad\forall a,b,c\in\mathbb{P}^{1}(\mathbb{Q}).
\]
and modulo any torsion. Let $\mathbb{M}_{k}:=\operatorname{Sym}^{k-2}\mathbb{Z}^{2}\otimes\mathbb{M}_{2}$,
where $\operatorname{Sym}^{k-2}$ is the representation of $GL_{2}$ of highest
weight $k-2$. 

There is a natural left action of $GL_{2}(\mathbb{Q})$ on $\mathbb{M}_{k}$
by 
\[
g(v\otimes\{a,b\})=gv\otimes\{ga,gb\}
\]

\begin{defn}
Let $\mathbb{M}_{k}(\Gamma)=(\mathbb{M}_{k})_{\Gamma}:=\mathbb{M}_{k}/\left\langle \gamma x-x\right\rangle $
be the space of $\Gamma$-coinvariants, the space of \emph{modular
symbols} \emph{of weight $k$ for }$\Gamma$ (over $\mathbb{Z}$). The space of \emph{modular
symbols of weight $k$ for }$\Gamma$\emph{ over a ring $R$ is 
\[
\mathbb{M}_{k}(\Gamma;R):=\mathbb{M}_{k}(\Gamma)\otimes_{\mathbb{Z}}R.
\]
}
The reason we are interested in the space of modular symbols $\mathbb{M}_{k}(\Gamma)$
is the following theorem proved by Manin:
\end{defn}

\begin{thm}
(\cite[Theorem 1.9]{manin1972parabolic}) The natural homomorphism
$$
\varphi:\mathbb{M}_{2}(\Gamma)\rightarrow H_{1}(X_{\Gamma},\text{cusps},\mathbb{Z}),
$$
sending the symbol $\{a,b\}$ to the geodesic path in $X_\Gamma$ between $\Gamma a$ and $\Gamma b$, is an isomorphism.
\end{thm}

This will give us the connection to modular forms. We now turn to
the reason that the modular symbols are useful for computations. 

\subsection{Manin Symbols}
\begin{defn}
Let $g\in SL_{2}(\mathbb{Z})$. The \emph{Manin symbol $[v,g]\in\mathbb{M}_{k}(\Gamma)$
}is defined as
\[
[v,g]:=g(v\otimes\{0,\infty\})
\]

Let 
\[
\sigma=\left(\begin{array}{cc}
0 & -1\\
1 & 0
\end{array}\right),\quad\tau=\left(\begin{array}{cc}
0 & -1\\
1 & -1
\end{array}\right),\quad J=\left(\begin{array}{cc}
-1 & 0\\
0 & -1
\end{array}\right)
\]

We have a right action of $SL_{2}(\mathbb{Z})$ on Manin symbols as
follows
\[
[v,g]h=[h^{-1}v,gh]
\]

The following theorem gives an explicit recipe for constructing a
concrete realization of $\mathbb{M}_{k}(\Gamma)$ as a free module
of finite rank.
\end{defn}

\begin{thm}
\label{thm:Explicit M_k}(\cite[Propositions 1,3]{merel1994universal}) The Manin symbol $[v,g]$ depends only on the class $\Gamma g\in\Gamma\backslash SL_{2}(\mathbb{Z})$
and on $v$. The Manin symbols $\{[v,g]\}_{g\in SL_{2}(\mathbb{Z}),v\in\text{Sym}^{k-2}\mathbb{Z}^{2}}$
generate $\mathbb{M}_{k}(\Gamma)$. Furthermore, if $x$ is a Manin
symbol, then 
\[
x+x\sigma=0
\]
\[
x+x\tau+x\tau^{2}=0
\]
\[
x-xJ=0
\]

Moreover, these are all the relations between Manin symbols, i.e.
if $B$ is a basis for $\text{Sym}^{k-2}\mathbb{Z}^{2}$ then
\[
\mathbb{M}_{k}(\Gamma)\cong\left(\bigoplus_{b\in B,g\in\Gamma\backslash SL_{2}(\mathbb{Z})}\mathbb{Z}\cdot[b,g]\right)/R
\]
where $R$ are the above relations, and torsion. 
\end{thm}

This allows one to compute a model for $\mathbb{M}_k(\Gamma)$ efficiently.

\begin{cor} \label{cor: construct modular symbols}
There exists an algorithm that given a finite index subgroup $\Gamma \subseteq SL_2(\mathbb{Z})$ and an integer $k \ge 2$, computes a basis of $\mathbb{M}_{k}(\Gamma)$ in $O([SL_2(\mathbb{Z}):\Gamma])$ basic $\operatorname{CosetIndex}$
operations. 
\end{cor}

\begin{proof}
This was implemented by Stein for $\Gamma=\Gamma_{0}(N)$ in \cite{stein2007modular}, and the general case is similar. 
\end{proof}

\begin{rem}
The implementation in \cite{stein2007modular} makes use of the fact that 
$\Gamma_0(N) \backslash SL_2(\ZZ) \cong \mathbb{P}^1(\ZZ / N \ZZ)$. This allows him to perform 
$\operatorname{CosetIndex}$ operations in $O(1)$ field operations. 
This is the only difference from the general case.
\end{rem}

\begin{defn}
The module $\text{Sym}^{k-2}\mathbb{Z}^{2}$ has a basis consisting of homogeneous
elements, namely $\{x^{w}y^{k-2-w}\}_{w=0}^{k-2}$. It follows that
$\mathbb{M}_{k}(\Gamma)$ has a basis consisting of elements of the
form $[x^{w}y^{k-2-w},g]$. From now on, we assume all given bases
to be of this form, and say that the \emph{weight} of a basis element
$[x^{w}y^{k-2-w},g]$ is $w$. 
\end{defn}

\subsection{Modular Symbols with Character} \label{subsection: modular symbols with characters}

Suppose $\Gamma \subseteq \Gamma'\subseteq SL_{2}(\mathbb{Z})$ are finite index subgroups
such that $\Gamma$ is a normal subgroup of $\Gamma'$ and let $Q:=\Gamma'/\Gamma$. Then $Q$ acts on $\mathbb{M}_{k}(\Gamma)$ via
\begin{equation} \label{eqn: Q action}
(\gamma'\Gamma)\cdot[v,g]=[v,\gamma'g]
\end{equation}
Let $\varepsilon:Q\rightarrow\mathbb{Q}(\zeta)^{\times}$ be a character,
where $\zeta=\zeta_{n}$ is an $n$-th root of unity, and $n$ is
the order of $\varepsilon$. Abusing notation, we will denote the
induced character on $\Gamma'$ by $\varepsilon$
as well. 

Let $\mathbb{M}_{k}(\Gamma,\varepsilon)$ be the quotient of $\mathbb{M}_{k}(\Gamma;\mathbb{Z}[\zeta])$
by the relations 
\[
\gamma'\cdot [v,g]-\varepsilon(\gamma')\cdot [v,g]=0
\]
for all $x\in\mathbb{M}_{k}(\Gamma;\mathbb{Z}[\zeta])$, $\gamma'\in Q$,
and by any $\mathbb{Z}$-torsion. 

Note that $\mathbb{M}_{k}(\Gamma,\varepsilon)$ has a basis consisting
solely of elements of the form $[v,g]$ where $g\in\Gamma'\backslash SL_{2}(\mathbb{Z})$. 
Then if $Q$ is abelian, this action factorizes to the isotypic components:
\begin{equation} \label{eqn: isotypic decomposition}
\mathbb{M}_{k}(\Gamma)\cong\bigoplus_{\varepsilon\in\hat{Q}}\mathbb{M}_{k}(\Gamma,\varepsilon)
\end{equation}

This decomposition allows for faster computation by implementing $\mathbb{M}_{k}(\Gamma)$
as the direct sum of these smaller spaces. 

\subsection{Cuspidal Modular Symbols}

Let $\mathbb{B}_{2}$ be the abelian group generated by the elements
$\{a\}$, $a \in\mathbb{P}^{1}(\mathbb{Q})$. Let $\mathbb{B}_{k}=\text{Sym}^{k-2}\mathbb{Z}^{2}\otimes\mathbb{B}_{2}$.
We have a left action of $GL_{2}(\mathbb{Q})$ on $\mathbb{B}_{k}$
via 
\[
g(v\otimes\{a\})=gv\otimes\{g a\}
\]

Let $\mathbb{B}_{k}(\Gamma):=(\mathbb{B}_{k})_{\Gamma}$. 
This can be viewed naturally as a space of "boundary Manin symbols'', using the following boundary map.
\begin{defn}
The \emph{boundary map $\partial:\mathbb{M}_{k}(\Gamma)\rightarrow\mathbb{B}_{k}(\Gamma)$
}is the natural map extending linearly the map 
\[
\partial(v\otimes\{a,b\})=v\otimes a-v\otimes b.
\]

Again we have a theorem that allows us to find a nice description of this module as a free
module of finite rank.
\begin{thm}
(\cite[Propositions 4,5]{merel1994universal}) Let $\mathcal{R}$
be the equivalence relation on $\Gamma\backslash\mathbb{Q}^{2}$ given
by 
\[
[\Gamma(\lambda u,\lambda v)]\sim sign(\lambda)^{k}[\Gamma(u,v)]
\]
for any $\lambda\in\mathbb{Q}^{\times}$. Let $\mu:\mathbb{B}_{k}(\Gamma)\rightarrow\mathbb{Q}[(\Gamma\backslash\mathbb{Q}^{2})/\mathcal{R}]$
be the natural map given by 
\[
\mu\left(P\otimes\left\{ \frac{u}{v}\right\} \right)=P(u,v)\cdot[\Gamma(u,v)]
\]
where $P\in\text{Sym}^{k-2}\mathbb{Z}^{2}$ is considered as a homogeneous
polynomial in two variables of degree $k-2$ and $u,v$ are coprime.

Then $\mu$ is well defined and injective. Moreover, 
\[
\mu\circ\partial([P,g])=P(1,0)[\Gamma g(1,0)]-P(0,1)[\Gamma g(0,1)].
\]
\end{thm}

The space $\mathbb{S}_{k}(\Gamma)=\ker\partial$ is called the space
of \emph{cuspidal modular symbols}. As in subsection \ref{subsection: modular symbols with characters}, we can define $\mathbb{B}_{k}(\Gamma, \varepsilon)$, a boundary map $\partial : \mathbb{M}_k(\Gamma,\varepsilon) \rightarrow \mathbb{B}_k(\Gamma, \varepsilon)$ and its kernel will be denoted by  $\mathbb{S}_k(\Gamma, \varepsilon)$. 

In order to have a concrete realization of $\mathbb{S}_{k}(\Gamma, \varepsilon)$,
it remains to efficiently compute the boundary map. 
\end{defn}

\subsection{Efficient computation of the boundary map}

In this subsection, we show how to compute the boundary map efficiently. 
For finite index subgroups $\Gamma \subseteq \Gamma' \subseteq SL_2(\mathbb{Z})$ such that $\Gamma$ is a normal subgroup of $\Gamma'$, and a character $\varepsilon : \Gamma'/ \Gamma \rightarrow \mathbb{Q}(\zeta)^{\times}$, we denote $c_{\varepsilon}(\Gamma) = \dim \mathbb{B}_2(\Gamma,\varepsilon)$.
We will proceed to prove the following result.

\begin{thm} \label{thm: cuspidal modular symbols}
There exists an algorithm that given groups $\Gamma \subseteq \Gamma' \subseteq SL_2(\mathbb{Z})$ such that $\Gamma$ is a normal subgroup of $\Gamma'$, a character $\varepsilon : \Gamma'/ \Gamma \rightarrow \mathbb{Q}(\zeta)^{\times}$, and an integer $k \ge 2$, computes a basis for $\mathbb{S}_k(\Gamma, \varepsilon)$ in 
$O([SL_2(\mathbb{Z}) : \Gamma'] \cdot c_{\varepsilon}(\Gamma) + [SL_2(\mathbb{Z}):\Gamma])$ basic $\operatorname{CosetIndex}$ operations. 
\end{thm}

The proof of Theorem \ref{thm: cuspidal modular symbols} will occupy this entire subsection. We will obtain the basis by computing the kernel of the boundary map. Thus, given $\mathbb{M}_{k}(\Gamma,\varepsilon)$, we would like to compute
the boundary map $\mu\circ\partial:\mathbb{M}_{k}(\Gamma,\varepsilon)\rightarrow\mathbb{B}_{k}(\Gamma,\varepsilon)$. 

The idea follows Cremona and Stein (\cite[Prop. 2.2.3, Lemma 3.2, Prop. 8.13]{cremona1997algorithms}
, \cite[Alg. 8.12, Prop. 8.13, Alg. 8.14]{stein2007modular} ) in that we do not have to compute the
space of cusps for $\Gamma$ a priori.
Instead, for every cusp in $\mathbb{P}^{1}(\mathbb{Q})$ that we encounter,
we check whether it is equivalent to the ones previously encountered. 

Following that route, we should be able to check whether a cusp is equivalent to $0$ (note that this is not the cusp $0$, but rather the neutral element of the module $\mathbb{B}_k(\Gamma, \varepsilon)$) or to a previously encountered cusp.
We will begin by describing the algorithms for testing cusp equivalence and cusp vanishing, and conclude by presenting the algorithm for computing the boundary map.

These algorithms are implemented in \cite{stein2007modular} and in \cite{cremona1997algorithms}
for $\Gamma=\Gamma_{0}(N),\Gamma_{1}(N)$. 

In order to treat the general case, we first prove a criterion for
the equivalence.
\begin{prop}\label{prop: cusp equivalence}
Let $a,b\in\mathbb{P}^{1}(\mathbb{Q})$. Let $\Gamma\subseteq SL_2(\mathbb{Z})$
be a subgroup. Let $g_{a},g_{b}\in SL_2(\mathbb{Z})$ be such that $g_{a}(\infty)=a$
and $g_{b}(\infty)=b$. Let $T=\left(\begin{array}{cc}
1 & 1\\
0 & 1
\end{array}\right)$. Then $\Gamma\backslash SL_2(\mathbb{Z})$ admits a right action by $T$,
via multiplication, and $\Gamma a=\Gamma b$ if and only if $\Gamma g_{a}$,
$\Gamma g_{b}$ lie in the same $\langle T \rangle$-orbit of $\Gamma\backslash SL_2(\mathbb{Z})$.
\end{prop}

\begin{proof}
If there is some $\gamma\in\Gamma$ such that $\gamma a=b$, then
$\gamma g_{a}(\infty)=g_{b}(\infty)$, and so 
\[
g_{b}^{-1}\gamma g_{a}\in \Stab\,_{SL_2(\mathbb{Z})}(\infty)=\left\langle T\right\rangle .
\]

Therefore we get $g_{b}^{-1}\gamma g_{a}=T^{n}$ for some integer
$n\in\mathbb{Z}$. Thus, there exists $n\in\mathbb{Z}$ for which
$g_{b}T^{n}\in\Gamma g_{a}$. The converse follows simialrly.
\end{proof}
Thus, in order to test for equivalence, we have to compute the $T$
action. This can be done a priori for all cosets, as we proceed to show.

\begin{defn}
Let $\Gamma \subseteq SL_2(\mathbb{Z})$ be a finite index subgroup, and let $R = (r_1, \ldots, r_h)$ be coset representatives for $\Gamma \backslash SL_2(\mathbb{Z})$. 
Let $S = ((o_1, e_1), \ldots, (o_h, e_h)) $ be pairs of integers. 
We say that $S$ is a \emph{$\langle T \rangle$-orbit table} of $\Gamma$ with respect to $R$ if the following two assumptions hold.
\begin{enumerate}
\item
The cosets $\Gamma r_i$ and $\Gamma r_j$ lie in the same $\langle T \rangle$-orbit iff $o_i = o_j$.
\item
For all $i,j$ such that $o_i = o_j$,  we have $r_j \cdot T^{e_i-e_j}\cdot r_i^{-1} \in \Gamma$.
\end{enumerate}
\end{defn}

\begin{lem} \label{lem: T orbit table}
There exists an algorithm that given a finite index subgroup $\Gamma \subseteq SL_2(\mathbb{Z})$, and a list of coset representatives for $\Gamma \backslash SL_2(\mathbb{Z})$, $R$, computes a $\langle T \rangle$-orbit table of $\Gamma$ with respect to $R$ in $O([SL_2(\mathbb{Z}):\Gamma])$ basic $\operatorname{CosetIndex}$ operations.
\end{lem}
\begin{proof}
Apply Algorithm \ref{algo: Constructing the T orbit table}. It just goes over the cycles, hence has linear running time.
\end{proof}

\begin{algo} 
\label{algo: Constructing the T orbit table} $\operatorname{OrbitTable}(\Gamma, R)$. 
Constructing the $\langle T \rangle$-orbit table.
\end{algo}

\textbf{Input :} 
\begin{itemize}
\item a finite index subgroup $\Gamma \subseteq SL_2(\mathbb{Z})$.
\item a list of coset representatives $R = (r_1, \ldots, r_h)$ for $\Gamma\backslash SL_2(\mathbb{Z})$.
\end{itemize}

\textbf{Output :} a list $S$, which is a $\langle T \rangle$-orbit table of $\Gamma$ with respect to $R$.
\begin{enumerate}
\item for $i \in \{1,\ldots,h\}$ do
\begin{enumerate}
\item $s_i := (0,0)$
\item $t_i := \operatorname{CosetIndex}(r_i \cdot T,R,\Gamma)$
\end{enumerate}
\item $i:=1, o:=1, e:=0$
\item while $i\le h$ do
\begin{enumerate}
\item $s_i:=(o,e)$
\item $i:=t_i$
\item if $s_i \ne (0,0)$ then
\begin{enumerate}
\item $o:=o+1;$
\item $e:=0$;
\item while $i\le h$ and $s_i \ne (0,0)$ do
\begin{enumerate}
\item $i:=i+1;$
\end{enumerate}
\end{enumerate}
\item else
\begin{enumerate}
\item $e:=e+1$
\end{enumerate}
\end{enumerate}
\item return $(s_1, \ldots, s_h)$
\end{enumerate}

And then our algorithm for cusp equivalence testing is rather simple.
\begin{algo} \label{algo: cusp equivalence}
$\operatorname{CuspEquiv}(\Gamma, R, S, a, b)$. Test for cusp equivalence. 
\end{algo}

\textbf{Input :} 
\begin{itemize}
\item a finite index subgroup $\Gamma \subset SL_2(\mathbb{Z})$.
\item a list of coset representatives $R = (r_1, \ldots, r_h)$ for $\Gamma\backslash SL_2(\mathbb{Z})$.
\item an orbit table $S = ((o_1, e_1), \ldots, (o_h, e_h))$ (the output of Algorithm \ref{algo: Constructing the T orbit table}).
\item $a,b\in\mathbb{P}^{1}(\mathbb{Q})$.
\end{itemize}

\textbf{Output :} If there exists $\gamma\in\Gamma$ such that $\gamma a=b$,
then $(true,\gamma)$, else $(false,1)$. 
\begin{enumerate}
\item Find $g_{a},g_{b}$ such that $g_{a}(\infty)=a$, $g_{b}(\infty)=b$.
\item $i_{a}:=\operatorname{CosetIndex}(g_{a},R,\Gamma)$, $i_{b}:=\operatorname{CosetIndex}(g_{b},R,\Gamma)$
\item if $o_{i_a}\ne o_{i_b}$ then return $(false,1)$;
\item return $(true,g_{b}\cdot T^{e_{i_{a}}-e_{i_{b}}}\cdot g_{a}^{-1})$
\end{enumerate}

\begin{cor} \label{cor: cusp equivalence}
There exists an algorithm that given a finite index subgroup $\Gamma \subset SL_2(\mathbb{Z})$, a list of coset representatives $R$ for $\Gamma\backslash SL_2(\mathbb{Z})$, the $\langle T \rangle$-orbit table of $\Gamma$ with respect to $R$, and $a,b\in\mathbb{P}^{1}(\mathbb{Q})$, returns whether $\Gamma a = \Gamma b$, in which case it returns also an element $\gamma \in \Gamma$ such that $\gamma a = b$, in $O(1)$ basic $\operatorname{CosetIndex}$ operations.
\end{cor}

\begin{proof}
Apply Algorithm \ref{algo: cusp equivalence}. By Proposition \ref{prop: cusp equivalence}, it returns the correct answer.
\end{proof}

Finally, to test whether a cusp vanishes, we use the following proposition
(an analog of \cite[Prop. 8.16]{stein2007modular}). Recall that if $\Gamma \subseteq \Gamma'$ is a normal subgroup, and $Q = \Gamma' / \Gamma$, then $Q$ acts on $\mathbb{M}_k(\Gamma)$ as in \eqref{eqn: Q action}.

\begin{prop} \label{prop: cusp vanishing}
Let $\Gamma\subseteq\Gamma'\subseteq SL_{2}(\mathbb{Z})$ be finite index subgroups
such that $\Gamma$ is a normal subgroup of $\Gamma'$. Let $Q:=\Gamma'/\Gamma$, and let 
$\varepsilon:Q\rightarrow\mathbb{Q}(\zeta)^{\times}$
be a character. Suppose $a\in\mathbb{P}^{1}(\mathbb{Q})$ is a cusp.
For $x\in\mathbb{P}^{1}(\mathbb{Q})$, write $[x]$ for the equivalence
class of $x$ in $\Gamma\backslash\mathbb{P}^{1}(\mathbb{Q})$. Then
$a$ vanishes modulo the relations
\[
[\gamma' x]=\varepsilon(\gamma')\cdot[x]\quad\forall\gamma'\in\Gamma', x \in \mathbb{P}^1(\QQ)
\]
if and only if there exists $q\in Q$ with $\varepsilon(q)\ne1$ such
that $[q\cdot a]=[a]$. 
\end{prop}

\begin{proof}
First suppose that such a $q$ exists. Then 
\[
[a]=[q\cdot a]=\varepsilon(q)\cdot[a].
\]
But $\varepsilon(q)\ne1$, hence $[a]=0$. 

Conversely, suppose that $[a]=0$. Because all relations are two-term
relations and the $\Gamma$-relations identify $\Gamma$-orbits, there
must exist $\alpha,\beta\in\Gamma'$ such that 
\[
[\alpha\cdot a]=[\beta\cdot a]
\]
and $\varepsilon(\alpha)\ne\varepsilon(\beta)$. Indeed, if this did
not occur, we could mod out by the $\varepsilon$ relations by writing
each $[\alpha\cdot a]$ in terms of $[a]$ and there would be no further
relations to kill $[a]$. Next observe that 
\[
[\beta^{-1}\alpha\cdot a]=\varepsilon(\beta^{-1})\cdot[\alpha\cdot a]=\varepsilon(\beta)^{-1}[\beta\cdot a]=[a]
\]
so that if $q \in Q$ is the image of $\beta^{-1}\alpha \in \Gamma'$, then $\varepsilon(q)=\varepsilon(\beta)^{-1}\varepsilon(\alpha)\ne1$. 
\end{proof}
This gives rise to an algorithm for checking whether a cusp vanishes.

\begin{algo} \label{algo: cusp vanishing}
$\operatorname{CuspVanishing}(\Gamma, \varepsilon, s, R, S, a)$. Test for cusp vanishing.
\end{algo}

\textbf{Input :} 
\begin{itemize}
\item a finite index subgroup $\Gamma \subset SL_2(\mathbb{Z})$.
\item a character $\varepsilon : Q = \Gamma' / \Gamma \rightarrow \mathbb{Q}(\zeta)^{\times}$. 
\item a section $s: Q \rightarrow \Gamma'$ of the quotient map.
\item a list of coset representatives $R$ for $\Gamma \backslash SL_2(\mathbb{Z})$.
\item an orbit table $S$ (the output of Algorithm  \ref{algo: Constructing the T orbit table}).
\item $a\in\mathbb{P}^{1}(\mathbb{Q})$.
\end{itemize}

\textbf{Output :} If $[a]=0$ in $\mathbb{B}_{k}(\Gamma,\varepsilon)$
return $true$, return $false$. 
\begin{enumerate}
\item For each $q\in Q$ such that $\varepsilon(q)\ne1$ do:
\begin{enumerate}
\item $equiv,\gamma:=\operatorname{CuspEquiv}(\Gamma, R,S,s(q)\cdot a,a)$.
\item If $equiv$ then return $true$.
\end{enumerate}
\item return $false$.
\end{enumerate}

\begin{cor} \label{cor: cusp vanishing}
There exists an algorithm that given finite index subgroup $\Gamma \subseteq \Gamma' \subseteq SL_2(\mathbb{Z})$ such that $\Gamma$ is normal in $\Gamma'$, a character $\varepsilon: \Gamma' / \Gamma \rightarrow \mathbb{Q}(\zeta)^{\times}$, a list of coset representatives $R$ for $\Gamma \backslash SL_2(\mathbb{Z})$, a $\langle T \rangle$-orbit table of $\Gamma$ with respect to $R$, and a cusp $ a \in \mathbb{P}^{1}(\mathbb{Q})$, returns whether $[a]=0$ in $\mathbb{B}_{k}(\Gamma, \varepsilon)$ in $O([\Gamma':\Gamma])$ basic $\operatorname{CosetIndex}$ operations. 
\end{cor} 

\begin{proof}
Apply algorithm \ref{algo: cusp vanishing}. The correctness follows from proposition \ref{prop: cusp vanishing}.
\end{proof}

Relying on \ref{algo: cusp equivalence} and \ref{algo: cusp vanishing}, we may now formulate an algorithm to compute the boundary map $\mathbb{M}_k(\Gamma, \varepsilon) \rightarrow \mathbb{B}_k(\Gamma, \varepsilon)$.
In the algorithm below, $\operatorname{CuspIndex}$ is a function that given a finite index subgroup $\Gamma \subseteq SL_2(\mathbb{Z})$, a list of coset representatives $R$ for $\Gamma \backslash SL_2(\mathbb{Z})$, a $\langle T \rangle$-orbit table of $\Gamma$ with respect to $R$, a cusp $a$
and a list of cusps $C$, returns an index $i$ such that $a$ is equivalent to $C_i$ in $\mathbb{B}_k(\Gamma)$ and an
element $\gamma \in \Gamma$ such that $a=\gamma \cdot C_i$.
If $a \notin C$, it appends $a$ to $C$. This function is obtained by using Algorithm \ref{algo: cusp equivalence} repeatedly. 
\begin{algo}
\label{algo: Computation of boundary map} $\operatorname{BoundaryMap}(\Gamma, \Gamma', \varepsilon, k, B, R, R')$. Computation of the boundary map.
\end{algo}

\textbf{Input : }
\begin{itemize}
\item finite index subgroups $\Gamma \subseteq \Gamma' \subseteq SL_2(\mathbb{Z})$ such that $\Gamma$ is normal in $\Gamma'$.
\item a character $\varepsilon: \Gamma' / \Gamma \rightarrow \mathbb{Q}(\zeta)^{\times}$. 
\item an integer $k \ge 2$.
\item $B = (b_1, \ldots, b_n)$, an ordered basis of Manin symbols for $\mathbb{M}_{k}(\Gamma,\varepsilon)$. 
\item $R = (r_1, \ldots, r_h)$, coset representatives for $\Gamma \backslash SL_2(\mathbb{Z})$.
\item $R' = (r'_1, \ldots, r'_{h'})$, coset representatives for $\Gamma' \backslash SL_2(\mathbb{Z})$.
\end{itemize}

\textbf{Output :} 
\begin{itemize}
\item a matrix $A$.
\item a list of boundary Manin symbols $C$.
\end{itemize}

such that $A$ represents the boundary map $\mu\circ\partial$
with respect to the bases $B,C$.

\begin{enumerate}
\item $C:=\emptyset$
\item $A:= 0 \in \QQ(\zeta)^{n \times \NN}$
\item $S := \operatorname{OrbitTable}(\Gamma, R)$
\item $S' := \operatorname{OrbitTable}(\Gamma', R')$
\item for $i \in \{1,2,\ldots, n \}$ do
\begin{enumerate}
\item Write $b_i = [P,g]$
\item $w:=\operatorname{weight}(b_i)$
\item if $w=k-2$:
\begin{enumerate}
\item $a := g \cdot \infty$
\item if not $\operatorname{CuspVanishing}(\Gamma, \varepsilon, s, R, S, a)$:
\begin{enumerate}
\item $j,\gamma, C :=\operatorname{CuspIndex}(\Gamma', R', S', a, C)$
\item $A_{ij} := A_{ij} +\varepsilon^{-1}(\gamma)$
\end{enumerate}
\end{enumerate}
\item if $w=0$:
\begin{enumerate}
\item $a := g \cdot 0$
\item if not $\operatorname{CuspVanishing}(\Gamma, \varepsilon, s, R, S, a)$:
\begin{enumerate}
\item $j,\gamma, C :=\operatorname{CuspIndex}(\Gamma', R', S', a, C)$
\item $A_{ij} := A_{ij} - \varepsilon^{-1}(\gamma)$
\end{enumerate}
\end{enumerate}
\end{enumerate}
\item return $(A_{ij})_{j=1}^{|C|},C$
\end{enumerate}

\begin{rem}
Note that when $k=2$, we execute both $(c)$ and $(d)$. 
\end{rem}
We can now proceed to prove Theorem \ref{thm: cuspidal modular symbols}. 

\begin{proof} (of Theorem  \ref{thm: cuspidal modular symbols}.).
The computation of $\mathbb{S}_k(\Gamma, \varepsilon)$ is performed by using Corollary \ref{cor: construct modular symbols} to construct a basis $B$ of modular symbols for $\mathbb{M}_k(\Gamma, \varepsilon)$,  and then applying Algorithm \ref{algo: Computation of boundary map}. 

The computation of the $\langle T \rangle$-orbit tables for both $\Gamma$ and $\Gamma'$ using Algorithm \ref{algo: Constructing the T orbit table} costs $O([SL_2(\mathbb{Z}):\Gamma])$ basic $\operatorname{CosetIndex}$ operations by Lemma \ref{lem: T orbit table}.

The function $\operatorname{CuspIndex}$ is implemented by applying Algorithm \ref{algo: cusp equivalence}. 
The check whether a cusp vanishes is done by applying Algorithm \ref{algo: cusp vanishing}.
By Corollary \ref{cor: cusp equivalence} and Corollary  \ref{cor: cusp vanishing}, these algorithms return the required output in $O(1)$ basic $\operatorname{CosetIndex}$ operations.

Since there are $2\cdot[SL_2(\mathbb{Z}):\Gamma']$ Manin symbols of either 
maximal or minimal weight, and we have to check equivalence of at most each of these against each of the cusps in the list, which has at most $c_{\varepsilon}(\Gamma)$ elements in it, it follows that the running time for Algorithm
\ref{algo: Computation of boundary map} is 
$O([SL_2(\mathbb{Z}):\Gamma'] \cdot c_{\varepsilon}(\Gamma) + [SL_2(\mathbb{Z}):\Gamma])$.
Thus, we have an algorithm for computing $\mathbb{S}_{k}(\Gamma,\varepsilon)$ in 
$O([SL_2(\mathbb{Z}):\Gamma'] \cdot c_{\varepsilon}(\Gamma) +[SL_2(\mathbb{Z}):\Gamma])$
basic $\operatorname{CosetIndex}$ operations.
\end{proof}

\subsection{Pairing Modular Symbols and Modular Forms}

Let $\overline{S}_{k}(\Gamma)=\{\overline{f}:f\in S_{k}(\Gamma)\}$
be the space of antiholomorphic cusp forms. We then have a pairing
\[
\left\langle \cdot,\cdot\right\rangle :\left(S_{k}(\Gamma)\oplus\overline{S}_{k}(\Gamma)\right)\times\mathbb{M}_{k}(\Gamma)\rightarrow\mathbb{C}
\]
 given by 
\begin{equation}
\left\langle (f_{1},f_{2}),P\otimes\{\alpha,\beta\}\right\rangle =\int_{\alpha}^{\beta}f_{1}(z)P(z,1)dz+\int_{\alpha}^{\beta}f_{2}(z)P(\overline{z},1)d\overline{z}\label{eq:pairing of modular symbols and modular forms}
\end{equation}

The following theorem is the final ingredient needed to get a handle
on the actual spaces of modular forms.
\begin{thm} \label{thm: pairing of cusps}
(\cite[Theorem 0.2]{shokurov1980study}, and \cite[Theorem 3]{merel1994universal}) The pairing 
\[
\left\langle \cdot,\cdot\right\rangle :\left(S_{k}(\Gamma)\oplus\overline{S}_{k}(\Gamma)\right)\times\mathbb{S}_{k}(\Gamma,\mathbb{C})\rightarrow\mathbb{C}
\]
 is a nondegenerate pairing of complex vector spaces.
\end{thm}

This is still not completely satisfactory, as we would like to separate
the holomorphic forms from their antiholomorphic counterparts. 

\subsubsection{The Action of Complex Conjugation }
\begin{defn}
Let $\eta=\text{\ensuremath{\left(\begin{array}{cc}
 -1  &  0\\
 0  &  1 
\end{array}\right)}}\in GL_{2}(\mathbb{Z})$. If $\eta$ normalizes $\Gamma$, i.e. $\eta\Gamma\eta=\Gamma$,
we say that $\Gamma$ is \emph{of real type}. 
\end{defn}

We then have the following result.
\begin{prop} \label{prop: star involution}
(Merel, \cite[Proposition 7]{merel1994universal}) Let $\Gamma\subseteq SL_2(\mathbb{Z})$
be a subgroup of real type. Then there is a complex linear involution
$\iota:S_{k}(\Gamma)\oplus\overline{S}_{k}(\Gamma)\rightarrow S_{k}(\Gamma)\oplus\overline{S}_{k}(\Gamma)$
which exchanges $S_{k}(\Gamma)$ and $\overline{S}_{k}(\Gamma)$,
given by $\iota(f)(z)=f(-\overline{z})$. The involution $\iota^{\vee}:\mathbb{M}_{k}(\Gamma)\rightarrow\mathbb{M}_{k}(\Gamma)$
defined by 
\begin{equation} \label{eqn: star involution}
\iota^{\vee}(v\otimes\{\alpha,\beta\})=-\eta v\otimes\{\eta\alpha,\eta\beta\}
\end{equation}
 is adjoint to $\iota$ with respect to the pairing \eqref{eq:pairing of modular symbols and modular forms}.
Moreover, $\iota^{\vee}$ acts on Manin symbols via 
\[
\iota^{\vee}([v,g])=-[\eta v,\eta g\eta^{-1}]
\]
\end{prop}

We may now state the final result about the pairing, which explains
how modular symbols and modular forms are related. 
\begin{thm} \label{thm: star involution}
(Merel, \cite[Proposition 8]{merel1994universal}) Let $\mathbb{S}_{k}(\Gamma)^{+}$
be the $+1$ eigenspace for $\iota^{*}$ on $\mathbb{S}_{k}(\Gamma)$.
The pairing \eqref{eq:pairing of modular symbols and modular forms}
induces a nondegenerate bilinear pairing
\[
\mathbb{S}_{k}(\Gamma)^{+}\times S_{k}(\Gamma)\rightarrow\mathbb{C}
\]
\end{thm}

We denote by $c(\Gamma) = \dim \mathbb{B}_2(\Gamma)$ the number of cusps of $\Gamma$. Then it follows that

\begin{cor}
There exists an algorithm that given a finite index subgroup $\Gamma \subseteq SL_2(\mathbb{Z})$ and an integer $k \ge 2$, computes a basis for $S_k(\Gamma)$ in $O([SL_2(\mathbb{Z}) : \Gamma] \cdot c(\Gamma))$ basic $\operatorname{CosetIndex}$ operations. 
\end{cor}

\begin{proof}
In order to compute $\mathbb{S}_{k}(\Gamma)^{+}$, it is necessary
to replace $\mathbb{B}_{k}(\Gamma)$ by its quotient modulo the additional
relation $[(-u,v)]=[(u,v)]$ for all cusps $(u,v)$. Algorithm \ref{algo: Computation of boundary map}
can be modified, as in \cite{stein2007modular}, to treat that case by changing $\operatorname{CuspEquiv}$ to check equivalence with both $(u,v)$ and $(-u,v)$.
Thus, the complexity described in Theorem \ref{thm: cuspidal modular symbols} is valid also for the computation of 
$\mathbb{S}_k(\Gamma, \varepsilon)^{+}$. In particular, taking $\Gamma' = \Gamma$ and $\varepsilon$ to be the trivial character, we obtain the result.
\end{proof}

\section{Hecke Operators} \label{sec: Hecke operators}

Now that we have a realization of $S_{k}(\Gamma_{G})$, we would like
to be able to compute Hecke Operators on this space.
First, let us recall a few basic facts about Hecke operators in general. 

\subsection{Hecke Operators on $M_{k}(\Gamma)$} \label{subsec: Hecke operators definitions}
\begin{defn}
\label{def:Hecke Operator Double Coset}Let $\Gamma\subseteq SL_{2}(\mathbb{Z})$
and suppose $\Delta\subseteq GL_{2}^{+}(\mathbb{Q})$ is a set such that
$\Gamma\Delta=\Delta\Gamma=\Delta$, and $\Gamma\backslash\Delta$
is finite. Let $R$ be a set of representatives for $\Gamma\backslash\Delta$.
Let 
\begin{align*}
T_{\Delta}: M_{k}(\Gamma) &\rightarrow M_{k}(\Gamma) \\
T_{\Delta}(f) &= \sum_{\alpha\in R} f\vert_{\alpha}
\end{align*}
where $f\vert_{\alpha}$ is the usual right action of $GL_{2}^{+}(\mathbb{Q})$
given by \eqref{eq:action of GL2Q on modular forms}. The operator
$T_{\Delta}$ is called the \emph{Hecke operator associated to $\Delta$. }
\end{defn}

It is a standard result that this operator is well-defined and independent of $R$. 

\begin{prop} (\cite[Prop. 3.37]{shimura1971introduction}, \cite[Chapter 5.1]{diamond2005first})
The image of $T_{\Delta}$ lies in $M_k(\Gamma)$, and $T_{\Delta}$ does not depend on the choice of $R$.
\end{prop}

\begin{defn}
The algebra generated by the $\{T_{\Delta}\}_{\Delta}$ when $\Delta$ runs over all subsets of $GL_2^{+}(\mathbb{Q})$ such that $\Gamma\Delta=\Delta\Gamma=\Delta$ is \emph{the Hecke algebra of $M_k(\Gamma)$}.
\end{defn}
For $\alpha \in GL_2^{+}(\mathbb{Q})$, we denote $T_{\alpha}=T_{\Gamma\alpha\Gamma}$. 

\begin{cor}\label{cor:Hecke operator double coset}
The Hecke algebra of $M_k(\Gamma)$ is generated by the operators $\{T_{\alpha}\}_{\alpha \in GL_2^{+}(\mathbb{Q})}$.
\end{cor}
\begin{proof}
It is a consequence of the definition that $\Delta$ must be a union of double cosets
$\Gamma\alpha\Gamma$, where $\alpha\in GL_{2}^{+}(\mathbb{Q})$. In particular,
the Hecke operators are linearly spanned by the $\{T_{\alpha}\}$.
\end{proof}

\begin{example} Here are two standard choices for $\Delta$. 
\begin{enumerate}
\item
If $\Gamma=\Gamma_{1}(N)$, one usually considers for any $n\in\mathbb{Z}$,
the set 
\[
\Delta_{n}^{1}:=\left\{ g\in M_{2}(\mathbb{Z})\mid\det(g)=n \text{ and } g\equiv\left(\begin{array}{cc}
1 & *\\
0 & n
\end{array}\right)\bmod N\right\} 
\]
One usually denotes $T_{n}=T_{\Delta_{n}^{1}}$. We note that $\Delta_{n}^{1}=\Gamma\left(\begin{array}{cc}
1 & 0\\
0 & n
\end{array}\right)\Gamma$ , so that $T_{n}=T_{\alpha}$, with $\alpha=\left(\begin{array}{cc}
1 & 0\\
0 & n
\end{array}\right)$.
\item
Let $\Gamma=\Gamma(N)$. We may take for $n\in\mathbb{Z}$, the set
\[
\Delta_{n}:=\left\{ g\in M_{2}(\mathbb{Z})\mid\det(g)=n \text{ and } g\equiv\left(\begin{array}{cc}
1 & 0\\
0 & n
\end{array}\right)\bmod N\right\} 
\]
Note that again we are taking $\Delta_{n}=\Gamma\left(\begin{array}{cc}
1 & 0\\
0 & n
\end{array}\right)\Gamma$. 
\end{enumerate}
\end{example}

As our implementation of modular forms is by using the dual space
of modular symbols, we need the dual notion.

\begin{defn}
Let $\Delta\subseteq GL_{2}(\mathbb{Q})$ be such that $\Gamma\Delta=\Delta\Gamma = \Delta$
and such that $\Gamma\backslash\Delta$ is finite. Let $R$ be a set
of representatives of $\Gamma\backslash\Delta$. Let 
\begin{align*}
T_{\Delta}^{\vee}:\mathbb{M}_{k}(\Gamma) &\rightarrow \mathbb{M}_{k}(\Gamma) \\
v\otimes\{\alpha,\beta\} &\mapsto \sum_{\delta\in R}(\delta v)\otimes\{\delta \alpha,\delta \beta\}
\end{align*}
\end{defn}

Again, this map does not depend on the choice of $R$. 
We then have the following result. 
\begin{prop} \label{prop: duality Hecke operators}
(Merel, \cite[Proposition 10]{merel1994universal}) The operators $T_{\Delta}$
and $T_{\Delta}^{\vee}$ are adjoint with respect to the bilinear
pairing $\left\langle \cdot,\cdot\right\rangle $ defined in \eqref{eq:pairing of modular symbols and modular forms}.
\end{prop}

Therefore, by Theorem \ref{thm: pairing of cusps}, in order to compute the restriction of the Hecke operators $T_{\Delta}$ to $S_k(\Gamma) \oplus \overline{S}_k(\Gamma)$, it will suffice to compute the Hecke operators $T_{\Delta}^{\vee}$ on $\mathbb{S}_k(\Gamma)$. Note that in order to compute the restriction to $S_k(\Gamma)$ using Theorem \ref{thm: star involution}, we would need $T_{\Delta}$ to commute with the map $\iota$, defined in Proposition~\ref{prop: star involution}. 

\begin{lem} \label{lem: Hecke commutes with star}
Let $\Gamma \subseteq SL_2(\mathbb{Z})$ be a subgroup of real type (recall Definition~\ref{def: real type}). Let $\alpha \in GL_2^{+}(\mathbb{Q})$ be such that $\eta^{-1} \alpha \eta \in \Gamma \alpha \Gamma$. Then in $\mathbb{M}_k(\Gamma)$ we have $T_{\alpha}^{\vee} \circ \iota^{\vee} = \iota^{\vee} \circ T_{\alpha}^{\vee}$.
\end{lem}

\begin{proof}
Let $R$ be a set of representatives for $\Gamma \backslash \Gamma \alpha \Gamma$.
For each $\delta \in \Gamma \alpha \Gamma$, we have by the assumptions on $\Gamma$ and $\alpha$ that
\begin{equation*}
\eta^{-1} \delta \eta \in \eta^{-1} \Gamma \alpha \Gamma \eta = 
\left(   \eta^{-1} \Gamma \eta  \right) \eta^{-1} \alpha \eta \left(   \eta^{-1} \Gamma \eta  \right) = 
 \Gamma \eta^{-1} \alpha \eta \Gamma = \Gamma \alpha \Gamma.
\end{equation*}
Therefore, we see that 
\begin{equation*}
\Gamma \alpha \Gamma = \bigsqcup_{\delta \in R} \Gamma \delta = 
\eta^{-1} \left( \bigsqcup_{\delta \in R} \Gamma \delta \right) \eta =  
\bigsqcup_{\delta \in R} \eta^{-1} \Gamma \delta \eta  =
\bigsqcup_{\delta \in R} \left( \eta^{-1} \Gamma \eta \right) \eta^{-1}\delta \eta =
\bigsqcup_{\delta \in R} \Gamma \eta^{-1} \delta \eta,
\end{equation*}
so that $R' = \{\eta^{-1} \delta \eta \mid \delta \in R \}$ is also a set of representatives for $\Gamma \backslash \Gamma \alpha \Gamma$.
Now
\begin{align*}
T_{\alpha}^{\vee} ( \iota^{\vee} (v \otimes \{\alpha, \beta\})) 
&=  T_{\alpha}^{\vee} (-\eta v \otimes \{\eta \alpha, \eta \beta\}) \\
&= \sum_{\delta \in R} (- \delta \eta v \otimes \{\delta \eta \alpha, \delta \eta \beta\}) \\
&= \sum_{\delta \in R} (- \eta (\eta^{-1}\delta \eta) v \otimes \{\eta (\eta^{-1} \delta \eta) \alpha, \eta (\eta^{-1} \delta \eta) \beta\}) \\
&=  \sum_{\delta' \in R'} (- \eta \delta' v \otimes \{ \eta \delta' \alpha, \eta \delta' \beta\})=
\iota^{\vee}(T_{\alpha}^{\vee}(v \otimes \{\alpha, \beta\})). \qedhere
\end{align*}
\end{proof}

\begin{cor} \label{cor: Hecke commutes with star}
Let $G \subseteq GL_2(\mathbb{Z} / N \mathbb{Z})$ be a subgroup of real type (recall Definition~\ref{def: real type}). Let $\alpha \in M_2(\mathbb{Z})$ be such that $\lambda_N(\alpha) \in G$. Then in $\mathbb{M}_k(\Gamma_G)$ we have $T_{\alpha}^{\vee} \circ \iota^{\vee} = \iota^{\vee} \circ T_{\alpha}^{\vee}$.
\end{cor}

\begin{proof}
Since $G$ is of real type, we have 
$$
\lambda_N( \eta \Gamma_G \eta^{-1}) = \lambda_N(\eta) G \lambda_N(\eta)^{-1} = G.
$$
It follows that $\eta \Gamma_G \eta^{-1} \subseteq \Gamma_G$, and hence $\Gamma_G$ is of real type. 

Also, as $\lambda_N(\alpha) \in G$, we see that 
\begin{equation*}
\lambda_N(\eta^{-1} \alpha \eta) \in \lambda_N(\eta)^{-1} G \lambda_N(\eta) = G.
\end{equation*}
Moreover, since $\det(\alpha) = \det(\eta^{-1} \alpha \eta)$ and $d_1(\alpha) = d_1(\eta^{-1} \alpha \eta)$, where $d_1$ is the greatest common divisor of all the entries, $\alpha$ and $\eta^{-1} \alpha \eta$ have the same Smith normal form, showing that
$SL_2(\mathbb{Z}) \cdot \alpha \cdot  SL_2(\mathbb{Z}) = SL_2(\mathbb{Z}) \cdot \eta^{-1} \alpha \eta \cdot SL_2(\mathbb{Z})$.
We invoke \cite[Lemma 3.29(1)]{shimura1971introduction}, noting that 
$\lambda_N(\Gamma_G \alpha) = G = \lambda_N(\alpha \Gamma_G)$, and so, in Shimura's notation, $\alpha \in \Phi$. 
We deduce that
$$\Gamma_G \alpha \Gamma_G = \{g \in M_2(\mathbb{Z}) \mid SL_2(\mathbb{Z}) \cdot \alpha \cdot  SL_2(\mathbb{Z}) = SL_2(\mathbb{Z}) \cdot g \cdot  SL_2(\mathbb{Z}) \text{ and } \lambda_N(g) \in G\}.$$
It then follows that $\eta^{-1} \alpha \eta \in \Gamma_G \alpha \Gamma_G$.
The result follows from Lemma \ref{lem: Hecke commutes with star}.
\end{proof}

\subsection{Naive Computation of $T_{\alpha}^{\vee}$} \label{subsec: naive Hecke operator}

We will now present a naive algorithm to compute the Hecke operator
$T_{\alpha}^{\vee}$ on $\mathbb{M}_k(\Gamma)$ for an arbitrary $\alpha\in GL_{2}^{+}(\mathbb{Q})$.

\begin{algo} \label{algo: Hecke operator double coset}
$\operatorname{HeckeOperator}(\alpha, x, B)$. Hecke Operator $T_{\alpha}^{\vee}$ on $\mathbb{M}_{k}(\Gamma)$.
\end{algo}

\textbf{Input :} 
\begin{itemize}
\item $\alpha\in GL_{2}^{+}(\mathbb{Q})$.
\item an element $x=v\otimes\{\alpha,\beta\}\in\mathbb{M}_{k}(\Gamma)$.
\item a basis $B$ for the space of modular symbols $\mathbb{M}_{k}(\Gamma)$. 
\end{itemize}

\textbf{Output :} a vector representing $T_{\alpha}^{\vee}(x)$ w.r.t.
the basis $B$.
\begin{enumerate}
\item Let $H:=\Gamma\cap\alpha^{-1}\Gamma\alpha$. 
\item Let $S$ be a set of representatives for $H\backslash\Gamma$.
\item Let $R:=\alpha\cdot S=\{\alpha\cdot x:x\in S\}$.
\item Return $\sum_{r\in R}[\pi(r\cdot x)]_{B}$, where $[\pi(r\cdot x)]_{B}$
is the vector representing the modular symbol $\Gamma\cdot(rv)\otimes\{r\alpha,r\beta\}$
w.r.t. $B$. 
\end{enumerate}
\begin{rem}
Even though this is unclear at first sight, the functions $(\Gamma_{1},\Gamma_{2})\mapsto\Gamma_{1}\cap\Gamma_{2}$
and $(\Gamma,x)\mapsto x^{-1}\Gamma x$ also require a nontrivial
implementation. 

Note also that $\Gamma\subseteq SL_{2}(\mathbb{Z})$ is defined using
$G\subseteq GL_{2}(\mathbb{Z}/N\mathbb{Z})$. Therefore, if $\alpha\in GL_{2}^{+}(\mathbb{Q})$
is such that $\alpha^{-1}\Gamma\alpha\nsubseteq SL_{2}(\mathbb{Z})$,
a straightforward computation could be expensive (either find $N$
such that $N\cdot\alpha^{-1}\Gamma\alpha\subseteq SL_{2}(\mathbb{Z})$
and change computations to fit, or work in $GL_{2}^{+}(\mathbb{Q})$
with generators - that requires finding the generators of $\Gamma$).
Moreover, as $GL_{2}^{+}(\mathbb{Q})$ is not finitely generated,
we do not have ready made tools for computing the intersection $\Gamma\cap\alpha^{-1}\Gamma\alpha$. 
\end{rem}

This has brought us to consider the following alternate path:

\subsubsection*{\label{subsec:Strategy-for-computation}Strategy for computation
of Algorithm \ref{algo: Hecke operator double coset}}

(a) Implement the function $(\Gamma,x)\mapsto x^{-1}\Gamma x$ for $x \in GL_{2}^{+}(\mathbb{Q})$ when $x^{-1}\Gamma x\subseteq SL_{2}(\mathbb{Z})$. 

(b) Implement an intersection function $(\Gamma_{1},\Gamma_{2})\mapsto\Gamma_{1}\cap\Gamma_{2}$
for $\Gamma_{1},\Gamma_{2}\subseteq SL_{2}(\mathbb{Z})$.

(c) Write the set of representatives $\Gamma\backslash\Gamma\alpha\Gamma=\alpha\cdot\left(\left((\alpha^{-1}\Gamma\alpha)\cap\Gamma\right)\backslash\Gamma\right)$
as a sequence of operations of type (a), (b). 

Here is an implementation of the intersection in step 1, using (a)
and (b).

Below, we denote by $\Gamma^0(N)$ the congruence subgroup $\Gamma_G$, where
$$ 
G=\left\{ \left(\begin{array}{cc}
a & 0 \\
c & d
\end{array}\right)\mid a,d \in(\mathbb{Z}/N\mathbb{Z})^{\times},c\in\mathbb{Z}/N\mathbb{Z}\right\}.
$$ 

\begin{algo} \label{algo: Conjugation and intersection}
$\operatorname{ConjInter}(\Gamma,\alpha)$. Conjugation and Intersection. 
\end{algo}

\textbf{Input : }
\begin{itemize}
\item $\Gamma\subseteq SL_{2}(\mathbb{Z})$ a congruence subgroup.
\item $\alpha\in GL_{2}^{+}(\mathbb{Q})$ such that $\alpha^{-1}\Gamma\alpha\subseteq SL_{2}(\mathbb{Z})$.
\end{itemize}

\textbf{Output :} $\Gamma\cap\alpha^{-1}\Gamma\alpha$.
\begin{enumerate}
\item Find matrices $x,\gamma\in SL_{2}(\mathbb{Z})$ such that $x\cdot \alpha^{-1} \cdot\gamma$
is of the form $\left(\begin{array}{cc}
a & 0\\
0 & b
\end{array}\right)$ with $a^{-1}b\in\mathbb{Z}$. (Smith Normal Form)
\item Let $H:=\Gamma\cap\left(\alpha^{-1}\cdot\left(\Gamma\cap\left(\gamma\cdot\Gamma^{0}(a^{-1}b)\cdot\gamma^{-1}\right)\right)\cdot\alpha\right)$. 
\end{enumerate}
First, let us prove that Algorithm \ref{algo: Conjugation and intersection}
indeed returns the correct answer.
\begin{prop}
Algorithm \ref{algo: Conjugation and intersection} returns
the group $\Gamma\cap\alpha^{-1}\Gamma\alpha$. 
\end{prop}

\begin{proof}
First, note that if $a^{-1}b\in\mathbb{Z}$, then 
\[
\left(\begin{array}{cc}
a & 0\\
0 & b
\end{array}\right)^{-1} SL_2(\mathbb{Z})\left(\begin{array}{cc}
a & 0\\
0 & b
\end{array}\right)\cap SL_2(\mathbb{Z})=\Gamma^{0}(a^{-1}b)
\]
Therefore
\[
\gamma^{-1}\alpha x^{-1} SL_2(\mathbb{Z}) x\alpha^{-1}\gamma\cap SL_2(\mathbb{Z})=\Gamma^{0}(a^{-1}b)
\]
 But $x\in SL_2(\mathbb{Z})$, so that $x^{-1}SL_2(\mathbb{Z})x=SL_2(\mathbb{Z})$. It follows
that 
\[
\alpha SL_2(\mathbb{Z})\alpha^{-1}\cap\gamma\cdot SL_2(\mathbb{Z})\gamma^{-1}=\gamma\Gamma^{0}(a^{-1}b)\gamma^{-1}
\]
But $\gamma\in SL_2(\mathbb{Z})$, so $\gamma SL_2(\mathbb{Z})\gamma^{-1}=SL_2(\mathbb{Z})$.
It follows that 
\[
\alpha SL_2(\mathbb{Z})\alpha^{-1}\cap SL_2(\mathbb{Z})=\gamma\Gamma^{0}(a^{-1}b)\gamma^{-1}
\]
Therefore
\[
\alpha SL_2(\mathbb{Z})\alpha^{-1}\cap\Gamma=\alpha SL_2(\mathbb{Z})\alpha^{-1}\cap SL_2(\mathbb{Z})\cap\Gamma=\Gamma\cap\gamma\Gamma^{0}(a^{-1}b)\gamma^{-1},
\]
so that 
\begin{equation}
SL_2(\mathbb{Z})\cap\alpha^{-1}\Gamma\alpha=\alpha^{-1}\left(\Gamma\cap\gamma\Gamma^{0}(a^{-1}b)\gamma^{-1}\right)\alpha. \label{eq:conjugation by alpha}
\end{equation}
Finally, intersecting with $\Gamma$, one obtains
\[
\alpha^{-1}\Gamma\alpha\cap\Gamma=\Gamma\cap\left(\alpha^{-1}\left(\Gamma\cap\gamma\Gamma^{0}(a^{-1}b)\gamma^{-1}\right)\alpha\right)=H. \qedhere
\]
\end{proof}
In order to be able to use the strategy \ref{subsec:Strategy-for-computation},
we should also show that the conjugations we are computing are in
$SL_{2}(\mathbb{Z})$. 
\begin{lem}
$\gamma\Gamma^{0}(a^{-1}b)\gamma^{-1}\subseteq SL_{2}(\mathbb{Z})$
and $\alpha^{-1}\left(\Gamma\cap\gamma\Gamma^{0}(a^{-1}b)\gamma^{-1}\right)\alpha\subseteq SL_{2}(\mathbb{Z})$.
\end{lem}

\begin{proof}
The first is trivial, as $\gamma\in SL_{2}(\mathbb{Z})$. The second
follows from \eqref{eq:conjugation by alpha}. 
\end{proof}
\begin{rem}
Algorithm \ref{algo: Hecke operator double coset} can
be used to compute $[T_{\alpha}^{\vee}]_{B}$ by applying it to each
of the basis vectors. 
\end{rem}

It remains to show how to compute the conjugation in (a), and the intersection in (b).
For the conjugation in (a), we compute the generators of $\Gamma$ using Farey symbols (see \cite[Theorem 6.1]{kulkarni1991arithmetic}). We then conjugate them to obtain the generators of the conjugated subgroup, and can compute the level and the reduction of this conjugated subgroup.
For the intersection, we present the following algorithm.

\begin{algo} \label{algo: Intersection of induced}
$\operatorname{InterInduced}(G,H)$. Intersection of induced congruence subgroups.
\end{algo}

\textbf{Input :} 
\begin{itemize}
\item $G\subseteq GL_{2}(\mathbb{Z}/N_{G}\mathbb{Z})$.
\item $H\subseteq GL_{2}(\mathbb{Z}/N_{H}\mathbb{Z})$.
\end{itemize}

\textbf{Output :} a group $K\subseteq GL_{2}(\mathbb{Z}/N\mathbb{Z})$
such that $\Gamma_{K}=\Gamma_{G}\cap\Gamma_{H}$. Here $N=\text{lcm}(N_{G},N_{H})$.
\begin{enumerate}
\item $d:=\gcd(N_{G},N_{H})$, $G_{d}:=\{g\in G\mid g\equiv1\bmod d\}$
and $H_{d}:=\{h\in H\mid h\equiv1\bmod d\}$. 
\item $K_{d}:=\lambda_{d}(G)\cap\lambda_{d}(H)$, and $s_{G}:K_{d}\rightarrow G$,
$s_{H}:K_{d}\rightarrow H$ sections of $\lambda_{d}$. 
\item Let $A_{G}\subseteq GL_{2}(\mathbb{Z}/N\mathbb{Z})$ be a set of matrices
$a$ lifting $Generators(G_{d})$ such that $a\equiv1\bmod N_{H}$.
(Chinese Remainder Theorem)
\item Let $A_{H}\subseteq GL_{2}(\mathbb{Z}/N\mathbb{Z})$ be a set of matrices
$a$ lifting $Generators(H_{d})$ such that $a\equiv1\bmod N_{G}$.
(Chinese Remainder Theorem)
\item Let $A_{d}=\{a_{z}:z\in Generators(K_{d})\}\subseteq GL_{2}(\mathbb{Z}/N\mathbb{Z})$
be such that $a_{z}\equiv s_{G}(z)\bmod N_{G}$ and $a_{z}\equiv s_{H}(z)\bmod N_{H}$.
(Chinese Remainder Theorem)
\item Return $K:=\langle A_{G}\cup A_{H}\cup A_{d} \rangle$. 
\end{enumerate}
\begin{prop}
Algorithm \ref{algo: Intersection of induced} returns a group
$K$ such that $\Gamma_{K}=\Gamma_{G}\cap\Gamma_{H}$ .
\end{prop}

\begin{proof}
First note that for any $p\nmid N$, and any $e$, $\lambda_{p^{e}}(\Gamma_{G}\cap\Gamma_{H})=GL_{2}(\mathbb{Z}/p^{e}\mathbb{Z})$.
Then note that 
\[
\lambda_{N}(\Gamma_{G}\cap\Gamma_{H})=\left\{ x\in GL_{2}(\mathbb{Z}/N\mathbb{Z})\mid\lambda_{N_{G}}(x)\in G,\quad\lambda_{N_{H}}(x)\in H\right\} 
\]

We will show that $K = \lambda_N(\Gamma_G \cap \Gamma_H)$, hence the result.
Let $x\in\lambda_{N}(\Gamma_{G}\cap\Gamma_{H})$. Then $\lambda_{d}(x)\in K_{d}$.
Then $\lambda_{d}(x)=\prod z_{i}^{e_{i}}$ for some $\{z_{i}\}\subseteq Generators(K_{d})$.
Write $a_{i}:=a_{z_{i}}$, and consider $a=\prod a_{i}^{e_{i}}$.
It satisfies $\lambda_{d}(a)=\lambda_{d}(x)$, hence $\lambda_{d}(a^{-1}x)=1$. 

In particular, we see that $\lambda_{N_{G}}(a^{-1}x)\in G_{d}$ and $\lambda_{N_{H}}(a^{-1}x)\in H_{d}$.
Therefore there are $\{x_{j}\}\subseteq Generators(G_{d})$ and $\{y_{k}\}\subseteq Generators(H_{d})$
such that $\lambda_{N_{G}}(a^{-1}x)=\prod x_{j}^{f_{j}}$, $\lambda_{N_{H}}(a^{-1}x)=\prod y_{k}^{g_{k}}$.
Let $\{b_{j}\}$ be the lift of the $\{x_{j}\}$ in $A_{G}$ and $\{c_{k}\}$
the lifts of the $\{y_{k}\}$ in $A_{H}$. Then
\[
\lambda_{N_{G}}\left(\prod b_{j}^{f_{j}}\cdot\prod c_{k}^{g_{k}}\right)=\lambda_{N_{G}}(a^{-1}x)
\]
and 
\[
\lambda_{N_{H}}\left(\prod b_{j}^{f_{j}}\cdot\prod c_{k}^{g_{k}}\right)=\lambda_{N_{H}}(a^{-1}x)
\]
hence 
\[
x=\prod a_{i}^{e_{i}}\cdot\prod b_{j}^{f_{j}}\cdot\prod c_{k}^{g_{k}}\in K. \qedhere
\]
\end{proof}

In order to measure the complexity of Algorithm \ref{algo: Hecke operator double coset},
in addition to the $\operatorname{CosetIndex}$ operation, we use $\operatorname{In}$ to denote group membership test in $G$.
We will also need some notation, which we now introduce.

For $\alpha\in GL_{2}^{+}(\mathbb{Q})$, let $d_1(\alpha)\in \mathbb{Q}_{>0}$ be maximal such that $ \alpha \in d_1(\alpha)\cdot M_2(\mathbb{Z})$. 

Let $D(\alpha) := \det(\alpha) / d_1(\alpha)^2 \in \mathbb{Z}$, and let $I_{\alpha,\Gamma} := [\Gamma : \alpha^{-1} \Gamma \alpha \cap \Gamma]$. 

This leads us to the main result of this subsection.

\begin{thm} \label{thm: Hecke operator double coset naive}
There exists an algorithm that given a congruence subgroup of real type $\Gamma \subseteq SL_{2}(\mathbb{Z})$ of level $N$, an element $\alpha \in GL_2^{+}(\mathbb{Q})$ such that $\eta^{-1} \alpha \eta \in \Gamma \alpha \Gamma$ and an integer $k \ge 2$, computes the Hecke operator $T_{\alpha}$ corresponding to the double coset $\Gamma \alpha \Gamma$, on the space of cusp forms $S_k(\Gamma)$, in complexity 
$$O(C \cdot I_{\alpha,\Gamma} \log(N \cdot D(\alpha)) + [SL_2(\mathbb{Z}):\Gamma]^2 \cdot \operatorname{In}).$$ 
\end{thm}

\begin{proof}
We apply Algorithm \ref{algo: Hecke operator double coset}, using Algorithm \ref{algo: Conjugation and intersection} to perform step $(1)$.
The conjugation function is computed by finding a set of generators for $\Gamma$, using Farey symbols (see \cite[Theorem 6.1]{kulkarni1991arithmetic}), and conjugating them. Intersection is computed using Algorithm \ref{algo: Intersection of induced}. 
This computes the dual Hecke operator $T_{\alpha}^{\vee}$ on $\mathbb{M}_k(\Gamma)$. Since $\Gamma$ is of real type, and $\eta^{-1} \alpha \eta \in \Gamma \alpha \Gamma$, by Corollary \ref{cor: Hecke commutes with star} it follows that $T_{\alpha}^{\vee}$ commutes with $\iota^{\vee}$. This shows that it induces an operator on $\mathbb{S}_k(\Gamma)^{+}$. 
Therefore, by Theorem \ref{thm: star involution} and Proposition \ref{prop: duality Hecke operators}, we obtain $T_{\alpha}$. 

The cost of step $(1)$ is dominated by computing the Farey symbols, which yields the cost $O([SL_2(\mathbb{Z}):\Gamma]^2)$.

In step $(4)$, for each element in the basis $[v,g]\in B$, we compute $g\{0,\infty\}$,
to get the modular symbol $v\otimes\{a,b\}$.
Then, for every $r\in R$ apply $r$ to get $rv\otimes\{ra,rb\}$,
and then use continued fractions to convert back to Manin symbols.
This is done $I_{\alpha, \Gamma}$ times.

Since $g\in\Gamma\backslash SL_2(\mathbb{Z})$, it can be represented by 
an element  in $SL_2(\ZZ)$ whose entries are bounded by $N^2$. 
Further, we can choose the representatives $r \in R$ up to a multiplication by a scalar, hence we can bound its entries by 
$N^2 \cdot D(\alpha)$.

It follows that the entries of $r\cdot g$ are bounded by $N^{4}\cdot\det\alpha$,
and so computing continued fraction expansion (Euclid's algorithm)
has $O(\log(N \cdot D(\alpha)))$ steps, for each of which one
has to perform a $\operatorname{CosetIndex}$ operation to find the corresponding
element in the vector space. 


Thus, the algorithm has complexity 
\[
O(C \cdot I_{\alpha,\Gamma} \log(N \cdot D(\alpha)) + [SL_2(\mathbb{Z}):\Gamma]^2 \cdot \operatorname{In}) . \qedhere
\]
\end{proof}

\subsection{Faster Implementation of the Hecke Operator $T_{\alpha}^{\vee}$} \label{subsec: fast double coset}

Algorithm \ref{algo: Hecke operator double coset} is still quite slow in general, since the conjugation function might be very slow, as it depends quadratically on the index. However, under some simplifying assumptions we may use the following algorithm for
the conjugation of an induced congruence subgroups.

Here, we assume that $(D(\alpha),N)=1$.

\begin{algo} \label{algo: Conjugation of induced}
$\operatorname{ConjInd}(G, \alpha)$. Conjugation of an induced congruence subgroup.
\end{algo}

\textbf{Input :} 
\begin{itemize}
\item $G\subseteq GL_{2}(\mathbb{Z}/N\mathbb{Z})$.
\item $\alpha\in GL_{2}^{+}(\mathbb{Q})$ such that $(D(\alpha),N)=1$. Write
$n:=D(\alpha)$. 
\end{itemize}

\textbf{Output :} a group $H\subseteq GL_{2}(\mathbb{Z}/nN\mathbb{Z})$
such that $\Gamma_{H}=\alpha^{-1}\Gamma_{G}\alpha\cap SL_2(\mathbb{Z})$. 
\begin{enumerate}
\item Find $x,y\in SL_{2}(\mathbb{Z})$  such
that $x\alpha y=\left(\begin{array}{cc}
d_1(\alpha) & 0\\
0 & n \cdot d_1(\alpha)
\end{array}\right)$.
\item Let $H_{N}:=\lambda_{N}(d_1(\alpha)^{-1}\alpha)^{-1} G\lambda_{N}(d_1(\alpha)^{-1}\alpha)\subseteq GL_{2}(\mathbb{Z}/N\mathbb{Z})$. 
\item Let $H_{n}:=\lambda_{n}(y) \lambda_{n}(\Gamma^{0}(n))\lambda_{n}(y)^{-1}\subseteq GL_{2}(\mathbb{Z}/n\mathbb{Z})$. 
\item Let $A_{N}\subseteq GL_{2}(\mathbb{Z}/nN\mathbb{Z})$ be a set of
matrices $a_{i}$ lifting $Generators(H_{N})$ such that $a_{i}\equiv1\bmod n$.
(Chinese Remainder Theorem)
\item Let $B_{n}\subseteq GL_{2}(\mathbb{Z}/nN\mathbb{Z})$ be a set of
matrices $b_{j}$ lifting $Generators(H_{n})$ such that $b_{j}\equiv1\bmod N$.
(Chinese Remainder Theorem) 
\item Return $H:=\langle A_{N}\cup B_{n} \rangle$. 
\end{enumerate}

\begin{prop}
Algorithm \ref{algo: Conjugation of induced} returns a group
$H$ such that $\Gamma_{H}=\alpha^{-1}\Gamma_{G}\alpha\cap SL_2(\mathbb{Z})$. 
\end{prop}

\begin{proof}
First note that replacing $\alpha$ by $d_1(\alpha)^{-1} \alpha$, we may assume $\alpha \in M_2(\mathbb{Z})$ and $d_1(\alpha) = 1$. 
Let $\Gamma':=\alpha^{-1}\Gamma_{G}\alpha\cap SL_2(\mathbb{Z})$. 
If $p\nmid n N$, then $\lambda_{p^{e}}(\Gamma')=SL_{2}(\mathbb{Z}/p^{e}\mathbb{Z})$
(for all $e$). We also have 
$$\lambda_{N}(\Gamma')=\lambda_{N}(\alpha)^{-1} G \lambda_{N}(\alpha)$$
and if $x\alpha y=\left(\begin{array}{cc}
1 & 0\\
0 & n
\end{array}\right)$, then 
\[
\lambda_{n}(\Gamma')=\lambda_{n}(y)\lambda_{n}\left(\begin{array}{cc}
1 & 0\\
0 & n
\end{array}\right)^{-1}\lambda_{n}(x^{-1}\Gamma_{G}x)\lambda_{n}\left(\begin{array}{cc}
1 & 0\\
0 & n
\end{array}\right)\lambda_{n}(y)^{-1}.
\]

But, since $(n,N)=1$ and $x^{-1} \Gamma_G x \supseteq \Gamma(N)$, 
we have $\lambda_{n}(x^{-1}\Gamma_{G}x)=SL_{2}(\mathbb{Z}/n\mathbb{Z})$,
and so we see that $\Gamma^{'}=\{g\in SL_{2}(\mathbb{Z})\mid\lambda_{n}(g)\in H_{n},\lambda_{N}(g)\in H_{N}\}$,
with $H_{n}$ and $H_{N}$ as in Algorithm \ref{algo: Conjugation of induced}
. Since $(n,N)=1$, we claim that $\Gamma'=\Gamma_{H}$, or equivalently
that $\lambda_{nN}(\Gamma')=H$. 

Indeed, if $g\in\lambda_{nN}(\Gamma')$, then $\lambda_{n}(g)\in H_{n}$
and $\lambda_{N}(g)\in H_{N}$. Therefore, there exist $\{x_{i}\}\subseteq Generators(H_{n})$
and $\{y_{j}\}\subseteq Generators(H_{N})$ such that $\lambda_{n}(g)=\prod x_{i}^{e_{i}}$
and $\lambda_{N}(g)=\prod y_{j}^{f_{j}}$. Let $\{a_{i}\}$ be the
matrices in $A_{n}$ lifting the $\{x_{i}\}$ and $\{b_{j}\}$ the
matrices in $B_{N}$ lifting the $\{y_{j}\}$. Then by the Chinese
Remainder Theorem, we have $g=\prod a_{i}^{e_{i}}\cdot\prod b_{j}^{f_{j}}$. 
\end{proof}

Next, we would like to understand the running time of Algorithm \ref{algo: Hecke operator double coset}, with this improvement.
\begin{cor}\label{cor:complexity-double-coset-operator}
Let $\alpha\in GL_{2}^{+}(\mathbb{Q})$ be such that $(D(\alpha),N)=1$.
Then Algorithm \ref{algo: Hecke operator double coset}
has complexity of 
\[
O(I_{\alpha, \Gamma} \cdot \log(N^{4}\cdot D(\alpha)))
\]
basic $\operatorname{CosetIndex}$ operations. 
\end{cor}

\begin{proof}
Using algorithms \ref{algo: Conjugation of induced} and \ref{algo: Intersection of induced} in step $(1)$, 
we see that the main contribution to the complexity of algorithm \ref{algo: Hecke operator double coset}
comes from step $(4)$, and  this was computed in the proof of Theorem \ref{thm: Hecke operator double coset naive}.
\end{proof}

\subsection{The Hecke Operators $T_{n}$} \label{subsec: Hecke operator at p}

The question of computing the Hecke operators $T_{n}$, when $(n,N)>1$, is
a difficult one. We have not been able to find in the literature a
good reference even for the definition of these operators. Therefore,
we restrict ourselves, for the time being, to $n$ such that $(n,N)=1$. 

\subsubsection{Adelic Definition}

The only definition given in the literature of the Hecke operator
for completely general level structure is in adelic terms. 

Therefore, before we state the definition, we need to set up some
notation.
\begin{defn}
Let $G\subseteq GL_{2}(\mathbb{Z}/N\mathbb{Z})$. We introduce
\[
S_{G}:=\left\{ r\cdot g\cdot s\mid r\in\mathbb{Q}^{\times},g\in GL_{2}(\hat{\mathbb{Z}}),s\in GL_{2}^{+}(\mathbb{R}),g\bmod N\in G\right\} \subseteq GL_{2}(\mathbb{A})
\]
where we use the identification $\hat{\mathbb{Z}}/N\hat{\mathbb{Z}}\cong\mathbb{Z}/N\mathbb{Z}$.
\end{defn}

Now we may define the Hecke correspondences on the corresponding Shimura variety. 
\begin{defn}
\label{def:Hecke Operator Adelic}(\cite[Section 7.3]{shimura1971introduction}) Let $G\subseteq GL_{2}(\mathbb{Z}/N\mathbb{Z})$. Let
$X_{G}:=GL_{2}^{+}(\mathbb{Q})\backslash(\mathcal{H}\times GL_{2}(\hat{\mathbb{Q}}))/S_{G}$
be the Shimura variety of level $S_{G}$. Let $w\in GL_{2}(\hat{\mathbb{Q}})$.
Let $W=S_{G}\cap wS_{G}w^{-1}$. Then $w$ induces a natural correspondence
$X_{G}\rightarrow X_{G}$ via

\[
[z,g]\mapsto\sum_{\alpha\in S_{G}/W}[z,g\alpha w]
\]
We denote this correspondence by $X(w)$. 
\end{defn}

\begin{lem}
The correspondence $X(w)$, for $w\in GL_{2}(\hat{\mathbb{Q}})$, is
well defined.
\end{lem}

\begin{proof}
First, the choice of representatives for $S_{G}/W$ does not matter: if we replace each $\alpha$
by $\alpha w_{\alpha}$, we would get 
\[
\sum_{\alpha\in S_{G}/W}[z,g\alpha w_{\alpha}w]=\sum_{\alpha\in S_{G}/W}[z,g\alpha w\cdot w^{-1}w_{\alpha}w]=\sum_{\alpha\in S_{G}/W}[z,g\alpha w]
\]
because $w^{-1}Ww\subseteq S_{G}$. If $\gamma \in GL_{2}^{+}(\mathbb{Q})$, then $X(w)[z,g]=X(w)[\gamma z, \gamma g]$.
Also, if we consider an element $[z,gs]$, with $s\in S_{G}$, it
will map to 
\[
\sum_{\alpha\in S_{G}/W}[z,gs\alpha w]=\sum_{\alpha\in S_{G}/W}[z,g\alpha w]
\]
since $s\alpha$ still runs through a set of representatives of $S_{G}/W$. 
\end{proof}

We now connect it to our previous definitions of moduli spaces
via the following Lemma.
\begin{lem}
(\cite[Lemma 5.13, Theorem 5.17]{milne2005introduction}) Let $G\subseteq GL_{2}(\mathbb{Z}/N\mathbb{Z})$.
Let $\mathcal{C}$ be a set of representatives for 
\[
GL_{2}^{+}(\mathbb{Q})\backslash GL_{2}(\hat{\mathbb{Q}})/S_{G}\cong\hat{\mathbb{Z}}/\det(S_{G})\cong\left(\mathbb{Z}/N\mathbb{Z}\right)^{\times}/\det(G)
\]
Then 
\[
GL_{2}^{+}(\mathbb{Q})\backslash\left(\mathcal{H}\times G(\hat{\mathbb{Q}})\right)/S_{G}\cong\bigsqcup_{g\in\mathcal{C}}\Gamma_{g}\backslash\mathcal{H}
\]
where $\Gamma_{g}=gS_Gg^{-1}\cap GL_{2}^{+}(\mathbb{Q})$. 
\end{lem}

Under this identification, if we denote by $\omega$ the canonical line bundle on $X_G$, then we may identify $M_k(\Gamma_G)$ with the space of global sections $H^{0}(X_{G},\omega^{\otimes k})$. 
The correspondence $X(w)$ induces (by pullback) a map on global sections of line bundles
$T(w):H^{0}(X_{G},\omega^{\otimes k})\rightarrow H^{0}(X_{G},\omega^{\otimes k})$
given by $f\mapsto\sum_{\alpha\in S/W}(\alpha w)^{*}f$ . We then get the following corollary.

\begin{cor}
The correspondences $X(w)$ may be identified  via this isomorphism
(explicitly given as $z\mapsto[z,1]$) as correspondences 
$\Gamma_g \backslash \mathcal{H} \rightarrow \Gamma_{gw} \backslash \mathcal{H}$.
In particular, if $\det(w) \in \det(S_G)$,
then we may identify the correspondence $X(w)$  as correspondences on $\Gamma_{G}\backslash\mathcal{H}$,
and the operator $T(w)$ as an operator on $M_{k}(\Gamma_{G})$. 
\end{cor}

The reason for using Definition \ref{def:Hecke Operator Adelic} is
the following important theorem.
\begin{thm}
(Shimura, \cite[Theorem 7.9]{shimura1971introduction}) Let $w\in GL_{2}(\hat{\mathbb{Q}})$
be such that $w_{l}=1$ for all $l\ne p$, and $w_{p}\in M_{2}(\mathbb{Z}_{p})$
is an element such that $\det(w_{p})=p$. Then if $p\nmid N$, the
correspondence $X(w)$ on $X_{G}$ is rational over the field of definition
of $X_{G}$ and satisfies $\tilde{X}(w)\equiv Fr_{p}+^{t}Fr_{p}^{p}\circ\widetilde{[\det(w)^{-1}]}$,
where $X\mapsto\tilde{X}$ is the reduction modulo $p$ and $Fr_{p}$
is the Frobenius correspondence on $\tilde{X}_{G}\times\tilde{X}_{G}^{p}$. 
\end{thm}

This, in turn, relates these Hecke operators to the zeta function of $X_G$. 
Before stating the connection, we introduce some notations.

Let $Z_p(s;X_G / \mathbb{Q})$ denote the local factor at $p$ of the zeta function of $X_G$.
Let $w \in GL_2(\hat{\mathbb{Q}})$ be such that $w_l = 1$ for all $l \ne p$ and $w_p \in M_2(\mathbb{Z}_p)$ is such that $\det(w_p) = p$. 
Let $\sigma \in SL_{2}(\mathbb{Z})$ be such that
$\lambda_{N}(\sigma)=p\cdot\lambda_{N}(w)^{-2}$. 
Via the diagonal embedding $SL_2(\mathbb{Z}) \hookrightarrow GL_2(\hat{\QQ})$, the correspondence $X(\sigma)$ is well defined and as 
$\det(\sigma)~=~1~\in~\det(S_G)$, it induces an operator $T(\sigma)$ on $M_k(\Gamma_G)$. 

\begin{cor} \label{cor: zeta function}
Let $w,\sigma$ be as above. Let $T(w)$, $T(\sigma)$ be the corresponding operators on $M_2(\Gamma_G)$. Then
\begin{equation}
(1-p^{-s})(1-p^{1-2s}) Z_{p}(s;X_{G}/\mathbb{Q})= \det\left(1-T(w)\cdot p^{-s}+T(\sigma)\cdot p^{1-2s}\right). \label{eq:Zeta function}
\end{equation}
\end{cor}

\begin{proof}
This is proved as in \cite{shimura1971introduction}, Corollary 7.10 and Theorem 7.11. 
\end{proof}

Although this, in general, does not yield an application to $q$-expansions, there are some special cases in which it does. We describe these cases, as we shall want to make use of them in some of our applications.

Recall that the Hecke operators commute, hence have common eigenvalues. These are called eigenforms.

\begin{defn}
An element $f \in S_k(\Gamma_G)$ is called an \emph{eigenform} if it is an eigenvector for $T_{\alpha}$ for all $\alpha \in GL_2^+(\mathbb{Q})$.
An eigenform $f = \sum_{n=1}^{\infty} a_n q^n $ is \emph{normalized} if  $a_1 = 1$. 
\end{defn}

\begin{cor}
Let $\Gamma_G = \Gamma(\mathfrak{h}, t)$, and
let $f = \sum_{n=1}^{\infty} a_n q^n \in S_k(\Gamma_G)$ be a normalized eigenform. Let $w \in GL_2(\hat{\mathbb{Q}})$ be such that $w_l = 1$ for all $l \ne p$ and $w_p \in M_2(\mathbb{Z}_p)$ is such that $\det(w_p) = p$. Then $T(w) f = a_p f$. 
\end{cor}

\begin{proof}
This is proved in  \cite[Section 3.5]{shimura1971introduction}. 
\end{proof}

Yet as another consequence, we get the following fact, which could also be shown directly by double coset computation.
\begin{cor}
If $p\nmid N$, and $w\in GL_{2}(\hat{\mathbb{Q}})$ is such that $w_{l}=1$
for all $l\ne p$, and $w_{p}\in M_{2}(\mathbb{Z}_{p})$ is an element
such that $\det(w_{p})=p$, then the operator $T(w)$ is independent
of the choice of $w$.
\end{cor}

This motivates the following definition. 

\begin{defn}
If $p\nmid N$, and $w\in GL_{2}(\hat{\mathbb{Q}})$ is such that $w_{l}=1$
for all $l\ne p$, and $w_{p}\in M_{2}(\mathbb{Z}_{p})$ is an element
such that $\det(w_{p})=p$, we define $T_{p}:=T(w)$. 
\end{defn}

Next, we would like to write this operator in terms of a double coset
operator, as in Definition \ref{def:Hecke Operator Double Coset}.
In order to do that, we show first that the subgroup $W$ in Definition
\ref{def:Hecke Operator Adelic} is also induced from a subgroup of
$GL_{2}(\mathbb{Z}/N\mathbb{Z})$. 
\begin{lem}
\label{lem:Conjugation of induced level}Let $G\subseteq GL_{2}(\mathbb{Z}/N\mathbb{Z})$.
Let $p\nmid N$. Let $w\in GL_{2}(\hat{\mathbb{Q}})$ be such that $w_{l}=1$
for all $l\ne p$, and $w_{p}\in M_{2}(\mathbb{Z}_{p})$ is an element
such that $\det(w_{p})=p$. Let $W=S_{G}\cap wS_{G}w^{-1}$. Then
one has 
\[
p\cdot M_{2}(\mathbb{Z}_{p})\subseteq M_{2}(\mathbb{Z}_{p})\cap w_{p}M_{2}(\mathbb{Z}_{p})w_{p}^{-1}
\]
Let $R_{p}$ be the image of $M_{2}(\mathbb{Z}_{p})\cap w_{p}M_{2}(\mathbb{Z}_{p})w_{p}^{-1}$
in $M_{2}(\mathbb{Z}/p\mathbb{Z})$, and let $G_{p}:=R_{p}^{\times}$.
Set 
\[
G_{w}:=G\times G_{p}\subseteq GL_{2}(\mathbb{Z}/(Np)\mathbb{Z})
\]
Then $W=S_{G_{w}}$. 
\end{lem}

\begin{proof}
First, let us show that 
\[
p\cdot M_{2}(\mathbb{Z}_{p})\subseteq M_{2}(\mathbb{Z}_{p})\cap w_{p}M_{2}(\mathbb{Z}_{p})w_{p}^{-1}
\]
Indeed, it is enough to prove that $p\cdot M_{2}(\mathbb{Z}_{p})\subseteq w_{p}M_{2}(\mathbb{Z}_{p})w_{p}^{-1}$,
which is equivalent to 
\[
w_{p}^{-1}p\cdot M_{2}(\mathbb{Z}_{p})\cdot w_{p}\subseteq M_{2}(\mathbb{Z}_{p})
\]
However, as $\det(w_{p})=p$, we see that $w_{p}^{-1}\cdot p=\tilde{w}_{p}\in M_{2}(\mathbb{Z}_{p})$,
thus the claim is trivial.

Now, let $x\in W$. Then, as $x\in S_{G}$, we may write $x=r\cdot g\cdot s$
with $r\in\mathbb{Q}^{\times},$ $g\in GL_{2}(\hat{\mathbb{Z}})$,
$s\in GL_{2}^{+}(\mathbb{R})$, and $g\bmod N\in G$.

Write $g=g^{p}g_{p}$, with $g^{p}\in\prod_{l\ne p}GL_{2}(\mathbb{Z}_{l})$
and $g_{p}\in GL_{2}(\mathbb{Z}_{p})$. Recall also that $w_l = 1$ for all $l \ne p$.
Since $x\in wS_{G}w^{-1}$,
we have
\[
r\cdot g^{p}\cdot w_{p}^{-1}g_{p}w_{p}\cdot s = r \cdot w^{-1} g w \cdot s = w^{-1}xw\in S_{G}
\]
Thus, there exists $r'\in\mathbb{Q}^{\times}$ such that 
\[
r'\cdot g^{p}\cdot w_{p}^{-1}g_{p}w_{p}\in GL_{2}(\hat{\mathbb{Z}})
\]
In particular, for all $l\ne p$, we have $r'\cdot g_{l}\in GL_{2}(\mathbb{Z}_{l})$,
and as $g_{l}\in GL_{2}(\mathbb{Z}_{l})$, we see that $r'\in\mathbb{Z}_{l}^{\times}$.
Also, as 
\[
r'\cdot w_{p}^{-1}g_{p}w_{p}\in GL_{2}(\mathbb{Z}_{p})
\]
and $g_{p}\in GL_{2}(\mathbb{Z}_{p})$, we see that $r'\in\{\pm1\}$,
and $w_{p}^{-1}g_{p}w_{p}\in GL_{2}(\mathbb{Z}_{p})$. 

Therefore, we obtain
\[
g_{p}\in w_{p}GL_{2}(\mathbb{Z}_{p})w_{p}^{-1}\subseteq w_{p}M_{2}(\mathbb{Z}_{p})w_{p}^{-1}
\]
As we already know $g_{p}\in M_{2}(\mathbb{Z}_{p})$, we see that
$g_{p}\bmod p\in R_{p}$. 

Since $g_{p}\in GL_{2}(\mathbb{Z}_{p})$, it follows that $g\bmod p=g_{p}\bmod p\in G_{p}$,
and from $g\bmod N\in G$, it follows that $g\bmod Np\in G\times G_{p}=G_{w}$.
Thus we have shown that $x\in S_{G_{W}}$, so that $W\subseteq S_{G_{w}}$. 

Conversely, assume $x\in S_{G_{w}}$. Then $x=r\cdot g\cdot s$ with
$r\in\mathbb{Q}^{\times},$ $g\in GL_{2}(\hat{\mathbb{Z}})$, $s\in GL_{2}^{+}(\mathbb{R})$,
and $g\bmod Np\in G_{w}$.

It follows that $g\bmod N\in G$ and $g\bmod p\in G_{p}$. In particular,
we see immediately that $x\in S_{G}$. Moreover, 
\[
w^{-1}xw=r\cdot g^{p}\cdot w_{p}^{-1}g_{p}w_{p}\cdot s
\]
Since $g_{p}\bmod p=g\bmod p\in G_{p}\subseteq R_{p}$, we have 
\[
g_{p}\in M_{2}(\mathbb{Z}_{p})\cap w_{p}M_{2}(\mathbb{Z}_{p})w_{p}^{-1}
\]
hence $w_{p}^{-1}g_{p}w_{p}\in M_{2}(\mathbb{Z}_{p})$. Looking at
the determinant, we see that it actually lies in $GL_{2}(\mathbb{Z}_{p})$.
In particular, 
\[
g^{'}=g^{p}\cdot w_{p}^{-1}g_{p}w_{p}\in GL_{2}(\hat{\mathbb{Z}})
\]
 and $g'\bmod N\in G$. Therefore, $w^{-1}xw\in S_{G}$, showing equality.
\end{proof}
We now show how to use strong approximation to rewrite the operator
$T(w)$ in the classical description of modular curves. 
\begin{lem}
\label{lem:Adelic to Classical}Let $G\subseteq GL_{2}(\mathbb{Z}/N\mathbb{Z})$
be a subgroup. Let $p \in \det(G)$ be a prime.
Let $w\in GL_{2}(\hat{\mathbb{Q}})$
be such that $w_{l}=1$ for all $l\ne p$, and $w_{p}\in M_{2}(\mathbb{Z}_{p})$
is such that $\det(w_{p})=p$.
Then the correspondence $X(w)$ can be realized as 
\[
\Gamma_{G}z\mapsto\sum_{\alpha\in\Gamma_{G}/\Gamma_{G_{w}}}\Gamma_{G}\cdot q(w)\cdot\alpha^{-1}z
\]
for some $q(w)\in GL_{2}^{+}(\mathbb{Q})$, explicitly constructed
as a function of $w$. Therefore $T_{p}$ is realized as $f\mapsto\sum_{\alpha\in\Gamma_{G}/\Gamma_{G_{w}}}f\vert_{q(w)\cdot\alpha^{-1}}$. 
\end{lem}

\begin{proof}
Let $\delta_p \in G$ be an element such that $\det(\delta_p) = p$. Let $w_{p}\in M_{2}(\mathbb{Z})$
be such that $\det(w_{p})=p$, and let $\alpha\in\Gamma_{G}$. Then 
\[
\det(\delta_p^{-1}\cdot\lambda_{N}(w_{p}))=\det(\delta(p))^{-1}\cdot\det(w_{p})=p^{-1}\cdot p=1
\]
so that $\delta_p^{-1}\cdot\lambda_{N}(w_{p})\in SL_{2}(\mathbb{Z}/N\mathbb{Z})$.
Since $\lambda_{N}:SL_{2}(\mathbb{Z})\rightarrow SL_{2}(\mathbb{Z}/N\mathbb{Z})$
is surjective, there exists $\beta_{p}\in SL_{2}(\mathbb{Z})$ such
that 
\[
\lambda_{N}(\beta_{p})=\delta_p^{-1}\cdot\lambda_{N}(w_{p})
\]
Let $q(w)=\beta_{p}\cdot pw_{p}^{-1}\in GL_{2}^{+}(\mathbb{Q})$ .
Then we have 
\[
\alpha\cdot q(w)^{-1}=p^{-1}\cdot\alpha w_{p}\beta_{p}^{-1}=p^{-1}\cdot\alpha w\cdot w^{-1}w_{p}\beta_{p}^{-1}
\]
Let $s=w^{-1}w_{p}\beta_{p}^{-1}$. Then $s_{p}=w_{p}^{-1}w_{p}\beta_{p}^{-1}=\beta_{p}^{-1}\in SL_{2}(\mathbb{Z})\subseteq GL_{2}(\mathbb{Z}_{p})$.
Also, for all $l\ne p$, one has $s_{l}=w_{p}\cdot\beta_{p}^{-1}$,
so that $\det(s_{l})=\det(w_{p})\cdot\det(\beta_{p})^{-1}=p$ and
$s_{p}\in M_{2}(\mathbb{Z})\subseteq M_{2}(\mathbb{Z}_{l})$, showing
that $s_{l}\in GL_{2}(\mathbb{Z}_{l})$. Moreover, 
\[
\lambda_{N}(s)=\lambda_{N}(w_{p})\cdot\lambda_{N}(\beta_{p})^{-1}=\delta_p\in G
\]
showing that $s\in S_{G}$. Therefore $\alpha w=\alpha\cdot q(w)^{-1}\cdot s^{-1}$,
so that (recall that $p\in\mathbb{Q}^{\times}$ acts trivially)
\[
[z,\alpha w]=[z,p^{-1}\cdot\alpha w\cdot s]=[z,\alpha\cdot q(w)^{-1}]=[q(w)\cdot\alpha^{-1}\cdot z,1]
\]
By Lemma \ref{lem:Conjugation of induced level}, $W=S_{G_{W}}$,
hence
\[
S_{G}/W\cong G\times GL_{2}(\mathbb{Z}/p\mathbb{Z})/G_{w}\cong\Gamma_{G}/\Gamma_{G_{w}}
\]
It follows that under the isomorphism $[z]\mapsto[z,1]$, the correspondence
$X(w)$ above is interpreted as 
\[
\Gamma_{G}z\mapsto\sum_{\alpha\in\Gamma_{G}/\Gamma_{G_{w}}}\Gamma_{G}\cdot q(w)\cdot\alpha^{-1}z. \qedhere
\]
\end{proof}

\subsubsection{Classical Definition}

We can now use Lemma \ref{lem:Adelic to Classical} to find an equivalent
definition in classical terms of the Hecke operator $T_{p}$. 
\begin{lem}
Let $\alpha\in M_{2}(\mathbb{Z})$ be such that $\det(\alpha)=p$
and $\lambda_{N}(\alpha)\in G$. Then $T_{\alpha}=T_{p}$. 
\end{lem}

\begin{proof}
Let $w\in GL_{2}(\hat{\mathbb{Q}})$ be such that $w_{l}=1$ for all
$l\ne p$, $w_{p}\in M_{2}(\mathbb{Z}_{p})$ and $\det(w_{p})=p$.
By Lemma \ref{lem:Adelic to Classical}, the operator $T_{p}=T(w)$
is given by $f\mapsto\sum_{\beta\in\Gamma_{G}/\Gamma_{G_{w}}}f\vert_{q(w)\cdot\beta^{-1}}$
, where $q(w)\in GL_{2}^{+}(\mathbb{Q})$. By Lemma \ref{lem:Conjugation of induced level},
we know that $G_{w}=G\times G_{p}$, with $G_{p}=R_{p}^{\times}$
where $R_{p}$ is the image in $M_{2}(\mathbb{Z}/p\mathbb{Z})$ of
$M_{2}(\mathbb{Z}_{p})\cap w_{p}M_{2}(\mathbb{Z}_{p})w_{p}^{-1}$,
therefore
\[
\Gamma_{G}/\Gamma_{G_{w}}\cong\left(G\times GL_{2}(\mathbb{Z}/p\mathbb{Z})\right)/G_{w}\cong GL_{2}(\mathbb{Z}/p\mathbb{Z})/G_{p}
\]

Recall also that $\det(w_{p})=p$, hence $w_{p}\in SL_{2}(\mathbb{Z}_{p})\cdot\left(\begin{array}{cc}
p & 0\\
0 & 1
\end{array}\right)\cdot SL_{2}(\mathbb{Z}_{p})$. 

Write $w_{p}=\gamma_{1}\cdot\left(\begin{array}{cc}
p & 0\\
0 & 1
\end{array}\right)\cdot\gamma_{2}$ , for some $\gamma_{1},\gamma_{2}\in SL_{2}(\mathbb{Z}_{p})$. Then
\begin{align*}
w_{p}M_{2}(\mathbb{Z}_{p})w_{p}^{-1}&=\gamma_{1}\cdot\left(\begin{array}{cc}
p & 0\\
0 & 1
\end{array}\right)\cdot\gamma_{2}M_{2}(\mathbb{Z}_{p})\gamma_{2}^{-1}\cdot\left(\begin{array}{cc}
p & 0\\
0 & 1
\end{array}\right)^{-1}\cdot\gamma_{1}^{-1}
\\
&=\gamma_{1}\cdot\left(\begin{array}{cc}
p & 0\\
0 & 1
\end{array}\right)\cdot M_{2}(\mathbb{Z}_{p})\cdot\left(\begin{array}{cc}
p & 0\\
0 & 1
\end{array}\right)^{-1}\cdot\gamma_{1}^{-1}
\end{align*}
hence 
\[
M_{2}(\mathbb{Z}_{p})\cap w_{p}M_{2}(\mathbb{Z}_{p})w_{p}^{-1}=\gamma_{1}\cdot\Delta^{0}(p)\cdot\gamma_{1}^{-1}
\]
where $\Delta^{0}(p)=\left\{ \left(\begin{array}{cc}
a & b\\
c & d
\end{array}\right)\in M_{2}(\mathbb{Z}_{p})\mid b\in p\mathbb{Z}_{p}\right\} $, and passing to the image in $M_{2}(\mathbb{Z}/p\mathbb{Z})$, we
get
\[
G_{p}=R_{p}^{\times}=\overline{\gamma}_{1}\cdot\Gamma^{0}(\mathbb{Z}/p\mathbb{Z})\cdot\overline{\gamma}_{1}^{-1}
\]
where $\Gamma^{0}(\mathbb{Z}/p\mathbb{Z})$ is the Borel subgroup
of lower triangular matrices, and $\overline{\gamma}_{1}\in SL_{2}(\mathbb{Z}/p\mathbb{Z})$. 

Also, if $\beta_{p}\in SL_{2}(\mathbb{Z})$ is such that $\lambda_{N}(\beta_{p})=\lambda_{N}(p)^{-1}\cdot\lambda_{N}(\alpha)\cdot\lambda_{N}(w_{p})$,
then $\lambda_{N}(\beta_{p}\cdot pw_{p}^{-1})=\lambda_{N}(\alpha)$
and $\alpha^{-1}\beta_{p}\cdot p\cdot w_{p}^{-1}\in\Gamma(N)\subseteq\Gamma$,
so that $\Gamma\beta_{p}\cdot pw_{p}^{-1}=\Gamma\alpha$,
thus $\Gamma\beta_{p}\cdot pw_{p}^{-1}\Gamma=\Gamma\alpha\Gamma$,
and so (as $T_{\alpha}(f)=\sum_{\delta\in\Gamma\backslash\Gamma\alpha\Gamma}f\vert_{\delta}$),
we may assume that $\beta_{p}\cdot pw_{p}^{-1}=\alpha$, hence $q(w)=\alpha$. 

Note then that, as $p\cdot\alpha^{-1}=\gamma_{1}\cdot\left(\begin{array}{cc}
p & 0\\
0 & 1
\end{array}\right)\cdot\gamma_{2}\beta_{p}^{-1}$, we see that
\begin{align*}
\lambda_{p}(\alpha^{-1}\Gamma\alpha) 
&= \lambda_{p}(\alpha^{-1}\cdot M_{2}(\mathbb{Z})\cdot\alpha)^{\times}=\lambda_{p}(\gamma_{1}\cdot\Delta^{0}(p)\cdot\gamma_{1}^{-1})^{\times} \\
&=\lambda_{p}(\gamma_{1}\cdot\Gamma^{0}(p)\cdot\gamma_{1}^{-1})=\overline{\gamma}_{1}\cdot\Gamma^{0}(\mathbb{Z}/p\mathbb{Z})\cdot\overline{\gamma}_{1}^{-1}=G_{p}.
\end{align*}

Therefore
\[
\Gamma/(\Gamma\cap\alpha^{-1}\Gamma\alpha)\cong GL_{2}(\mathbb{Z}/p\mathbb{Z})/G_{p}\cong\Gamma_{G}/\Gamma_{G_{w}}
\]

But we have bijection 
\begin{align*}
\Gamma/(\Gamma\cap\alpha^{-1}\Gamma\alpha) & \rightarrow(\Gamma\cap\alpha^{-1}\Gamma\alpha)\backslash\Gamma\rightarrow\Gamma\backslash\Gamma\alpha\Gamma\\
\beta & \mapsto\qquad\beta^{-1}\qquad\mapsto\alpha\cdot\beta^{-1}
\end{align*}

Thus, the map $\beta\mapsto\alpha\cdot\beta^{-1}:\Gamma_{G}/\Gamma_{G_{w}}\rightarrow\Gamma\backslash\Gamma\alpha\Gamma$
is a bijection, so that 
\[
T_{\alpha}(f)=\sum_{\delta\in\Gamma\backslash\Gamma\alpha\Gamma}f\vert_{\delta}=\sum_{\beta\in\Gamma_{G}/\Gamma_{G_{w}}}f\vert_{\alpha\cdot\beta^{-1}}=\sum_{\beta\in\Gamma_{G}/\Gamma_{G_{w}}}f\vert_{q(w)\cdot\beta^{-1}}=T(w)(f)=T_p(f)
\]
establishing the result. 
\end{proof}

\begin{cor} \label{cor: coefficients of eigenforms}
Assume $p\nmid N$. Let $\alpha\in M_{2}(\mathbb{Z})$ be such that
$\det(\alpha)=p$ and $\lambda_{N}(\alpha)\in G$. Then $T_{\alpha}$
is independent of $\alpha$. Moreover, if $\Gamma_G = \Gamma(\mathfrak{h},t)$ and $f = \sum_{n=1}^{\infty} a_n q^n$ is an eigenform of the Hecke algebra then $T_{\alpha} f = a_p f$. 
\end{cor}

This allows us to define the Hecke operators at primes not dividing
the level (and by multiplicativity to all $n$ such that $(n,N)=1$).
\begin{defn}
\label{def:Hecke Operator Classic Definition}Let $n$ be such that
$(n,N)=1$ and $n \in \det(G)$. Let $\alpha\in M_{2}(\mathbb{Z})$ be such that $\det(\alpha)=n$
and $\lambda_{N}(\alpha)\in G$. Let $T_{n}:=T_{\alpha}$ . 
\end{defn}

Therefore, Algorithm \ref{algo: Hecke operator double coset}
can be used to compute the Hecke operators $T_{n}$, when $n$ is coprime to $N$. 
\begin{rem}
Note that the definition of the Hecke operators $\{T_{n}\}$, although
independent of the choice of representatives, does depend on $G$,
and not only on $\Gamma_{G}$. This is not just an artifact of the
proof, as the following example shows.
\end{rem}

\begin{example}
Let 
\[
G_{1}=\left\{ \pm\left(\begin{array}{cc}
1 & 0\\
0 & d
\end{array}\right)\in GL_{2}(\mathbb{Z}/7\mathbb{Z})\mid d\in(\mathbb{Z}/7\mathbb{Z})^{\times}\right\} .
\]
Let $G_{2}\subseteq GL_{2}(\mathbb{Z}/7\mathbb{Z})$ be the subgroup
\[
G_{2}:=\left\langle \pm \left(\begin{array}{cc}
1 & 4\\
2 & 6
\end{array}\right) \right\rangle \subseteq GL_{2}(\mathbb{Z}/7\mathbb{Z})
\]
which is an abelian group isomorphic to $\mathbb{Z}/6\mathbb{Z} \oplus \mathbb{Z} /2 \mathbb{Z}$.

Then $\Gamma_{G_{1}}=\Gamma_{G_{2}}=\Gamma(7)\cdot\{\pm1\}$. But
while $T_{2}^{G_{1}}$ on $S_{2}(\Gamma(7))$ is the familiar Hecke
operator $T_{\alpha}$ for $\alpha=\left(\begin{array}{cc}
1 & 0\\
0 & 2
\end{array}\right)$, $T_{2}^{G_{2}}$ on $S_{2}(\Gamma(7))$ is simply the identity, since $\alpha \notin G_2$. 
\end{example}

\begin{rem}
As in \cite[Proposition 3.31]{shimura1971introduction}, one can show that the map 
$$
\Gamma \alpha \Gamma \mapsto SL_2(\mathbb{Z}) \alpha SL_2(\mathbb{Z}),
$$
for all $\alpha \in M_2(\mathbb{Z})$ such that $\lambda_N(\alpha) \in \mathfrak{N}(G)$, defines a homomorphism on the Hecke algebras. 

However, as the above example shows, this homomorphism is not injective, and elements $\alpha$ in different cosets of $G \backslash \mathfrak{N}(G)$ give rise to different operators.
\end{rem}

\subsection{Efficient Implementation of the Hecke Operators $T_{n}^{\vee}$,
$n \in \det(G)$ } \label{subsec: Hecke efficient Merel}

Algorithm \ref{algo: Hecke operator double coset} is
not as efficient as we would have liked. Specifically, the logarithmic
factor obtained from the passage to modular symbols is, in practice, due to the constant, 
a particularly high cost. We therefore use the ideas of Merel in \cite{merel1994universal}
to calculate, at least the operators $\{T_{n}\}$ more efficiently. 
Following Section 2.1 in \cite{merel1994universal}, we introduce
the definition of a Merel pair. 
\begin{defn}
\label{def:Merel pair}Let $\Delta\subseteq GL_{2}(\mathbb{Q})$ be
such that $\Gamma\Delta=\Delta\Gamma$ and such that $\Gamma\backslash\Delta$
is finite. Let $\tilde{\Delta}=\{g\in GL_{2}(\mathbb{Q})\mid\tilde{g}:=g^{-1}\det(g)\in\Delta\}$.
Let $\phi:\tilde{\Delta}\cdot SL_{2}(\mathbb{Z})\rightarrow SL_{2}(\mathbb{Z})$
be a map such that 

1. For all $\gamma\in\tilde{\Delta}\cdot SL_{2}(\mathbb{Z})$ and
$g\in SL_{2}(\mathbb{Z})$ we have $\Gamma\cdot\phi(\gamma g)=\Gamma\cdot\phi(\gamma)\cdot g$.

2. For all $\gamma\in\tilde{\Delta}\cdot SL_{2}(\mathbb{Z})$, we
have $\gamma\cdot\phi(\gamma)^{-1}\in\tilde{\Delta}$ (or equivalently
$\phi(\gamma)\cdot\tilde{\gamma}\in\Delta$).

3. The map $\Gamma\backslash\Delta\rightarrow\tilde{\Delta}\cdot SL_{2}(\mathbb{Z})/SL_{2}(\mathbb{Z})$
which associates to $\Gamma\delta$ the element $\tilde{\delta}SL_{2}(\mathbb{Z})$
is injective. (it is necessarily surjective).

We say that the pair $(\Delta,\phi)$ is a \emph{Merel pair} for $\Gamma$.
\end{defn}

We can now state Merel's condition $(C_{\Delta})$:
\begin{defn}
Let $\sum u_{M}M\in\mathbb{C}[M_{2}(\mathbb{Z})]$, and let $(\Delta,\phi)$
be a Merel pair for $\Gamma$. We will say that $\sum u_{M}M$ satisfies
the condition $(C_{\Delta})$ if and only if for all $K\in\tilde{\Delta}SL_{2}(\mathbb{Z})/SL_{2}(\mathbb{Z})$
we have the following equality in $\mathbb{C}[\mathbb{P}^{1}(\mathbb{Q})]$:
\[
\sum_{M\in K}u_{M}([M\infty]-[M0])=[\infty]-[0].
\]
\end{defn}

Finally, we recall the following extremely useful theorem.
\begin{thm}
[{\cite[Theorem 4]{merel1994universal}}]Let $P\in\mathbb{C}_{k-2}[X,Y]$
and $g\in SL_{2}(\mathbb{Z})$. Let $(\Delta,\phi)$ be a Merel pair
for $\Gamma$. Let $\sum u_{M}M\in\mathbb{C}[M_{2}(\mathbb{Z})]$
satisfy the condition $(C_{\Delta})$. We have in $\mathbb{M}_{k}(\Gamma)$
\[
T_{\Delta}^{\vee}([P,g])=\sum_{M,gM\in\tilde{\Delta}SL_{2}(\mathbb{Z})}u_{M}[P\vert_{\tilde{M}},\phi(gM)]
\]
\end{thm}

We also recall Merel's condition $(C_{n})$. 
\begin{defn}
\label{def:Condition C_n}Denote by $M_{2}(\mathbb{Z})_{n}$ the set
of matrices of $M_{2}(\mathbb{Z})$ of determinant $n$. We say that
an element $\sum_{M}u_{M}M\in\mathbb{C}[M_{2}(\mathbb{Z})_{n}]$ satisfies
the condition $(C_{n})$ if for all $K\in M_{2}(\mathbb{Z})_{n}/SL_{2}(\mathbb{Z})$,
we have in $\mathbb{C}[\mathbb{P}^{1}(\mathbb{Q})]$
\[
\sum_{M\in K}u_{M}([M\infty]-[M0])=[\infty]-[0].
\]
\end{defn}

\begin{cor}
\label{cor:Merel Universal Fourier Expansion}Let $(\Delta_{n},\phi_{n})$
be a Merel pair for $\Gamma$ with $\Delta_{n}\subseteq M_{2}(\mathbb{Z})_{n}$
. Let $\sum_{M}u_{M}M$ satisfy the condition $(C_{n})$, then in
$\mathbb{M}_{k}(\Gamma)$
\[
T_{\Delta_{n}}^{\vee}([P,g])=\sum_{M}u_{M}[P\vert_{\tilde{M}},\phi_{n}(gM)]
\]
where the sum is resricted to matrices $M$ such that $gM\in\tilde{\Delta}_{n}SL_{2}(\mathbb{Z})$. 
\end{cor}

\begin{example} 
\label{exa:Merel pair nonsplit Cartan}
Consider the following examples.
\begin{enumerate}
\item
Let 
\[
\Delta_{n}:=\left\{ g=\left(\begin{array}{cc}
a & b\\
c & d
\end{array}\right)\in M_{2}(\mathbb{Z})\mid\det(g)=n,N\mid c,N\mid a-1\right\} 
\]
and $\phi_{n}:\tilde{\Delta}_{n}SL_{2}(\mathbb{Z})\rightarrow SL_{2}(\mathbb{Z})$
is a map such that $\pi(\phi_{n}(g))=(0:1)\cdot\lambda_{N}(g)\in\mathbb{P}^{1}(\mathbb{Z}/N\mathbb{Z})$,
where $\pi:SL_{2}(\mathbb{Z})\rightarrow\mathbb{P}^{1}(\mathbb{Z}/N\mathbb{Z})$
is the natural surjection. Then in \cite[Lemma 1]{merel1994universal}, Merel shows that $(\Delta_{n},\phi_{n})$ is a Merel pair for $\Gamma_{1}(N)$,
which is key to modern efficient implementation of Hecke operators. 

\item
Let $p>2$ be a prime, and
let $u\in\mathbb{F}_{p}^{\times}$ be a non-square. Let
\[
\Delta_{n}:=\left\{ g=\left(\begin{array}{cc}
a & b\\
c & d
\end{array}\right)\mid\det(g)=n,p\mid a-d,p\mid b-uc\right\} 
\]
and $\phi_{n}:\tilde{\Delta}_{n}SL_{2}(\mathbb{Z})\rightarrow SL_{2}(\mathbb{Z})$
is a map such that $\phi_{n}(g)^{-1}\cdot\sqrt{u}=\lambda_{p}(g){}^{-1}\cdot\sqrt{u}$,
where the action of $\tilde{\Delta}_{n}SL_{2}(\mathbb{Z})$ on $\mathbb{F}_{p^{2}}-\mathbb{F}_{p}$
is via Mobius transformations. Then similarly $(\Delta_{n},\phi_{n})$
is a Merel pair for $\Gamma_{ns}(p)$ - the nonsplit Cartan subgroup
of level $p$.
\end{enumerate}
\end{example}

Let us write $G_{0}:=G\cap SL_{2}(\mathbb{Z}/N\mathbb{Z})$, and let
$\pi:SL_{2}(\mathbb{Z})\rightarrow G_{0}\backslash SL_{2}(\mathbb{Z}/N\mathbb{Z})$
be the natural map $\pi(g)=G_{0}\cdot\lambda_{N}(g)$, inducing an
isomorphism $\Gamma\backslash SL_{2}(\mathbb{Z})\cong G_{0}\backslash SL_{2}(\mathbb{Z}/N\mathbb{Z})$
. Let $s:G_{0}\backslash SL_{2}(\mathbb{Z}/N\mathbb{Z})\rightarrow SL_{2}(\mathbb{Z})$
be a section of $\pi$. 

Let $\delta_n \in G$ be an element such that $\det(\delta_n) = n$, and let 
\begin{equation}
\Delta_{n}:=\left\{ \alpha\in M_{2}(\mathbb{Z})\mid\det(\alpha)=n,\quad\lambda_{N}(\alpha)\in G\right\} \label{eq:Delta_n}
\end{equation}
Let $\phi_{n}:M_{2}(\mathbb{Z})_{n}\rightarrow SL_{2}(\mathbb{Z})$
be the map defined by 
\[
\phi_{n}(\alpha)=s(G_{0}\cdot n^{-1}\delta_n\cdot\lambda_{N}(\alpha))
\]

\begin{prop}
$(\Delta_{n},\phi_{n})$ is a Merel pair. 
\end{prop}

\begin{proof}
First, note that $\alpha \in \tilde{\Delta}_n$ if and only if 
$$
\det(\alpha) = \det(\tilde{\alpha}) = n \text{ and }
n \lambda_N(\alpha)^{-1}~=~\lambda_N(\tilde{\alpha})~\in~G. 
$$
Therefore
\[
\tilde{\Delta}_{n}=\left\{ \alpha\in M_{2}(\mathbb{Z})\mid\det(\alpha)=n,n^{-1}\lambda_{N}(\alpha)\in G\right\}.
\]
Next, we verify the three conditions in Definition \ref{def:Merel pair}:

1. Let $\gamma\in\tilde{\Delta}_{n}SL_{2}(\mathbb{Z})$ and $g\in SL_{2}(\mathbb{Z})$.
Then
\begin{align*} 
\pi(\phi_{n}(\gamma g)) &= G_{0}\cdot n^{-1}\delta_n\cdot\lambda_{N}(\gamma g)=G_{0}\cdot n^{-1}\delta_n\cdot\lambda_{N}(\gamma)\cdot\lambda_{N}(g) \\
&= \pi(\phi_{n}(\gamma))\cdot\lambda_{N}(g)=G_{0}\cdot\lambda_{N}(\phi_{n}(\gamma))\cdot\lambda_{N}(g) \\
&= G_{0}\cdot\lambda_{N}(\phi_{n}(\gamma)g)=\pi(\phi_{n}(\gamma)g)
\end{align*}
hence $\Gamma\cdot\phi_{n}(\gamma g)=\Gamma\cdot\phi_{n}(\gamma)g$.

2. Let $g\in\tilde{\Delta}_{n}SL_{2}(\mathbb{Z})$. Assume that $n^{-1}\delta_n\cdot\lambda_{N}(g)\in G_{0}$.
Then $n^{-1}\lambda_{N}(g)\in G$, hence $g\in\tilde{\Delta}_{n}$.
Thus $g\in\tilde{\Delta}_{n}$ iff $n^{-1}\delta_n\cdot\lambda_{N}(g)\in G_{0}$.
Now, note that 
\[
G_{0}\cdot n^{-1}\delta_n\cdot\lambda_{N}(g)=\pi(\phi_{n}(g))=G_{0}\cdot\lambda_{N}(\phi_{n}(g))
\]
hence 
\[
n^{-1}\delta_n\cdot\lambda_{N}(g\phi_{n}(g)^{-1})\in G_{0}
\]
showing that $g\phi_{n}(g)^{-1}\in\tilde{\Delta}_{n}$. 

3. Let $\delta_{1},\delta_{2}\in\Delta_{n}$ be such that $\delta_{1}\delta_{2}^{-1}\in SL_{2}(\mathbb{Z})$.
By definition $\lambda_{N}(\tilde{\delta}_{1}),\lambda_{N}(\tilde{\delta}_{2})\in G$,
hence 
\begin{align*}
\lambda_{N}(\delta_{1}\delta_{2}^{-1})
&=\lambda_{N}(\tilde{\delta}_{1}^{-1}\cdot n\cdot n^{-1}\cdot\tilde{\delta}_{2})= \\
&=\lambda_{N}(\tilde{\delta}_{1}^{-1}\tilde{\delta}_{2})=\lambda_{N}(\tilde{\delta}_{1})^{-1}\cdot\lambda_{N}(\tilde{\delta}_{2})\in G
\end{align*}
hence $\delta_{1}\delta_{2}^{-1}\in\Gamma$. 
\end{proof}
\begin{cor}
\label{cor:Hecke Operator Merel}Assume $n \in \det(G)$. Let $\sum_{M}u_{M}M$
satisfy the condition $(C_{n})$. Then, in $\mathbb{M}_{k}(\Gamma_G)$
\begin{equation}
T_{n}^{\vee}([P,g])=\sum_{M}u_{M}[P\vert_{\tilde{M}},\phi_{n}(gM)].
\label{eq:Hecke Operator Merel}
\end{equation}
\end{cor}

\begin{proof}
By Definition \ref{def:Hecke Operator Classic Definition} and \eqref{eq:Delta_n},
we see that $T_{n}^{\vee}=T_{\alpha}^{\vee}$ for any $\alpha\in\Delta_{n}$,
i.e. $T_{n}^{\vee}=T_{\Delta_{n}}^{\vee}$. The corollary now follows
from Corollary \ref{cor:Merel Universal Fourier Expansion}, once
we realize that $M_{2}(\mathbb{Z})_{n}=\tilde{\Delta}_{n}\cdot SL_{2}(\mathbb{Z})$. 

Indeed, for $M\in M_{2}(\mathbb{Z})_{n}$, consider 
\[
g=M\cdot\phi_{n}(M)^{-1}=M\cdot s(G_{0}\cdot n^{-1}\delta_n\cdot\lambda_{N}(M))^{-1}
\]

Then 
\[
\lambda_{N}(g)^{-1}=\lambda_{N}(\phi_{n}(M))\cdot\lambda_{N}(M)^{-1}\in G_{0}\cdot n^{-1}\delta(n)\cdot\lambda_{N}(M)\cdot\lambda_{N}(M)^{-1}=G_{0}\cdot n^{-1}\delta_n
\]
hence 
\[
n^{-1}\lambda_{N}(g)\in\delta_n\cdot G_{0}\subseteq G
\]
showing that $g\in\tilde{\Delta}_{n}$. 
\end{proof}
\begin{cor}
\label{cor:Complexity Hecke Operator}Computation of the Hecke operator
$T_{p}$ on $S_{k}(\Gamma_G)$, for $p \in \det(G)$, can
be done in $O(k \log k \cdot p \log p)$ basic $\operatorname{CosetIndex}$ operations. 
\end{cor}

\begin{proof}
In \cite[Proposition 8]{merel1991operateurs}, Merel shows that one
can compute a certain set $\mathcal{S}_{n}$ of $O(\sigma_{1}(n)\cdot\log n)$
matrices, which in \cite[Proposition 18]{merel1994universal}, he
proves to satisfy condition $(C_{n})$ (Definition \ref{def:Condition C_n}).
Therefore, by Corollary \ref{cor:Hecke Operator Merel}, it is enough
to compute, for a vector $[P,g]\in\mathbb{M}_{k}(\Gamma)$, the sum
\eqref{eq:Hecke Operator Merel}, where each modular symbol computation costs $O(k \log k)$ basic $\operatorname{CosetIndex}$ operations, implying the resulting complexity for computation of the operator $T_{p}^{\vee}$ on $\mathbb{M}_k(\Gamma_G)$. 

Let $\alpha \in M_2(\mathbb{Z})$ be such that $\det(\alpha)=p$ and $\lambda_N(\alpha) \in G$. Then by definition $T_{\alpha} = T_p$. 
Since $\lambda_N(\alpha) \in G$ and $G$ is of real type, by Corollary \ref{cor: Hecke commutes with star}, one can then compute $T_{p}^{\vee}$ on $\mathbb{S}_k(\Gamma_G)$.
Finally, from Proposition \ref{prop: duality Hecke operators} and Theorem \ref{thm: star involution}, the result follows.
\end{proof}
\begin{rem}
In general, the computation of the Hecke operators $\{T_{n}\}$ for
$(n,N)=1$, can be done by using their multiplicativity property,
so Corollary \ref{cor:Complexity Hecke Operator} is all that's needed,
and gives a slightly better complexity of $O(p_{\max}\cdot\log n)$
where $p_{\max}$ is the largest prime factor of $n$. 

Note also that this gives an improvement by a factor of $\log(N)$
over Algorithm \ref{algo: Hecke operator double coset},
which is very significant in practice.
\end{rem}

We would also like to mention that under some mild assumptions, the decomposition \ref{eqn: isotypic decomposition} is Hecke-equivariant, hence one can compute the Hecke operator on each of the isotypic subspaces separately.

\begin{lem}
Let $G \subseteq G' \subseteq GL_2(\mathbb{Z}/ N\mathbb{Z})$ be such that $G$ is normal in $G'$. 
Let $\varepsilon: G' / G \rightarrow \mathbb{Q}(\zeta)^{\times}$ be a character, and let $\alpha \in M_2(\mathbb{Z})$ be such that $\det(\alpha) = p$ is a prime and $\lambda_N(\alpha) \in G$. Then the subspace $S_k(\Gamma_G, \varepsilon)$ is invariant under $T_{\alpha}$.
\end{lem}

\begin{proof}
Let $f \in S_k(\Gamma_G, \varepsilon)$, and let $\gamma \in \Gamma_{G'}$. Let $R$ be a set of coset representatives for $\Gamma_G \backslash \Gamma_G \alpha \Gamma_G$. Then 
\begin{equation*}
T_\alpha(f) \vert_{\gamma} = \sum_{\delta \in R} f \vert_{\delta \gamma}
= \sum_{\delta \in R} \left( f \vert_{\gamma} \right) \vert_{\gamma^{-1} \delta \gamma}
= \sum_{\delta \in R} \left( \varepsilon(\gamma) \cdot f \right) \vert_{\gamma^{-1} \delta \gamma}
= \varepsilon(\gamma) \sum_{\delta \in R} f \vert_{\gamma^{-1} \delta \gamma}.
\end{equation*}
Since $\gamma \in \Gamma_{G'}$ and $G$ is normal in $G'$, for any $x \in \Gamma_G$ we have 
$\lambda_N(\gamma^{-1} x \gamma) \in G$, hence $\gamma^{-1} x \gamma \in \Gamma_G$. 
Thus $\gamma^{-1} \Gamma_G \gamma = \Gamma_G$.

As $\Gamma_G \alpha \Gamma_G = \bigsqcup_{\delta \in R} \Gamma_G \delta$, it follows that
$$ \Gamma_G \gamma^{-1} \alpha \gamma \Gamma_G = \gamma^{-1} \Gamma_G \alpha \Gamma_G \gamma = 
\bigsqcup_{\delta \in R} \gamma^{-1} \Gamma_G \delta \gamma =  \bigsqcup_{\delta \in R} \Gamma_G \gamma^{-1} \delta \gamma. $$
However, since $\gamma \in \Gamma_{G'}$ and $\lambda_N(\alpha) \in G$, it follows that 
$ \lambda_N(\gamma^{-1} \alpha \gamma) \in G $, and so by \cite[Lemma 3.29]{shimura1971introduction}, we see that
$$ \Gamma_G \gamma^{-1} \alpha \gamma \Gamma_G = \{g \in M_2(\mathbb{Z}) \mid \det(g)=p \text{ and } \lambda_N (g) \in G \} = 
\Gamma_G \alpha \Gamma_G.$$
It follows that $T_{\alpha}(f) = \sum_{\delta \in R} f\vert_{\gamma^{-1} \delta \gamma}$, so that 
$T_{\alpha}(f)\vert_{\gamma} = \varepsilon(\gamma) \cdot T_{\alpha}(f)$, hence the result.
\end{proof}

\begin{rem}
Note that it is not enough to demand that $\Gamma_G$ will be normal in $\Gamma_{G'}$. For example, if $t^2 \mid N$ and $t \mid 24$, then, as described in \cite{ogg1969modular}, there is an element $u_t \in \mathfrak{N}(\Gamma_0(N))$ defined over $\mathbb{Q}(\zeta_t)$ which commutes only with the Hecke operators $T_p$ such that $p \equiv 1 \bmod t$ (equivalently, $p$ splits in $\mathbb{Q}(\zeta_t)$).
Indeed, $u_t$ does not normalize the group $G \subseteq GL_2(\mathbb{Z} / N \mathbb{Z})$ appearing in Example \ref{exa: Gamma0} $(3)$.
\end{rem}

\section{Degeneracy Maps} \label{sec: degeneracy maps}

Our next step would be to decompose our spaces of modular forms to irreducible modules for the Hecke algebra. This should be possible once we apply our algorithms from Section \ref{sec: Hecke operators}. The only obstacle in our path to obtain the eigenvectors lie in the fact that we still do not have a multiplicity one result due to the existence of oldforms. We therefore make a small detour, and describe the generalization to our situation of the degeneracy maps, and the decomposition to spaces of newforms.

\subsection{Petersson Inner Product} \label{subsec: Petersson}

Before we describe the degeneracy maps themselves, in order to obtain a good notion of oldforms, we briefly recall the definitions and properties of the Petersson inner product, so that we will consequently be able to define newforms.

\begin{defn}
Let $\Gamma \subseteq SL_2(\mathbb{Z})$ be a congruence subgroup. The \emph{Petersson inner product},
$$
\langle , \rangle_{\Gamma} : S_k(\Gamma) \times S_k(\Gamma) \rightarrow \mathbb{C},
$$
is given by 
\begin{equation} \label{eqn: petersson}
\langle f,g \rangle_{\Gamma} = \frac{1}{V_{\Gamma}} \int_{X_{\Gamma}} f(\tau) \overline{g(\tau)} \left( \Im(\tau) \right)^k d \mu(\tau)
\end{equation}
where $d \mu{\tau} = \frac{dx dy}{y^2}$ is the hyperbolic measure on $\mathcal{H}$ and $V_{\Gamma} = \int_{X_{\Gamma}} d \mu{\tau}$ is the volume of $X_{\Gamma}$.
\end{defn}

The following standard results then help us to see that it is indeed a well defined inner product, and to compute adjoints of Hecke operators with respect to this inner product.

\begin{prop}[{\cite[Section 5.4 and Proposition 5.5.2]{diamond2005first}}] \label{prop: adjoints}
Let $\Gamma \subseteq SL_2(\mathbb{Z})$ be a congruence subgroup.
The integral in the Petersson inner product \eqref{eqn: petersson} is well defined and convergent. The pairing is linear in $f$, conjugate linear in $g$, Hermitian-symmetric and positive definite. 
Let $\alpha \in GL_2^{+}(\mathbb{Q})$, and set $\alpha' := \det(\alpha) \cdot \alpha^{-1}$. Then
\begin{enumerate}
\item If $\alpha^{-1} \Gamma \alpha \subseteq SL_2(\mathbb{Z})$ then for all $f \in S_k(\Gamma)$ and $g \in S_k(\alpha^{-1} \Gamma \alpha)$ 
$$
\langle f \vert_{\alpha} , g \rangle_{\alpha^{-1} \Gamma \alpha} = \langle f, g \vert_{\alpha'} \rangle_{\Gamma}.
$$
\item For all $f,g \in S_k(\Gamma)$,
$$
\langle T_\alpha f, g \rangle = \langle f, T_{\alpha'} g \rangle.
$$
In particular, if $\alpha^{-1} \Gamma \alpha = \Gamma$ then the adjoint of $f \mapsto f \vert_{\alpha}$ is $g \mapsto g \vert_{\alpha'}$ and in any case $T_\alpha^{\star} = T_{\alpha'}$. 
\end{enumerate}
\end{prop}

The correspondence between $\alpha$ and $\alpha'$ forces us to consider some more operators on this space. 
\begin{defn}
Let $G \subseteq GL_2(\mathbb{Z} / N \mathbb{Z})$, and let $\mathfrak{N}(G)$ be its normalizer in 
$GL_2(\mathbb{Z} / N \mathbb{Z})$. Then the quotient $Q := \Gamma_{\mathfrak{N}(G)} / \Gamma_G$ acts naturally on $M_k(\Gamma_G)$ via
\begin{align*}
\langle q \rangle : M_k(\Gamma_G) &\rightarrow M_k(\Gamma_G) \\
f & \mapsto f \vert_{\alpha} \text{ for any } \alpha \in \Gamma_{\mathfrak{N}(G)} \text{ with } \alpha \equiv q \bmod \Gamma_G.
\end{align*}
This operator is called a \emph{diamond operator}. 
\end{defn}

For any irreducible representation $\rho: Q \rightarrow GL_{n_{\rho}}(\mathbb{C})$, we let 
$$
M_k(\Gamma_G, \rho) := \rho \otimes \Hom_{Q} (\rho, M_k(\Gamma_G)) \hookrightarrow M_k(\Gamma_G)
$$
be the $\rho$-isotypic component of $M_k(\Gamma)$. 
Then $M_k(\Gamma_G) = \bigoplus_{\rho} M_k(\Gamma_G, \rho)$ where $\rho$ runs over all irreducible representations of $Q$. 

\begin{lem}
Let $\alpha \in M_2(\mathbb{Z})$ be such that $\lambda_N(\alpha) \in G$, and let $q \in \Gamma_{\mathfrak{N}(G)} / \Gamma_G$. Then $T_\alpha \circ \langle q \rangle = \langle q \rangle \circ T_\alpha$.
\end{lem}

\begin{proof}
Denote by $\tilde{q} \in \Gamma_{\mathfrak{N}(G)}$ an element such that $\tilde{q} \equiv q \bmod \Gamma_G$. 
Since $\lambda_N(\alpha) \in G$ and $\lambda_N(\tilde{q}) \in \mathfrak{N}(G)$, we see that 
$$
\lambda_N(\tilde{q}^{-1} \alpha \tilde{q})  = \lambda_N(\tilde{q})^{-1} \lambda_N(\alpha) \lambda_N(\tilde{q}) \in \lambda_N(\tilde{q})^{-1} G \lambda_N(\tilde{q}) = G
$$ 
Moreover, $\det(\tilde{q}^{-1} \alpha \tilde{q}) = \det(\alpha)$, so that $\lambda_N(\tilde{q}^{-1} \alpha \tilde{q}) \in G_0 \lambda_N(\alpha)$.
By \cite[Lemma 3.29 (a)]{shimura1971introduction}, it follows that $\tilde{q}^{-1} \alpha \tilde{q} \in \Gamma_G \alpha \Gamma_G$ hence 
$$
\tilde{q}^{-1} \Gamma_G \alpha \Gamma_G \tilde{q} = 
\Gamma_G \tilde{q}^{-1} \alpha \tilde{q} \Gamma_G = \Gamma_G \alpha \Gamma_G 
$$ 
so that if 
$\Gamma_G \alpha \Gamma_G = \bigsqcup_{i} \Gamma_G \alpha_i $ is a disjoint union, then
$$
\Gamma_G \alpha \Gamma_G = \tilde{q}^{-1} \Gamma_G \alpha \Gamma_G \tilde{q} = 
\tilde{q}^{-1} \left( \bigsqcup_{i} \Gamma_G \alpha_i \right) \tilde{q} = 
\bigsqcup_{i} \tilde{q}^{-1} \Gamma_G \alpha_i \tilde{q} = \bigsqcup_{i} \Gamma_G \tilde{q}^{-1} \alpha \tilde{q}.
$$

Therefore, for any $f \in M_k(\Gamma_G)$ we have
$$
\langle q \rangle T_\alpha f = \sum_{i} f \vert_{\alpha_i \tilde{q}} = \sum_{i} f \vert_{\tilde{q} \tilde{q}^{-1} \alpha_i \tilde{q}}
= \sum_{i} (\langle q \rangle f) \vert_{\tilde{q}^{-1} \alpha_i \tilde{q}} = T_\alpha \langle q \rangle f. \mbox{\qedhere}
$$
\end{proof}

Before we continue, we need a simple Lemma, which is a variant of \cite[Lemma 3.29]{shimura1971introduction}.

\begin{lem} \label{lem: disjoint union}
Let $\Gamma_{1}\subseteq\Gamma_{2}$ be congruence subgroups of levels
$N_{1},N_{2}$ respectively. Let 
\[
\Phi_i=\{\alpha\in M_{2}(\mathbb{Z})\mid\det\alpha>0,(\det(\alpha),N_{i})=1,\lambda_{N_{i}}(\Gamma_{i}\alpha)=\lambda_{N}(\alpha\Gamma_{i})\},\quad i \in \{1,2\}
\]
Then the following assertions hold.
\begin{enumerate}
\item $\Gamma_{1}\alpha\Gamma_{1}=\{\xi\in\Gamma_{2}\alpha\Gamma_{2}\mid\lambda_{N_{1}}(\xi)\in\lambda_{N_{1}}(\Gamma_{1}\cdot\alpha)\}$
if $\alpha\in\Phi_1$.
\item $\Gamma(N_{1})\alpha\Gamma(N_{1})=\Gamma(N_{1})\beta\Gamma(N_{1})$
if and only if $\Gamma_{2}\alpha\Gamma_{2}=\Gamma_{2}\beta\Gamma_{2}$
and $\alpha\equiv\beta\mod N_{1}$. (for $\alpha,\beta$ such that
$(\det\alpha,N_{1})=(\det\beta,N_{1})=1$)
\item $\Gamma_{2}\alpha\Gamma_{2}=\Gamma_{2}\alpha\Gamma_{1}=\Gamma_{1}\alpha\Gamma_{2}$
if $\alpha\in\Phi_{2}$ and $(\det\alpha,N_{1})=1$.
\item $\Gamma_{i}\alpha\Gamma_{i}=\Gamma_{i}\alpha\Gamma(N_{i})=\Gamma(N_{i})\alpha\Gamma_{i}$
if $\alpha\in\Phi_{i}$ for $i =1,2$.
\item If $\alpha\in\Phi_{1}$ and $\Gamma_{1}\alpha\Gamma_{1}=\bigsqcup_{i}\Gamma_{1}\alpha_{i}$, then $\Gamma_{2}\alpha\Gamma_{2}=\bigsqcup_{i}\Gamma_{2}\alpha_{i}$. 
\end{enumerate}
\end{lem}

\begin{proof}
First, (4) is simply a restatement of \cite[Lemma 3.29 (4)]{shimura1971introduction}.
To show (3), note that $\gcd(N_1, \det \alpha \cdot N_2)=N_2$. 
By \cite[Lemma 3.28]{shimura1971introduction} and \cite[Lemma 3.9]{shimura1971introduction}, we have 
\[
\Gamma(N_{2})=\Gamma(N_{2}\cdot \det \alpha)\cdot\Gamma(N_{1})\subseteq\alpha^{-1}\Gamma(N_{2})\alpha\Gamma(N_{1})
\]
so that $\alpha^{-1}\Gamma(N_{2})\alpha\Gamma(N_{2})\subseteq\alpha^{-1}\Gamma(N_{2})\alpha\Gamma(N_{1})$.
Hence $\Gamma(N_{2})\alpha\Gamma(N_{2})\subseteq\Gamma(N_{2})\alpha\Gamma(N_{1})\subseteq\Gamma(N_{2})\alpha\Gamma_{1}$.
Therefore $\Gamma_{2}\alpha\Gamma(N_{2})\subseteq\Gamma_{2}\alpha\Gamma_{1}\subseteq\Gamma_{2}\alpha\Gamma_{2}$.
But from (4), as $\alpha\in\Phi_{2}$, we know that there is an equality
$\Gamma_{2}\alpha\Gamma(N_{2})=\Gamma_{2}\alpha\Gamma_{2}$ . This
shows (3). Next, to see (1) note that by \cite[Lemma 3.29 (1)]{shimura1971introduction},
we have $\Gamma_{1}\alpha\Gamma_{1}=\{\xi\in SL_2(\mathbb{Z}) \alpha SL_2(\mathbb{Z}) \mid\lambda_{N_{1}}(\xi)\in\lambda_{N_{1}}(\Gamma_{1}\cdot\alpha)\}$.
Since we also know $\Gamma_{1}\alpha\Gamma_{1}\subseteq\Gamma_{2}\alpha\Gamma_{2}$,
(1) follows. (2) is a special case of (1). Finally, let $\alpha\in\Phi_{1}$,
and $\Gamma_{1}\alpha\Gamma_{1}=\bigsqcup_{i}\Gamma_{1}\alpha_{i}$.
Then $\Gamma_{2}\alpha\Gamma_{2}=\Gamma_{2}\alpha\Gamma_{1}=\bigcup_{i}\Gamma_{2}\alpha_{i}$
by (3). Assume $\Gamma_{2}\alpha_{i}=\Gamma_{2}\alpha_{j}$. Then
$\alpha_{i}=\gamma\alpha_{j}$ for some $\gamma\in\Gamma_{2}$. By
(1) 
\[
\lambda_{N_{1}}(\alpha_{i})\in\lambda_{N_{1}}(\Gamma_{1})\cdot\lambda_{N_{1}}(\alpha)=\lambda_{N_{1}}(\Gamma_{1})\cdot\lambda_{N_{1}}(\alpha_{j}).
\]
Thus, there exists some $\delta\in\Gamma_{1}$ such that $\alpha_{i}\equiv\delta\alpha_{j}\mod N_{1}$.
It follows that $\gamma\equiv\delta\mod N_{1}$. Since $\Gamma(N_{1})\subseteq\Gamma_{1}$,
we have $\gamma\in\Gamma_{1}$. This proves (5). 
\end{proof}

We may now state the adjoints of the Hecke operators with respect to the Petersson inner product. 
\begin{thm}
In the inner product spaces $S_k(\Gamma_G)$, the diamond operators $\langle q \rangle$ and the Hecke operator $T_p$ for $p \in \det(G)$ a prime, have adjoints
\begin{align*}
\langle q \rangle^{\star} = \langle q \rangle^{-1} \quad \text{ and } \quad 
T_{p}^{\star} = \langle \sigma_p \rangle \cdot T_{p}
\end{align*}
where $\sigma_p \in \Gamma_{\mathfrak{N}(G)}$ is an element such that 
$\lambda_N(\sigma_p) = p \cdot \lambda_N(\alpha)^{-2}$, and 
$\alpha \in M_2(\mathbb{Z})$ is any element with $\det(\alpha) = p$ such that $\lambda_N(\alpha) \in G$.
Thus the Hecke operators $\langle q \rangle$ and $T_n$ for $n \in \det(G)$ are normal.
\end{thm}

\begin{proof}
The first assertion is just a restatement of Proposition \ref{prop: adjoints} (1). 
For the second assertion, we note that if $\alpha \in M_2(\mathbb{Z})$ is such that $\lambda_N(\alpha) \in G$ and $\det(\alpha) = p$, then $\alpha \in \Phi_{\mathfrak{N}(G)}$. Indeed, if $\nu \in \mathfrak{N}(G)$ then for any $g \in G$ we have
$$
(\lambda_N(\alpha) \nu \lambda_N(\alpha)^{-1}) g (\lambda_N(\alpha) \nu^{-1} \lambda_N(\alpha)^{-1}) \in 
\lambda_N(\alpha) \nu G \nu^{-1} \lambda_N(\alpha)^{-1} \subseteq G
$$
so that $\lambda_N(\alpha) \nu \lambda_N(\alpha)^{-1} \in \mathfrak{N}(G)$.

Therefore, by Lemma \ref{lem: disjoint union} (3), we see that 
$\Gamma_{\mathfrak{N}(G)} \alpha \Gamma_{\mathfrak{N}(G)} =  \Gamma_G \alpha \Gamma_{\mathfrak{N}(G)}$.
Moreover, we have 
$$
\lambda_N(\alpha') = \det(\alpha) \lambda_N(\alpha)^{-1} \in \mathfrak{N}(G)
$$
hence 
$
\Gamma_{\mathfrak{N}(G)} \alpha' \Gamma_{\mathfrak{N}(G)} = \Delta_p^{\mathfrak{N}(G)} = 
\Gamma_{\mathfrak{N}(G)} \alpha \Gamma_{\mathfrak{N}(G)} = \Gamma_G \alpha \Gamma_{\mathfrak{N}(G)}
$.

Let $\gamma \in \Gamma_{\mathfrak{N}(G)}$ be such that $\alpha' \in \Gamma_G \alpha \gamma$. Now, as $\gamma$ normalizes $\Gamma_G$ we see that
$\Gamma_G \alpha' \Gamma_G = \Gamma_G \alpha \gamma \Gamma_G = \Gamma_G \alpha \Gamma_G \gamma$
and hence if $\Gamma_G \alpha \Gamma_G = \bigsqcup_{i} \Gamma_G \alpha_i$, then 
$ \Gamma_G \alpha' \Gamma_G = \bigsqcup_{i} \Gamma_G \alpha_i \gamma$.
It then follows from Proposition \ref{prop: adjoints} (2) that 
$T_p^{\star} = T_{\alpha}^{\star} = T_{\alpha'} = \langle \gamma \rangle T_{\alpha}$.
Finally, we note that
$$
\lambda_N(\gamma) \in \lambda_N(\Gamma_G) \lambda_N(\alpha)^{-1} \lambda_N(\alpha') \subseteq
G \cdot p \lambda_N(\alpha)^{-2} = G \cdot \lambda_N(\sigma_p)
$$
so that $\gamma \in \Gamma_G \cdot \sigma_p$, showing that $\langle \gamma \rangle = \langle \sigma_p \rangle$.
\end{proof}

\begin{rem}
The $\sigma_p$ actually acts via a character. Indeed, if $A \subseteq \mathfrak{N}(G)$ is such that $A / G$ is abelian, and $A$ contains the center of diagonal matrices, then $\sigma_p \in A$ and all the irreducible representations of $A / G$ are $1$-dimensional.
\end{rem}

This leads to the following Theorem.
\begin{cor} \label{cor: basis of eigenvectors}
The space $S_k(\Gamma_G)$ has an orthogonal basis of simultaneous eigenvectors for the Hecke operators $\{ T_n \mid n \in \det(G) \} $ and the diamond operators $\{ \langle q \rangle \mid q \in \Gamma_{\mathfrak{N}(G)} / \Gamma_G \}$.
\end{cor}

\begin{proof}
This is a corollary of the Spectral Theorem for a family of commuting normal operators.
\end{proof}

\begin{defn}
We say that a modular form $0 \ne f \in S_k(\Gamma_G)$ is an \emph{eigenform} if for all $n \in \det(G)$, $f$ is an eigenvector of $T_n$ and for all $q \in \Gamma_{\mathfrak{N}_G} / \Gamma_G$, $f$ is an eigenvector of $\langle q \rangle$.  
\end{defn}

\subsection{Degeneracy Maps} \label{subsec: deg maps and Hecke}

We begin by defining degeneracy maps in general, on our space of modular symbols.
\begin{defn}
Let $\Gamma_1 \trianglelefteq \Gamma_{1}^{'}$, $\Gamma_2 \trianglelefteq \Gamma_{2}^{'}$, be two pairs of finite index subgroups in $SL_2(\mathbb{Z})$. Let $\varepsilon_i : \Gamma_{i}^{'} / \Gamma_{i} \rightarrow \mathbb{Q}(\zeta_i)^{\times}$ for $i=1,2$ be characters. Let $t \in GL_2^{+}(\mathbb{Q})$ be such that $t^{-1} \Gamma_{1}^{'} t \subseteq \Gamma_2^{'}$, and $\varepsilon_{2}\circ\text{Inn}(t)^{-1}\mid_{\Gamma_{1}^{'}}=\varepsilon_{1}$. Fix a choice $R_t$ of coset representatives for 
$\Gamma_{1}^{'} \backslash t \cdot \Gamma_2{'}$. Let
\begin{align*}
& \alpha_t^{\vee}: \mathbb{M}_k(\Gamma_1, \varepsilon_1) \rightarrow \mathbb{M}_k(\Gamma_2, \varepsilon_2),  
& \beta_t^{\vee}: \mathbb{M}_k(\Gamma_2, \varepsilon_2) \rightarrow \mathbb{M}_k(\Gamma_1, \varepsilon_1) \\
& \alpha_t^{\vee}(x) = t^{-1} x
& \beta_t^{\vee}(x) = \sum_{t \gamma_2 \in R_t} \varepsilon_2(\gamma_2)^{-1} t \gamma_2 \cdot x
\end{align*}
\end{defn}

We show that these operators are well defined.
\begin{lem} \label{lem: deg maps are defined}
The operators $\alpha_t^{\vee}$ and $\beta_t^{\vee}$ are well defined, and moreover the composition $\alpha_t^{\vee} \circ \beta_t^{\vee}$ is multiplication by $[\Gamma_{2}^{'} : t^{-1} \Gamma_{1}^{'} t]$.
\end{lem}

\begin{proof}
First, note that as $t^{-1}\Gamma_{1}^{'}t\subseteq\Gamma_{2}^{'}$,
for any $\gamma_{1}\in\Gamma_{1}^{'}$, $\gamma_{2}\in\Gamma_{2}^{'}$
one has 
\[
\gamma_{1} t\gamma_{2}=t t^{-1}\gamma_{1}t\gamma_{2}\in t \Gamma_{2}^{'}
\]
so that $\Gamma_{1}^{'}$ indeed acts on $t \Gamma_{2}^{'}$. 
To show that $\alpha_{t}$ is well defined, we must show that for
each $x\in\mathbb{M}_{k}(\Gamma_{1},\varepsilon_{1})$ and $\gamma_{1}\in\Gamma_{1}^{'}$
, we have 
\[
\alpha_{t}^{\vee}(\gamma_{1}x-\varepsilon_{1}(\gamma_{1}) x)=0\in\mathbb{M}_{k}(\Gamma_{2},\varepsilon_{2})
\]
We have, by assumption
\[
\alpha_{t}^{\vee}(\gamma_{1}x)=t^{-1}\gamma_{1} x=t^{-1}\gamma_{1}t t^{-1}x=\varepsilon_{2}(t^{-1}\gamma_{1}t) t^{-1}x=\varepsilon_{1}(\gamma_{1}) \alpha_{t}^{\vee}(x)
\]
hence the result.

We next verify that $\beta_{t}$ is well defined. 

Suppose that $\gamma_{1}\in\Gamma_{1}^{'}$ and $\gamma_{2}\in\Gamma_{2}^{'}$
, then 
\[
\gamma_{1} t\gamma_{2}=t (t^{-1}\gamma_{1}t) \gamma_{2}
\]
Moreover, for any $x\in\mathbb{M}_{k}(\Gamma_{2},\varepsilon_{2})$,
we have
\begin{equation} \label{eq: degeneracy maps}
\varepsilon_{2}((t^{-1}\gamma_{1}t)\gamma_{2})^{-1}\gamma_{1} t\gamma_{2} x 
= \varepsilon_{1}(\gamma_{1})^{-1} \varepsilon_{2}(\gamma_{2})^{-1} \gamma_{1} t\gamma_{2} x 
= \varepsilon_{1}(\gamma_{1})^{-1} \gamma_{1} (\varepsilon_{2}(\gamma_{2})^{-1} t\gamma_{2} x)
\end{equation}
But as $\varepsilon_{2}(\gamma_{2})^{-1} t\gamma_{2} x\in\mathbb{M}_{k}(\Gamma_{1},\varepsilon_{1})$,
\eqref{eq: degeneracy maps} is simply $\varepsilon_{2}(\gamma_{2})^{-1} t\gamma_{2} x$. 

Thus, replacing the representative $t \gamma_{2}$ by $\gamma_{1}  t\gamma_{2}$
does not change the result, so that $\beta_{t}^{\vee}$ is independent of
the choice of the representatives $R_{t}$. 

Next, we must show that for any $\gamma\in\Gamma_{2}^{'}$ and any
$x\in\mathbb{M}_{k}(\Gamma_{2},\varepsilon_{2})$ , $\beta_{t}^{\vee}(\gamma x)=\varepsilon_{2}(\gamma) \beta_{t}^{\vee}(x)$.
However, for $\gamma\in\Gamma_{2}^{'}$, using the fact that $\beta_{t}$
is independent of the choice of representatives, and that $\Gamma_{2}^{'}$
acts on $\Gamma_{1}^{'}\backslash t \Gamma_{2}^{'}$ by right
translations, we get
\begin{align*}
\beta_{t}^{\vee}(\gamma x)&=\sum_{t\gamma_{2}\in R_{t}}\varepsilon_{2}(\gamma_{2})^{-1} t\gamma_{2} \gamma x
=\sum_{t\gamma_{2}\gamma^{-1}\in R_{t}}\varepsilon_{2}(\gamma_{2}\gamma^{-1})^{-1}t\gamma_{2}\gamma^{-1} \gamma x \\
&=\varepsilon_{2}(\gamma) \sum_{t\gamma_{2}\in R_{t}}\varepsilon_{2}(\gamma_{2})^{-1} t\gamma_{2} x=\varepsilon_{2}(\gamma) \beta_{t}^{\vee}(x)
\end{align*}

To compute $\alpha_{t}^{\vee}\circ\beta_{t}^{\vee}$, we use that $\#R_{t}=[\Gamma_{2}^{'}: t^{-1} \Gamma_{1}^{'} t]$:
\begin{align*}
\alpha_{t}^{\vee}(\beta_{t}^{\vee}(x))
&=\alpha_{t}^{\vee}\left(\sum_{t\gamma_{2}\in R_{t}}\varepsilon_{2}(\gamma_{2})^{-1} t\gamma_{2} x\right) \\
&=\sum_{t\gamma_{2}\in R_{t}}\varepsilon_{2}(\gamma_{2})^{-1} \gamma_{2}x=\sum_{t\gamma_{2}\in R_{t}}x=[\Gamma_{2}^{'}: t^{-1} \Gamma_{1}^{'} t]  x. \qedhere
\end{align*}
\end{proof}

We also note that by using the pairing \eqref{eq:pairing of modular symbols and modular forms} we can get maps on spaces of cusp forms.

\begin{defn}
Let $\Gamma_1 \trianglelefteq \Gamma_{1}^{'}$, $\Gamma_2 \trianglelefteq \Gamma_{2}^{'}$, be two pairs of finite index subgroups in $SL_2(\mathbb{Z})$. Let $\varepsilon_i : \Gamma_{i}^{'} / \Gamma_{i} \rightarrow \mathbb{Q}(\zeta_i)^{\times}$ for $i=1,2$ be characters. Let $t \in GL_2^{+}(\mathbb{Q})$ be such that $t^{-1} \Gamma_{1}^{'} t \subseteq \Gamma_2^{'}$, and $\varepsilon_{2}\circ\text{Inn}(t)^{-1}\mid_{\Gamma_{1}^{'}}=\varepsilon_{1}$. Fix a choice $R_t$ of coset representatives for 
$\Gamma_{1}^{'} \backslash t \cdot \Gamma_2{'}$. Let
\begin{align*}
& \alpha_t: S_k(\Gamma_2, \varepsilon_2) \rightarrow S_k(\Gamma_1, \varepsilon_1),  
& \beta_t: S_k(\Gamma_1, \varepsilon_1) \rightarrow S_k(\Gamma_2, \varepsilon_2) \\
& f \mapsto f \vert_{t^{-1}}
& f \mapsto \sum_{t \gamma_2 \in R_t} \varepsilon_2(\gamma_2)^{-1} f \vert_{t \gamma_2}
\end{align*}
\end{defn}

\subsection{Degeneracy maps and Hecke operators}
Next, as in the classical theory for Iwahori level, we would like to show that the degeneracy maps commute with the Hecke operators. Alas, this is not always the case as the following example shows.

\begin{example}
Let $\Gamma_1 = \Gamma(7)$ and let $\Gamma_2 = \Gamma_{ns}(7)$ be the non-split Cartan at level $7$. Then there is a natural inclusion map $\alpha_1^{\vee}: \mathbb{M}_2(\Gamma_1) \rightarrow \mathbb{M}_2(\Gamma_2)$, but $\alpha_1$ does not commute with the standard Hecke operators for these spaces. 
In this case $\dim S_2(\Gamma_{ns}(7)) = 1$ is spanned by a single element $f$. We have 
$\alpha_1(f) = f \in S_2(\Gamma(7))$, but $f$ is not an eigenform of the Hecke algebra on $\Gamma(7)$. 
Indeed, there are $3$ such eigenforms, whose $q$-expansions are all defined over $\QQ(\sqrt{-3})$, while the $q$-expansion of $f$ is only defined over $\QQ(\zeta_7)^+$. 
\end{example}

In the next proposition we describe the conditions under which the degeneracy maps commute with the Hecke operators.

\begin{prop} \label{prop: deg commutes with Hecke}
Let $N_G \mid N_H$ be positive integers. Let $H \subseteq GL_2(\mathbb{Z} / N_{H} \mathbb{Z})$ and 
$G\subseteq GL_2(\mathbb{Z} / N_{G} \mathbb{Z})$ be subgroups such that $\lambda_{N_G}(H) \subseteq G$. 
Let $p$ be a prime number such that $p \in \det(H)$, and let $t \in M_2(\mathbb{Z})$ be such that 
$N_G \cdot \det(t) \mid N_H$ and $t^{-1} \Gamma_{H} t \subseteq \Gamma_{G}$.
Let $T_{p,?}^{\vee}$ be the Hecke operator at $p$ on the space $\mathbb{M}_k(\Gamma_{?})$ for $? \in \{G,H\}$.
Then 
$$
T_{p,G}^{\vee} \circ \alpha_t^{\vee} = \alpha_t^{\vee} \circ T_{p,H}^{\vee}.
$$
\end{prop}

\begin{proof}
Denote $d := \det(t)$, so that $d \cdot t^{-1} \in M_2(\mathbb{Z})$. 
Since $t^{-1} \Gamma_{H} t \subseteq \Gamma_G$, we have also $d \cdot t^{-1} \Gamma_{H} t \subseteq d \cdot \Gamma_G$ so that 
\begin{equation} \label{eqn: degeneracy conj 1}
\lambda _{N_G \cdot d}(d t^{-1}) \cdot \lambda_{N_G \cdot d} (\Gamma_H) \cdot \lambda_{N_G \cdot d}(t) 
=\lambda _{N_G \cdot d} \left(d t^{-1} \Gamma_H t \right) \subseteq \lambda_{N_G \cdot d} (d \cdot \Gamma_G) = d \cdot \lambda_{N_G \cdot d}(\Gamma_G).
\end{equation}
However, we also assumed $N_G \cdot d \mid N_H$, hence 
\begin{equation} \label{eqn: degeneracy conj 2}
\lambda _{N_G \cdot d}(\Gamma_H) = \lambda _{N_G \cdot d}(\lambda _{N_H}(\Gamma_H)) = \lambda _{N_G \cdot d}(H).
\end{equation}
Let
$$
\Delta_p^{?} := \left\{  \alpha \in M_2(\mathbb{Z}) \mid \det(\alpha) = p \text{ and } \lambda_{N_{?}}(\alpha) \in ? \right\},\quad ? \in \{G,H\}
$$
and let $\alpha \in \Delta_p^{H}$.
Then combining \eqref{eqn: degeneracy conj 1} and \eqref{eqn: degeneracy conj 2} we see that 
\begin{align*}
\lambda_{N_G \cdot d} (d t^{-1}\alpha t) &= 
\lambda_{N_G \cdot d} (dt^{-1}) \lambda_{N_G \cdot d} (\lambda_{N_H} (\alpha)) \lambda_{N_G \cdot d}(t) \\
& \in \lambda_{N_G \cdot d} (dt^{-1}) \lambda_{N_G \cdot d} (H) \lambda_{N_G \cdot d}(t) \subseteq d \cdot \lambda_{N_G \cdot d}(\Gamma_G).
\end{align*}
In particular, reducing modulo $d$, we see that $dt^{-1}\alpha t \in d \cdot M_2(\ZZ)$, hence $t^{-1} \alpha t \in M_2(\mathbb{Z})$.
Therefore $\lambda_{N_G}(t^{-1} \alpha t) \in \lambda_{N_G}(\Gamma_G) = G$. It follows that $t^{-1} \alpha t \in \Delta_p^{G}$. 

Since $p \in \det(H)$, we also get $p \in \det(\lambda_{N_G}(H)) \subseteq \det(G)$. Therefore, we have by \cite[Lemma 3.29]{shimura1971introduction} that $\Delta_p^{H} = \Gamma_{H} \alpha \Gamma_{H}$ and 
$\Delta_p^{G} = \Gamma_{G} \cdot t^{-1} \alpha t \cdot \Gamma_G$.

Let $\Delta_p^{H} = \bigsqcup_{i} \Gamma_{H} \alpha_i$ be its coset decomposition. Then 
\begin{align*}
t^{-1} \Gamma_H t \cdot t^{-1} \alpha t \cdot t^{-1} \Gamma_H t 
&=t^{-1} \Gamma_{H} \alpha \Gamma_H t 
= t^{-1} \left( \bigsqcup_{i} \Gamma_H \alpha_i \right) t \\
&= \bigsqcup_{i} t^{-1} \Gamma_H \alpha_i t = 
\bigsqcup_{i} t^{-1} \Gamma_H t \cdot t^{-1} \alpha_i t.
\end{align*}

Now, applying Lemma \ref{lem: disjoint union} (5) to $t^{-1} \Gamma_H t \subseteq \Gamma_G$ and $t^{-1} \alpha t$, we obtain that 
\begin{equation} \label{eqn: G double coset decomp}
\Delta_p^{G} = \Gamma_G \cdot t^{-1} \alpha t \cdot \Gamma_G = \bigsqcup_{i} \Gamma_G\cdot  t^{-1} \alpha_i t. 
\end{equation}

Finally, we use \eqref{eqn: G double coset decomp} to compute
$$
T_{p,G}^{\vee} \circ \alpha_t^{\vee} (x) = \sum_{i} t^{-1} \alpha_i t \cdot t^{-1} x = t^{-1} \sum_{i} \alpha_i x = \alpha_t^{\vee} \circ T_{p,H}^{\vee} (x). \mbox{\qedhere}
$$

\end{proof}

We also note that by using the pairing \eqref{eq:pairing of modular symbols and modular forms} we have the same for cusp forms.

\begin{cor} \label{cor: Hecke commutes with degeneracy}
Under the assumptions of Corollary \ref{prop: deg commutes with Hecke},  
$ \alpha_t \circ T_{p,G} = T_{p,H} \circ \alpha_t $.
\end{cor}

\begin{rem}
Note that these do not cover all degeneracy maps. Indeed, there might be $t \in M_2(\mathbb{Z})$ such that $N_G \cdot \det(t) \nmid N_H$, but 
$t^{-1} \Gamma_H t \subseteq \Gamma_G$ and such that $\alpha_t$ commute with all the Hecke operators (though we must have $\det(t) \mid N_H$).
As an example, consider $\Gamma_G = \Gamma_0(N)$, $\Gamma_H = \Gamma(N)$ and 
$t = \left( \begin{array}{cc} N & 0 \\ 0 & 1 \end{array} \right)$.
However, as the proof shows, for $\alpha_t$ to commute with $T_p$. it is enough to check that $dt^{-1} \alpha_p t \in d \cdot M_2(\mathbb{Z})$ for all $p$, which is equivalent to verifying that $\lambda_d(dt^{-1}) \lambda_d(\alpha_p) \lambda_d(t) = 0$. As this last expression only depends on $\lambda_{N_H}(\alpha_p)$, hence only on $p \bmod N_H$, we just have to check commutation with $\phi(N_H)$ elements.
Therefore, one may include these degeneracy maps as well.
\end{rem}

This motivates the following definition.

\begin{defn}
Let $N$ be a positive integer. Let $H \subseteq GL_2(\mathbb{Z} / N \mathbb{Z})$ be a subgroup.
For each $G \subseteq GL_2(\mathbb{Z} / N \mathbb{Z})$ such that $H \subseteq G$, we denote by $N_G$ the level of $G$. 
Let $T_G$ be a set of representatives for the orbits of $\Gamma_H$ on the following set
\begin{equation} \label{eqn: T_G}
T_G := \Gamma_H \backslash \left \{ t \in M_2(\mathbb{Z}) \mid \det(t) \mid N / N_G \text{ and } t^{-1} \Gamma_H t \subseteq \Gamma_G  \right \}.
\end{equation}
Let $i_G$ be the map
\begin{align*}
i_G : \left( S_k (\Gamma_G) \right)^{T_G} & \rightarrow S_k(\Gamma_H) \\
(f_t)_{t \in T_G} & \mapsto \sum_{t \in T_G} \alpha_t(f_t)
\end{align*} 
Let $\operatorname{Over}(H)$ be the set of minimal overgroups of $H$ in $GL_2(\mathbb{Z} / N \mathbb{Z})$. 
The subspace of \emph{oldforms at level $H$} is
$$
S_k(\Gamma_H)^{\operatorname{old}} = 
\sum_{G \in \operatorname{Over}(H)} i_G \left( \left( S_k (\Gamma_G) \right)^{T_G} \right)
$$ 
and the subspace of \emph{newforms at level $H$} is the orthogonal complement with respect to the Petersson inner product,
$$
S_k(\Gamma_H)^{\operatorname{new}} = \left( S_k(\Gamma_H)^{\operatorname{old}} \right)^{\perp}
$$
\end{defn}

We have to show that the sum in this definition is well defined.
\begin{lem}
Let $N$ be a positive integer. Let $H \subseteq G \subseteq GL_2(\mathbb{Z} / N \mathbb{Z})$ be subgroups.
Denote by $N_G$ the level of $G$. 
Let $T_G$ be the set given in \eqref{eqn: T_G}.
Then $T_G$ is finite.
\end{lem}

\begin{proof}
We first note that by the condition $\det(t) \mid N / N_G$, the determinant of $t$ assumes finitely many values, hence it is enough to show that the subset $T_{G,d}$ of such $t$'s with $\det(t) = d$ is finite.

Next, we recall that if $\Delta_d$ is the set of matrices in $M_2(\mathbb{Z})$ of determinant $d$, then $SL_2(\mathbb{Z}) \backslash \Delta_d$ is finite. Since $\Gamma_H$ is of finite index in $SL_2(\mathbb{Z})$, the result follows.
\end{proof}

The Hecke operators respect the decomposition of $S_k(\Gamma_G)$ into old and new.

\begin{prop}
The subspaces $S_k(\Gamma_G)^{\operatorname{old}}$ and $S_k(\Gamma_G)^{\operatorname{new}}$ are stable under the Hecke operators $T_n$ for $n \in \det(G)$ and the diamond operators $\langle \sigma_p \rangle$, for $p \in \det(G)$ prime. 
\end{prop}

\begin{proof}
By Corollary \ref{cor: Hecke commutes with degeneracy}, the Hecke operators $T_p$ for $p \in \det(G)$ commute with the degeneracy maps $\alpha_t$. By multiplicativity, it extends to $T_n$ such that $n \in \det(G)$. 

For the diamond operators, let $G_{i} \subseteq GL_2(\mathbb{Z} / N_{i} \mathbb{Z})$ for $i=1,2$ be such that $\lambda_{N_2}(G_1) \subseteq G_2$, and let $t \in M_2(\mathbb{Z})$ be such that $\det(t) \mid N_1 / N_2$ and
$t^{-1} \Gamma_1 t \subseteq \Gamma_2$, where $\Gamma_{i}$ is the congruence subgroup induced by $G_{i}$. Let $d = \det(t)$. Then $dt^{-1} \in M_2(\mathbb{Z})$ and so if $\alpha$ is such that $\lambda_{N_1}(\alpha) \in G_1$ and $\det(\alpha) = p$, we have
$$
\lambda_{N_2 \cdot d} (dt^{-1} \sigma_p t) = \lambda_{N_2 \cdot d}(dt^{-1}) \lambda_{N_2 \cdot d} (p \cdot \lambda_{N_2 \cdot d}(\alpha)^{-2}) \cdot \lambda_{N_2 \cdot d} (t) \in p \cdot \lambda_{N_2 \cdot d}(dt^{-1}) \lambda_{N_2 \cdot d}(G_1)  \lambda_{N_2 \cdot d} (t) 
$$
However, since $t^{-1} \Gamma_1 t \subseteq \Gamma_2$, we also have 
$$
\lambda_{N_2 \cdot d}(dt^{-1}) \lambda_{N_2 \cdot d}(G_1) \lambda_{N_2 \cdot d}(t) \subseteq d \cdot \lambda_{N_2 \cdot d}(\Gamma_2)
$$
hence $t^{-1} \sigma_p t \in M_2(\mathbb{Z})$ and moreover $\lambda_{N_2} (t^{-1} \sigma_p t) \in p \cdot \lambda_{N_2}(\Gamma_2) = p \cdot G_2 \subseteq \mathfrak{N}(G_2)$.
It follows that $t^{-1} \sigma_p t \in \Gamma_{\mathfrak{N}(G_2)}$, and so for any $f \in S_k(\Gamma_2)$ we have
$$
\left( \langle \sigma_p \rangle \circ \alpha_t \right)(f) = f \vert_{t^{-1} \sigma_p} = f \vert_{t^{-1} \sigma_p t t^{-1}} = 
\left( \alpha_t \circ  \langle  t^{-1} \sigma_p t \rangle \right)(f).
$$
Finally, we note that $\lambda_{N_2}(t^{-1} \sigma_p t) = p\cdot \lambda_{N_2}(t^{-1} \alpha t)^{-2}$, and that by the proof of Corollary \ref{prop: deg commutes with Hecke}, $t^{-1} \alpha t \in M_2(\mathbb{Z})$ is an element such that $\lambda_{N_2}(t^{-1} \alpha t) \in G_2$. It follows that $\langle t^{-1} \sigma_p t \rangle$ is the diamond operator at $p$ for $G_2$. This concludes the proof.
\end{proof}

\section{\texorpdfstring{Zeta functions and $q$-expansions}{}} \label{sec: zeta functions}

Using the results of Sections \ref{sec: Hecke operators} and \ref{sec: degeneracy maps}, we may now compute the zeta functions associated
to eigenforms in $S_{k}(\Gamma_{G})$.

All the ideas and methods used in this section are the same as in
the case of level $\Gamma_1(N)$ or $\Gamma_0(N)$, and are known (see \cite{stein2000explicit}, \cite{stein2007modular}) and used to compute the $q$-expansions. 

The process for doing so is divided to three steps:
\begin{enumerate}
\item Compute a subspace of $\mathbb{M}_{k}(\Gamma)^{\vee}$ that is isomorphic
to $\mathbb{S}_{k}(\Gamma)$ as a Hecke module, denoted by $\mathbb{S}_{k}(\Gamma)^{\vee}$.
\item Decompose the space $\mathbb{S}_{k}(\Gamma)^{\vee}$ to irreducible
Hecke modules.
\item For each irreducible Hecke module in the decomposition, find the system
of Hecke eigenvalues, hence the zeta function. 
\end{enumerate}

\subsection{Constructing Dual Vector Spaces}

Since our interest lies in $S_{k}(\Gamma)$, we would like to have
a vector space which is isomorphic to $S_{k}(\Gamma)\oplus\overline{S}_{k}(\Gamma)$
as a Hecke module. We could do it by using the pairing \eqref{eq:pairing of modular symbols and modular forms},
but that would involve computation of integrals and approximation if done directly.

Instead, we can use the fact that $S_{k}(\Gamma)$ and $\mathbb{S}_{k}(\Gamma)$
are isomorphic as Hecke modules, thus leading to the following algorithm.

\begin{algo}
$\operatorname{DualVectorSpace}(M)$. Compute the dual vector space.
\end{algo}

\textbf{Input :} $M\subseteq M_{k}(\Gamma)$ a subspace.

\textbf{Output : }$M^{\vee}\subseteq M_{k}(\Gamma)^{\vee}$ which
is isomorphic to $M$ as a Hecke module.
\begin{enumerate}
\item $V:=M_{k}(\Gamma)^{\vee}$, $p:=2$
\item While ($\dim V>\dim M)$ do
\begin{enumerate}
\item $T_{p}:=\operatorname{HeckeOperator}(V,p)$.
\item $T_{p,M}:=\operatorname{HeckeOperator}(M_{k}(\Gamma),p)\vert_{M}$.
\item $c_{p}:=\operatorname{CharPoly}(T_{p,M})$.
\item $V:=\ker(c_{p}(T_{p}))$. 
\item $p:=\operatorname{NextPrime}(p)$. 
\end{enumerate}
\item Return $V$.
\end{enumerate}

Here, we assume the existence of a function $\operatorname{CharPoly}$, that given an operator on a vector space, computes its characteristic polynomial. 

\begin{rem}
We have not defined the Hecke operators $T_{n}$ for $(n,N)>1$, and
these cannot be ignored (see \cite[Example 12.2.11]{voight2020cmf}).
It is possible to circumvent this problem, as is done in \cite[Section 8.10]{voight2020cmf}.
Practically when $p \mid N$, we use instead of $T_p$  the double coset Hecke operators $T_{\alpha}$ with $\det(\alpha) = p$. There are finitely many such operators, and the operator $T_p$ is a linear combination of them. Thus, for the purpose of describing the dual vector space, and later on, decomposition of the space to irreducible Hecke modules, these suffice.
\end{rem}

\subsection{Decomposition of $\mathbb{S}_{k}(\Gamma)$}

The following theorem is a generalization of \cite[Theorem 9.23]{stein2007modular}. 
It is an application of Sturm's Theorem, which will
enable us to decompose the Hecke module $\mathbb{S}_{k}(\Gamma)$ in a finite (effectively bounded) amount of steps.
\begin{thm}
\label{thm:Sturm bound}Suppose $\Gamma$ is a congruence subgroup
of level $N$, and let 
\begin{equation} \label{eqn: Sturm bound}
r=\operatorname{Sturm}(k,\Gamma):=\left\lfloor \frac{km}{12}-\frac{m-1}{N}\right\rfloor 
\end{equation}
where $m=[SL_{2}(\mathbb{Z}):\Gamma]$. Then the Hecke algebra 
\[
\mathbb{T}=\mathbb{Z}[\ldots,T_{n},\ldots]\subseteq\operatorname{End}(S_{k}(\Gamma))
\]
is generated as a $\mathbb{Z}$-module by the Hecke operators $T_{n}$
for $n\le r$.
\end{thm}

\begin{proof}
Same as in \cite[Theorem 9.23]{stein2007modular}. 
\end{proof}

This allows us to decompose the space $\mathbb{S}_{k}(\Gamma)^{\vee}=S_{k}(\Gamma)\oplus\overline{S}_{k}(\Gamma)$
to irreducible Hecke modules, as follows.

\begin{algo} \label{algo: Decomposition}
$\operatorname{Decomposition}(M,p)$. Decomposition to irreducible Hecke modules.
\end{algo}

\textbf{Input :} 
\begin{itemize}
\item $M\subseteq\mathbb{M}_{k}(\Gamma)$ - a vector subspace stable under
the Hecke action.
\item $p$ - a prime.
\end{itemize}

\textbf{Output :} $D$, a list of subspaces of $M$, such that $M=\bigoplus_{V_{d}\in D}V_{d}$
and $V_{d}^{\vee}$ is an irreducible Hecke module.
\begin{enumerate}
\item If $p\mid N$, replace by a larger prime such that $p\nmid N$. 
\item $T_{p}:=\operatorname{HeckeOperator}(M^{\vee},p)$, $D:= \emptyset$.
\item $f:=\operatorname{CharPoly}(T_{p})$. Write $f=\prod_{i=1}^{l}f_{i}^{a_{i}}$. 
\item for $i\in\{1,2,\ldots,l\}$ do
\begin{enumerate}
\item $V:=\ker(f_{i}(T_{p})^{a_{i}})\subseteq M^{\vee}$, $W:=V^{\vee}\subseteq M$
. 
\item if $\operatorname{IsIrreducible}(W)$
\begin{enumerate}
\item $D := D \cup \{ W \}$
\end{enumerate}
else
\begin{enumerate}
\item if $W=M$ then $q:=\operatorname{NextPrime}(p)$ else $q:=2$.
\item $D:=D\cup \operatorname{Decompose}(W,q)$.
\end{enumerate}
\end{enumerate}
\item Return $D$.
\end{enumerate}
Here $\operatorname{Decompose}$ is a recursive call to the function, and $\operatorname{IsIrreducible}$
is creating a random linear combination of Hecke operators, and checks
that its characteristic polynomial (on the plus subspace) is irreducible. 
One could eliminate the randomness to provably verify the irreduciblity, as in 
\cite[Chapter 7]{stein2007modular}.

Calling Algorithm \ref{algo: Decomposition} with $M=\mathbb{S}_{k}(\Gamma)$,
$p=2$ we obtain the decomposition we wanted.

Note that the algorithm terminates if there are no eigenvectors with multiplicities, that is there are no oldforms arising as images of forms of lower levels. In that case, the number of steps in the algorithm can be bounded using \ref{thm:Sturm bound}. 

In order to deal with the images of oldforms, we may begin by identifying the new subspace, and continuing recursively.
The new subspace is computed by noting that it is dual to 
$\mathbb{S}_k(\Gamma_G)^{\operatorname{new}} := \bigcap_{G' \in \operatorname{Over}(G),t \in T_{G'}} \ker(\alpha_t^{\vee})$.

Also, in order to get the old subspace one need not compute the Petersson inner product, but instead note that it is generated by the images of the $\beta_t^{\vee}$.
$$
\mathbb{S}_k(\Gamma_G)^{\operatorname{old}} := \left \langle \operatorname{Im}(\beta_t^{\vee}) \right \rangle_{G' \in \operatorname{Over}(G),t \in T_{G'}}.
$$

\subsection{Computing the zeta functions}

Finally, for a subspace $V\subseteq\mathbb{S}_{k}(\Gamma_G)$ such that
$V^{\vee}$ is an irreducible Hecke module, we can compute the zeta function
of an eigenform (and hence of all eigenforms in $V^{\vee}$). 

\begin{algo} \label{algo: zeta function}
$\operatorname{ZetaFunction}(V,L)$. Compute the zeta function associated to an eigenform.
\end{algo}

\textbf{Input :} 
\begin{itemize}
\item $V\subseteq\mathbb{S}_{k}(\Gamma_G)$ such that $V^{\vee}$ is an irreducible
Hecke module.
\item $L$ - a positive integer. 
\end{itemize}

\textbf{Output : }$Z(f,s)=\sum_{n=1}^{\infty}a_{n}n^{-s} $ such that $T_{n}f=a_{n}f$ for all $n$ (set $a_{1}=1$), given to precision $L^{-s}$.
\begin{enumerate}
\item Find (randomly) a linear combination of Hecke operators which is irreducible
on $(V^{\vee})^{+}$, denote it by $T$.
\item Let $v$ be an eigenvector of $T$ over the field $\mathbb{Q}[x]/f_{T}(x)$,
where $f_{T}=\operatorname{CharPoly}(T)$. 
\item Write $v=\sum_{j=i}^{d}c_{j}e_{j}^{\vee}$, where $e_{j}$ are the
basis vectors of $\mathbb{M}_{k}(\Gamma)$ given by Manin symbols,
and $c_{i}\ne0$. 
\item For primes $p<L$, set $a_{p}:=\frac{1}{c_{i}}\left\langle T_{p}(e_{i}),v\right\rangle $.
Complete the other $a_{n}$'s using multiplicativity. 
\item Return $\sum_{n=1}^{L-1}a_{n} n^{-s}+O(L^{-s})$. 
\end{enumerate}

We start by showing that the algorithm returns the correct output, and measure its complexity.

\begin{lem} \label{lem: qexpansion}
Let $G \subseteq GL_2(\mathbb{Z} / N \mathbb{Z})$ be a group of real type of index $I_G$ with surjective determinant such that $T_p$ is effectively computable for all $p \mid N$.
There exists an algorithm that given a subspace $V$ of $\mathbb{S}_k(\Gamma_G)$ of dimension $d$, such that $V^{\vee}$ is irreducible as a module for the Hecke algebra, and an integer $L$, returns the zeta function associated to an eigenform up to precision $L^{-s}$ in 
$O(C \log N (L \log L + N) + N I_G^2 \cdot \operatorname{In} + d^3)$, where $C$ is the cost of a basic $\operatorname{CosetIndex}$ operation and $\operatorname{In}$ is the cost of the group membership test for $G$.
\end{lem}

\begin{proof}
We apply Algorithm \ref{algo: zeta function}. Since $G$ is of real type, by Corollary \ref{cor: Hecke commutes with star}, the Hecke operator commutes with the star involution, so that step (1) makes sense.
For any $p<L$ we have
\[
\sum_{j=i}^{d}a_{p}c_{j}e_{j}^{\vee}=a_{p}v=T_{p}^{\vee}v.
\]
In particular, 
\begin{equation*}
\left\langle T_{p}(e_{i}),v\right\rangle =\left\langle e_{i},T_{p}^{\vee}v\right\rangle 
=\left\langle e_{i},\sum_{j=i}^{d}a_{p}c_{j}e_{j}^{\vee}\right\rangle =a_{p}\cdot c_{i}.
\end{equation*}
For computation of the Hecke operators away from the level, we apply Corollary \ref{cor:Complexity Hecke Operator}, while for the Hecke operators at primes $p \mid N$, we apply Theorem \ref{thm: Hecke operator double coset naive}. 
Note that as the $T_p$ are effectively computable for all $p \mid N$, we may do so with no additional cost to our complexity.
Therefore, the complexity of computing the Hecke operators is given by
\begin{align}
\nonumber
C  \sum_{p < L, p \nmid N} p \log p &+ C  \sum_{p<L, p \mid N} p \log (N^4 p)  + 
\sum_{p < L, p \mid N} I_G^2 \cdot \operatorname{In} \\
\nonumber
&= C \sum_{p<L} p \log p + 4 C \log N \sum_{p<L, p \mid N} p  + N I_G^2 \cdot \operatorname{In} \\
&= O(C (L \log L + N) \log N + N I_G^2 \cdot \operatorname{In}).
\label{eq : Hecke complexity}
\end{align}
Finally, the $d^3$ contribution comes from the linear algebra operations.
\end{proof}

We can now pack everything together to produce the following corollary.

\begin{cor} \label{cor: q expansion basis}
There exists an algorithm that given a group of real type $G \subseteq GL_2(\mathbb{Z} / N \mathbb{Z})$ with surjective determinant such that for all $p \mid N$ the Hecke operator is effectively computable and a positive integer $L$, returns the zeta functions of the factors of $Jac(X_G)$ using 
$$
O(d ( C \log N (L \log L + N) + N I_G^2 \cdot \operatorname{In} + I_G \log I_G) + d^3)
$$
field operations, 
where $d := \dim S_2(\Gamma_G)$, $I_G := [SL_2(\mathbb{Z}):\Gamma_G]$, $C$ is the cost of a $\operatorname{CosetIndex}$ operation, and 
$\operatorname{In}$ is the cost of membership testing in $G$.
\end{cor}

\begin{proof}
We first compute a basis for $\mathbb{S}_2(\Gamma)$ using Theorem \ref{thm: cuspidal modular symbols}.
Then we can apply Algorithm \ref{algo: Decomposition} to decompose it to irreducible subspaces. Here, we have to compute the matrices of the Hecke operators up to the Sturm bound \eqref{eqn: Sturm bound}, which is linear in $I_G$. 

For computation of the Hecke operators away from the level, we apply Corollary~\ref{cor:Complexity Hecke Operator}, while for the Hecke operators at primes $p \mid N$, we apply Theorem~\ref{thm: Hecke operator double coset naive}, as in \eqref{eq : Hecke complexity}.
Since we now compute their action on a basis, we have to multiply by $d$, and compute them up to $I_G$ to obtain
$O(d (C (I_G \log I_G + N) \log N + N I_G^2 \cdot \operatorname{In}))$.

Finally, for each of the irreducible subspaces, we should apply Algorithm \ref{algo: zeta function}, which by \ref{lem: qexpansion}  costs
$O(C \log N (L \log L + N) + N I_G^2 \cdot \operatorname{In} + d^3)$.
\end{proof}

\section{Applications} \label{sec: applications}

In this section, we present a few of the applications that the above
result contributed to, and also some future applications.
All time measurements were taken on a MacBook with 2.3 GHz 8-Core Intel Core i9 processor, and 16 GB 2400 MHz DDR4 memory.

\subsection{Classification of $2$-adic images of Galois representations associated
to elliptic curves over $\mathbb{Q}$}

In \cite{rouse2015elliptic}, the authors compute the maximal tower of
$2$-power level modular curves containing a non-cuspidal non-CM rational point, and develop special techniques to
compute their equations. Using our implementation, one can obtain
the $q$-expansions of a basis for $S_{2}(\Gamma)$ for each of those groups. 

For example, in \cite[Example 6.1]{rouse2015elliptic} the authors present
a curious example associated with the group $H_{155}$, which is of
index $24$ and level $16$, whose image in $GL_{2}(\mathbb{Z}/16\mathbb{Z})$
is generated by 
\[
\left(\begin{array}{cc}
1 & 3\\
0 & 3
\end{array}\right),\left(\begin{array}{cc}
1 & 0\\
2 & 3
\end{array}\right),\left(\begin{array}{cc}
1 & 3\\
12 & 3
\end{array}\right),\left(\begin{array}{cc}
1 & 1\\
12 & 7
\end{array}\right)
\]

Then a quick calculation of 1.78 seconds with our implementation yields the following.

\begin{verbatim}
> tt := Cputime();
> gens := [[1,3,12,3],[1,1,12,7],[1,3,0,3],[1,0,2,3]];
> N := 16;
> H_N := sub<GL(2,Integers(N)) | gens>;
> H := PSL2Subgroup(H_N);
> M := ModularSymbols(H, 2, Rationals(), 0);
> S := CuspidalSubspace(M);
> D := Decomposition(S, HeckeBound(S));
> qEigenform(D[1],100);
q - 4*q^5 - 3*q^9 - 4*q^13 - 2*q^17 + 11*q^25 - 4*q^29 + 
12*q^37 - 10*q^41 + 12*q^45 - 7*q^49 - 4*q^53 + 12*q^61 + 
16*q^65 - 6*q^73 + 9*q^81 + 8*q^85 + 10*q^89 - 18*q^97 + O(q^100)
> Cputime(tt);
1.780
\end{verbatim}

Proceeding to compute the invariants of the elliptic curve associated
to this eigenform, we see that the modular curve $X_{H}$ is isogenous to the elliptic
curve labelled \href{https://www.lmfdb.org/EllipticCurve/Q/256b1/}{256b1}, defined by the equation $y^{2}=x^{3}-2x$ over $\mathbb{Q}$, for which 
$X_{H}(\mathbb{Q}) \cong \mathbb{Z} / 2 \mathbb{Z} \times \mathbb{Z}$ is generated by $(0,0)$ and $(-1,-1)$.

\begin{rem}
For this application, if one wants to compute the equation for the elliptic curve up to isomorphism, one has to compute the $j$-map. This can be done by computing the $q$-expansion of Eisenstein modular forms. A way to do so (in an even more complicated case) is explained in detail in \cite[Example 6.1]{rouse2015elliptic}.
Our methods simply replicate the efforts originally done by the authors. For convenience, we have also supplied in our code package the code to generate Eisenstein series, e.g. for the weight 2 Eisenstein series, we can do the following.
\begin{verbatim}
> M := ModularForms(H);
> qExpansion(EisensteinSeries(M)[1], 100);
q^16 + 32*q^32 + 244*q^48 + 1024*q^64 + 3126*q^80 + 7808*q^96 + O(q^100)
\end{verbatim}
\end{rem}

\subsection{Modular curves of prime-power level with infinitely many rational
points}

In \cite[Corollary 1.6]{sutherland2017modular}, the authors introduce
a finite set of subgroups $G$ of $GL_{2}(\mathbb{Z}_{l})$ that arise
as the image of a Galois representation for infinitely many elliptic
curves over $\mathbb{Q}$ with distinct $j$-invariants.

In section 6, they consider the 250 cases where $G$ has genus $1$
and show that the Jacobian of $X_{G}$, $J_{G}$, for each such group
is isogenous to one of a certain finite set of elliptic curves. Then
they compute point counts of the reduction modulo several primes to pinpoint
the correct isogeny class over $\mathbb{Q}$, and show that one needs
only consider a certain set of 28 groups. 

These are included in the examples computed in \cite{rouse2015elliptic},
and again, using our implementation, it is possible to construct for
each of these subgroups $G$ a basis of $q$-eigenforms for the modular
curve $X_{G}$.

The following example took 2.120 seconds.

\begin{verbatim}
> tt := Cputime();
> gens := [[2,1,3,2],[0,3,5,8],[1,0,0,5],[1,8,0,3]];
> N := 16;
> H_N := sub<GL(2,Integers(N)) | gens>;
> H := PSL2Subgroup(H_N);
> M := ModularSymbols(H, 2, Rationals(), 0);
> S := CuspidalSubspace(M);
> D := Decomposition(S, HeckeBound(S));
> qEigenform(D[1],100);
q - 2*q^3 + q^9 - 6*q^11 - 6*q^17 - 2*q^19 - 5*q^25 + 4*q^27 
    + 12*q^33 + 6*q^41 + 10*q^43 - 7*q^49 + 12*q^51 + 4*q^57 
    - 6*q^59 + 14*q^67 - 2*q^73 + 10*q^75 - 11*q^81 - 18*q^83
     - 18*q^89 + 10*q^97 - 6*q^99 + O(q^100)
> Cputime(tt);
2.120
\end{verbatim}

This is the newform \href{https://www.lmfdb.org/ModularForm/GL2/Q/holomorphic/256/2/a/a/}{256.2.2.a}. Computing the elliptic invariants, we see that this is isogenous to the curve \href{https://www.lmfdb.org/EllipticCurve/Q/256a2/}{256a2} defined by $y^2 = x^3 + x^2 - 13x - 21$.

\subsection{Efficient computation of $q$-eigenforms for $X_{ns}(N)$ and $X_{ns}^{+}(N)$}

Using the Merel pair introduced in Example \ref{exa:Merel pair nonsplit Cartan},
and a similar one for $\Gamma_{ns}^{+}(p)$, its normalizer, one obtains
a highly efficient implementation of the Hecke operators $T_{n}$
for all $n$. This allows one to compute the $q$-eigenform and verify
the results of \cite{dose2014automorphism} for $X_{ns}(11)$, \cite{baran2014exceptional} for
$X_{ns}^{+}(13)$, and \cite{mercuri2020modular} for $X_{ns}^{+}(17),X_{ns}^{+}(19)$
and $X_{ns}^{+}(23)$, as follows.

\begin{verbatim}
> tt := Cputime();
> G := GammaNSplus(13);
> M := ModularSymbols(G, 2, Rationals(), 0);
> S := CuspidalSubspace(M);
> D := Decomposition(S, HeckeBound(S));
> qEigenform(D[1], 20);
q + a*q^2 + (-a^2 - 2*a)*q^3 + (a^2 - 2)*q^4 + 
(a^2 + 2*a - 2)*q^5 + (-a - 1)*q^6 + (a^2 - 3)*q^7 + 
(-2*a^2 - 3*a + 1)*q^8 + (a^2 + 3*a - 1)*q^9 + (-a      
+ 1)*q^10 + (-a^2 - 2*a - 2)*q^11 + (a^2 + 3*a)*q^12 + 
(-2*a^2 - 2*a + 1)*q^14 + (a^2 + a - 2)*q^15 + 
(-a^2 - a + 2)*q^16 + (-a^2 + a + 2)*q^17 + 
(a^2 + 1)*q^18 + (-2*a^2 - a + 2)*q^19 + O(q^20)
> BaseRing(Parent($1));
Number Field with defining polynomial x^3 + 2*x^2 - x - 1 
over the Rational  Field
> Cputime(tt);
0.160
\end{verbatim}

This coincides with the result in section 4 in \cite{baran2014exceptional}, and took merely 0.16 seconds.

Similarly, in only 0.33 seconds we get the result for $X_{ns}^{+}(17)$.

\begin{verbatim}
> tt := Cputime();
> G := GammaNSplus(17);
> M := ModularSymbols(G, 2, Rationals(), 0);
> S := CuspidalSubspace(M);
> D := Decomposition(S, HeckeBound(S));
> [*qEigenform(d,20) : d in D*];
[*
q - q^2 - q^4 + 2*q^5 - 4*q^7 + 3*q^8 - 3*q^9 - 2*q^10 - 
2*q^13 + 4*q^14 - q^16 + 3*q^18 - 4*q^19 + O(q^20),
q + a*q^2 + (-a - 1)*q^3 + (-a + 1)*q^4 + a*q^5 - 3*q^6 + 
(-a - 2)*q^7 - 3*q^8 + (a + 1)*q^9 + (-a + 3)*q^10 - 3*q^11 + 
(-a + 2)*q^12 + (a - 1)*q^13 + (-a - 3)*q^14 - 3*q^15 + 
(-a - 2)*q^16 + 3*q^18 + (-3*a - 1)*q^19 + O(q^20),
q + a*q^2 + (-a^2 + 1)*q^3 + (a^2 - 2)*q^4 + 
(-a - 2)*q^5 + (-2*a + 1)*q^6 + (a^2 - 2)*q^7 + (-a - 1)*q^8 + 
(a^2 - a - 2)*q^9 + (-a^2 - 2*a)*q^10 + (2*a^2 + 2*a - 6)*q^11 + 
(a - 2)*q^12 + (-2*a^2 - 3*a + 6)*q^13 + (a - 1)*q^14 + 
(2*a^2 + 2*a - 3)*q^15 + (-3*a^2 - a + 4)*q^16 + 
(-a^2 + a -1)*q^18 + a*q^19 + O(q^20)
*]
> BaseRing(Parent(qEigenform(D[2])));
Number Field with defining polynomial x^2 + x - 3 
over the Rational Field
> BaseRing(Parent(qEigenform(D[3])));
Number Field with defining polynomial x^3 - 3*x + 1 
over the Rational Field
> Cputime(tt);
0.330
\end{verbatim}
which coincides with the results of \cite{mercuri2020modular}. 

In \cite{baran2010normalizers}, section 6, there is a decsription
of coset representatives for $\Gamma_{ns}(N)$ and $\Gamma_{ns}^{+}(N)$
for general $N$, which again makes explicit the implementation via
modular symbols. 

\subsection{Decomposition of the Jacobian of $X_{ns}^{+}(p)$}

Applying Algorithm \ref{algo: Decomposition} to the
space $\mathbb{S}_{2}(\Gamma_{ns}^{+}(p))$, where $\Gamma_{ns}^{+}(p)$
is yields a decomposition of the Jacobian of $X_{ns}^{+}(p)$. This
has been computed using our code for all $p\le97$. (see \cite{assaf2019nscartan})

\subsection{Computation of $q$-Eigenforms for $X_{G}(p)$ when $G$ is exceptional}

Similarly, the code can be used to compute $q$-eigenforms for the
modular curves $X_{G}(N)$, when $G$ is an exceptional subgroup of
$GL_{2}(\mathbb{F}_{p})$ with surjective determinant, e.g. running it for $X_{S_{4}}(13)$ we
recover \cite[Theorem 1.8]{banwait2014tetrahedral}. We obtain step 3 of section 4 in the original paper in just a few seconds. 
(We use the embedding of $S_{4}\hookrightarrow GL_{2}(\mathbb{F}_{p})$
using the quaternions as describes in \cite[Section 11.5]{voight2020arithmetic}). 

\begin{verbatim}
> tt := Cputime();
> p := 13;
> B<i,j,k> := QuaternionAlgebra(Rationals(), -1, -1);
> O := QuaternionOrder([1,i,j,k]);
> _, mp := pMatrixRing(O,p);
> S4tp := sub<GL(2,p) | [mp(1+s) : s in [i,j,k]] 
>  cat [mp(1-s) : s in [i,j,k]] cat [-1]>;
> H_S4 := sub<GL(2,Integers(p)) | Generators(S4tp)>;
> G_S4 := PSL2Subgroup(H_S4);
> M := ModularSymbols(G_S4, 2, Rationals(), 0);
> S := CuspidalSubspace(M);
> D := Decomposition(S, HeckeBound(S));
> [*qEigenform(d,20) : d in D*];
[*
q + a*q^2 + (-a^2 - 2*a)*q^3 + (a^2 - 2)*q^4 + 
(a^2 + 2*a - 2)*q^5 + (-a - 1)*q^6 + (a^2 - 3)*q^7 + 
(-2*a^2 - 3*a + 1)*q^8 + (a^2 + 3*a - 1)*q^9 +
(-a + 1)*q^10 + (-a^2 - 2*a - 2)*q^11 + (a^2 + 3*a)*q^12 + 
(-2*a^2 - 2*a + 1)*q^14 + (a^2 + a - 2)*q^15 + 
(-a^2 - a + 2)*q^16 + (-a^2 + a + 2)*q^17 + 
(a^2 + 1)*q^18 + (-2*a^2 - a + 2)*q^19 + O(q^20)
*]
> BaseRing(Parent(qEigenform(D[1])));
Number Field with defining polynomial x^3 + 2*x^2 - x - 1 
over the Rational  Field
\end{verbatim}

\subsection{Smooth plane models for modular curves}

Modular curves of genus $0,1$ obviously admit a smooth plane model
always. In general, if a curve admit such a model, of degree $d$,
then its genus would be $g=\frac{(d-1)(d-2)}{2}$, so the next numbers
to check are $g=3,6$. When $g=3$, by \cite{gonzalez2010non}, the
generic case is a smooth plane quartic, so the answer is "almost
always''. 

However, for $g>3$, a generic curve does not admit a smooth plane
model. 

The smallest case for which we do not know the answer is $g=6$. Thus,
computing a basis of eigenforms of $\mathbb{S}_{2}(\Gamma)$ for each
of the congruence subgroups $\Gamma$ of genus $6$ (see \cite{cummins2017congruence}),
will help us to get a canonical model and from it (maybe) decide whether
there exists a plane model.

\bibliographystyle{plain}
\bibliography{Eran_papers_db.bib}

\end{document}